\documentclass[11pt,reqno]{amsart}
\usepackage{graphicx,amsmath,amsfonts,latexsym,amssymb,amsthm,color}
\usepackage{latexsym}
\usepackage{verbatim}
\usepackage{enumerate}
\usepackage{tikz-cd}

\usepackage[a4paper]{geometry}
\definecolor{blau}{rgb}{0.05,0.2,0.7}
\definecolor{auchblau}{rgb}{0.03,0.3,0.7}
\usepackage[pdftex,linkbordercolor={0.8 0.5 0.2}]{hyperref} 
\hypersetup{
pdfauthor = {Alexander Strohmaier, Alden Waters},
pdftitle = {Analytic properties of Resolvents},
pdfsubject ={Hahn-holomorphic functions}{Hahn series}{Fredholm Theorem}{Scattering Theory}{Spectral Theory},
colorlinks = true,
linkcolor = blau,
citecolor = auchblau
}

\renewcommand{\Re}{\operatorname{Re}}
\renewcommand{\Im}{\operatorname{Im}}

\DeclareMathOperator{\supp}{supp}
\newcommand{\der}{\mathrm{d}}
\newcommand{\rmi}{\mathrm{i}}
\newcommand{\ee}{\mathrm{e} }
\newcommand{\sphere}{{\mathbb{S}^{d-1}}}
\newcommand{\calO}{\mathcal{O}}
\newcommand{\tr}{\mathrm{tr}}
\newcommand{\Tr}{\mathrm{Tr}}

\newcommand{\Complex}{\mathbb{C}}
\newcommand{\Reell}{\mathbb{R}}
\newcommand{\R}{\mathbb{R}}
\newcommand{\C}{\mathbb{C}}

\newcommand{\delbar}{\overline \partial}

\newcommand{\comp}{0}
\newcommand{\compp}{\mathrm{comp}}
\newcommand{\CnI}{C^\infty_{0}}
\newcommand{\loc}{\mathrm{loc}}

\newtheorem{theorem}{Theorem}[section]
\newtheorem{definition}[theorem]{Definition}
\newtheorem{example}[theorem]{Example}
\newtheorem{lemma}[theorem]{Lemma}

\newtheorem{corollary}[theorem]{Corollary}
\newtheorem{proposition}[theorem]{Proposition}

\newtheorem{rem}[theorem]{Remark}

\title[Geometric and Obstacle Scattering]{Geometric and Obstacle Scattering at Low Energy}

\author[A. Strohmaier]{Alexander Strohmaier}
\address{School of Mathematics,  University of Leeds,  Leeds , Yorkshire, LS2 9JT,
UK} \email{a.strohmaier@leeds.ac.uk}
\thanks{Supported by Leverhulme grant RPG-2017-329}

\author[A. Waters]{Alden Waters}
\address{ University of Groningen, Bernoulli Institute,
Nijenborgh 9,
9747 AG Groningen,
The Netherlands} \email{a.m.s.waters@rug.nl}

\begin{document}

\begin{abstract}
 We consider scattering theory of the Laplace Beltrami operator on differential forms on a Riemannian manifold that is Euclidean at infinity. The manifold may have several boundary components caused by obstacles at which relative boundary conditions are imposed. Scattering takes place because of the presence of these obstacles and possible non-trivial topology and geometry. Unlike in the case of functions eigenvalues generally exist at the bottom of the continuous spectrum and the corresponding eigenforms represent cohomology classes. We show that these eigenforms appear in the expansion of the resolvent, the scattering matrix, and the spectral measure in terms of the spectral parameter $\lambda$ near zero, and we determine the first terms in this expansion explicitly. In dimension two an additional cohomology class appears as a resonant state in the presence of an obstacle.
In even dimensions the expansion is in terms of
$\lambda$ and $\log \lambda$. The theory of Hahn holomorphic functions is used to describe these expansions effectively. We also give a Birman-Krein formula in this context.
The case of one forms with relative boundary conditions has direct applications in physics as it describes the scattering of electromagnetic waves.
  \end{abstract}


\maketitle

{\small
\tableofcontents}

\section{Introduction and Setting}

The analysis of the spectrum and the spectral decomposition of geometric operators on manifolds is important in both physics and mathematics. A full spectral decomposition allows one to solve linear equations such as the wave equation, Schr\"odinger's equation, and the heat equation. The long term behaviour of the latter is determined by the bottom of the spectrum. 
For the Laplace operator on $p$-forms on a closed Riemannian manifold the Hodge isomorphism (\cite{hodge1989theory}) identifies the harmonic forms with the de-Rham cohomology groups. Such connections between the bottom of the spectrum of geometric operators on closed manifolds give rise to an extremely rich interplay between topology and the analysis of partial differential equations. One of the important examples is the Atiyah-Singer index theorem (\cite{michael1968atiyah}) that relates the index of an elliptic operator to the $K$-theory class determined by its principal symbol. On non-compact manifolds the situation is slightly more complicated due to the presence of an essential spectrum. One approach to analyse the topology of the space using Hodge theory is to study $L^2$-cohomology which in many examples can be identified with the space of $L^2$-harmonic forms, i.e. the zero eigenspace of the Laplace operator. Well-studied examples are manifolds with cylindrical ends (\cite{MR0397797,melrose1993atiyah}), cusp-ends and variants of these (\cite{melrose1993atiyah,Vaillant2001IndexAS,MR2057017}), as well as manifolds with conical singularities (\cite{cheeger1980hodge}) and conical ends
(\cite{melrose1994spectral}). To illustrate this we briefly explain the situation for manifolds with cylindrical ends.
Here there may be a finite dimensional space of $L^2$-eigenfunctions at zero, but in general zero is also contained in the absolutely continuous spectrum.
The $L^2$-harmonic forms on the manifold describe the image of cohomology with compact support in the cohomology of the space. A complement of this image can be described by the values of the generalised eigenfunctions at zero. This relation between cohomology and the low lying values of the continuous part of the spectral decomposition was somewhat anticipated by the work of Atiyah Patodi and Singer (\cite{MR0397797}) on the index theorem for manifolds with boundary. The relationship between these concepts was further clarified by Melrose (\cite{melrose1993atiyah}) and M\"uller (\cite{muller19882}). A detailed analysis of the bottom of the continuous spectrum for manifolds with cylindrical ends can be found in (\cite{muller2010scattering}).

Another class of important examples are manifolds with conical ends and the subclass of manifolds with one Euclidean end. Similarly to the case of cylindrical ends
the $L^2$-cohomology groups can be identified with the zero eigenspace of the Laplace operator on forms. These groups can be computed here within a very general framework and related to de-Rham cohomology groups c.f. also (\cite{melrose1994spectral,carron2003l2,MR2057017}).

The goal of this paper is to clarify the role of the continuous spectrum in this context. 
Namely, we analyse the spectral decomposition of the Laplace-Beltrami operator $\Delta$ acting on $p$-forms on oriented manifolds that are asymptotically Euclidean at infinity and with possible compact boundary on which relative boundary conditions are imposed. The boundary components are thought of as obstacles and scattering takes place because of these obstacles, and possibly because of a non-trivial geometry and topology.

\begin{figure}[h]
 \begin{center}
 \includegraphics[trim=0 8cm 0 8cm,clip, width=7cm]{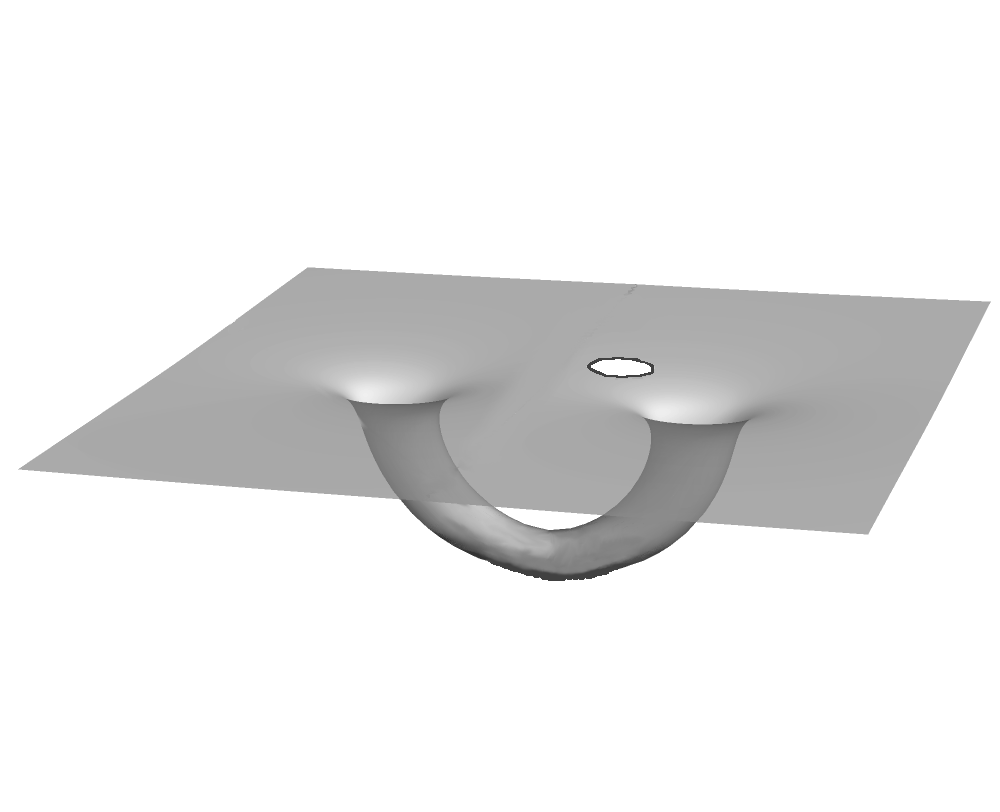} 
 \end{center}
 \caption{Surface with one boundary component and non-trivial topology that is Euclidean outside a compact set.}
\end{figure}

The spectrum of $\Delta$ lies on the positive real line. It consists of an absolutely continuous part, described by generalised eigenfunctions, and possibly a zero eigenvalue of finite multiplicity given by the $L^2$-Betti numbers.  In fact the space of $L^2$-harmonic forms has a finer filtration that we describe in this paper
that encodes how fast the corresponding eigenfunctions decay at infinity. In dimensions $d \geq 3$ the structure of the singularities of the resolvent, the spectral measure and the scattering matrix, can be completely characterised in terms of the $L^2$-eigenfunctions and their decay properties. Some information about the cohomology of the manifold is therefore retained in the continuous spectrum. In the case $d=2$ when the boundary is non-empty there is a non-trivial cohomology class in relative cohomology that is not represented by an $L^2$-harmonic form but rather by a zero resonant state. We completely clarify the singularity structure of the resolvent near zero and its meromorphic continuation. We also give the leading term in the expansion of the scattering amplitude. In even dimensions, the resolvent, the scattering matrix, and the generalised eigenfunctions, are not holomorphic at zero but have convergent generalised expansions into power series containing both powers of $\lambda$ and $-\log \lambda$. The theory of these functions was developed in \cite{muller2014theory} and this paper makes extensive use of this theory. Our approach is similar to that of Vainberg \cite{vainberg1975short} in its treatment of logarithmic terms. 

The low energy behaviour of Schr\"odinger operators has been studied by Kato and Jensen (\cite{jensen1979spectral}) who also computed expansion coefficients for the resolvent in various dimensions (see for example \cite{jensen2001unified}). Murata (\cite{murata1982asymptotic}), using also the method of Vainberg (\cite{vainberg1975short}) analysed the low energy behaviour of constant coefficient operators with potentials. The two dimensional case is quite complicated and was analysed for potential scattering in great detail in \cite{bolle1988threshold}. 

Resolvent expansions in the more general setting of conical ends were given by Wang (\cite{wang2006asymptotic, wang2004threshold}). Perhaps closest to our results are expansions obtained in the works by Guillarmou and Hassell (\cite{guillarmou2008resolvent,guillarmou2009resolvent}) where various resolvent expansions for the Laplace operator on functions are proved in the setting of conical manifolds. In \cite{guillarmou2009resolvent} the authors compute one of the expansion coefficients in the resolvent in the case of functions and reproduce the formula of Jensen and Kato in this more general setting. Expansions for differential forms were used 
in \cite{guillarmou2014low}) to show boundedness of the Riesz transform on $L^p$-spaces. Other recent work discussing the low energy behaviour of the resolvent is by Bony and H\"afner \cite{bony2009low} for second order operators in divergence form, and by Rodnianski and Tao \cite{rodnianski2015effective}  who  also consider potentials and general asymptotically  conic manifolds.
We would also like to mention the very recent work of Vasy on the low energy resolvent on asymptotically conic spaces (\cite{vasy2018resolvent}, \cite{vasy2019limiting}, \cite{vasy2019resolvent}) and the fact that the long time behaviour of solutions of the wave equation on differential forms also plays a role in stability questions in general relativity (\cite{hintz2018asymptotics, hafner2019linear}).

A relation between the topology of manifolds with Euclidean ends and the continuous spectrum has been noticed by Carron who gives expansions of the determinant of the scattering matrix in terms of $L^2$-Betti numbers and resonant states (\cite{carron2004saut}), which shows in particular that the jump of the spectral shift function at zero is of topological significance. The significance of the spectral shift function in this context was also seen by Borisov Mueller and Schrader in their proof of the Chern-Gauss-Bonnet formula for asymptotically Euclidean manifolds (\cite{borisov1988relative}).
The detailed structure of the resolvent on non-compact manifold is also important in quantum field theory in the quantisation of the electromagnetic field as poles of the resolvent manifest themselves as ``infrared problems''. As an example the Gupta-Bleuler quantisation of the electromagnetic field as constructed rigorously in \cite{finster2015gupta,finster2018correction}  requires the absence of a zero resonance state for the Laplace operator on one forms.

\subsection{Precise setup and notations} \label{sec:1.1}

Let $(X,g)$ be an oriented complete connected Riemannian manifold of dimension $d \geq 2$ which is Euclidean at infinity, i.e.
there exists compact subsets $K \subset X$ and $\tilde K \subset \Reell^d$ such that $X \backslash K$ is isometric to $\Reell^d \backslash \tilde K$.
Let $\mathcal{O}$ be an open subset in $X$ with compact closure and smooth boundary. The (finitely many) connected components will be denoted
by $\mathcal{O}_i$ with some index $i$. We will think of these as obstacles placed in $X$. Removing these obstacles from $X$ results in a Riemannian
manifold $M = X \backslash \mathcal{O}$ with smooth boundary $\partial \mathcal{O}$. We will assume throughout that $M$ is connected, that the 
$\calO \subset K$ so that the obstacles are contained in $K$. We will also fix the isometry to $\Reell^d \backslash \tilde K$ so that we have a natural coordinate system on $X \setminus K$. 

Let as usual $\der: \CnI(M;\Lambda T^* M) \to \CnI(M;\Lambda T^* M)$ be the differential on smooth forms 
and $\delta: \CnI(M;\Lambda T^* M) \to \CnI(M; \Lambda T^* M)$
its formal adjoint.
The Laplace-Beltrami operator $\Delta$ on differential forms is defined as $\Delta= \der \delta + \delta \der$.  We denote the restriction to forms
of degree $p$ by $\Delta_p$. 
There are natural boundary conditions that can be imposed on $\Delta_p$ to make this into an essentially self-adjoint operator which we now describe.
For a differential form $\omega \in \CnI(M; \Lambda^p T^* M)$ 
we denote its restriction to $\partial \mathcal{O}$ by $\omega|_{\partial \mathcal{O}}$. If $\iota: \partial \mathcal{O} \to M$ is the natural
inclusion map the restriction $\omega|_{\partial \mathcal{O}}$ is therefore a section in the pull back bundle $\iota^* (\Lambda^p T^* M)$.
This bundle is canonically isomorphic to $\Lambda^p T^* (\partial \mathcal{O}) \oplus \Lambda^{p-1} T^* (\partial \mathcal{O})$, the induced splitting being the split of $\omega_{\partial \mathcal{O}}$ 
into tangential and normal components $\omega|_{\partial \mathcal{O}} = \omega|_{\partial \mathcal{O},\mathrm{tan}} + \omega|_{\partial \mathcal{O}, \mathrm{nor}}$. The tangential component is the same as the
pull-back $\iota^* \omega$ of the differential form $\omega$ to $\partial \mathcal{O}$. There are several distinguished boundary conditions for the Laplace operator that lead to self-adjoint
extensions of the Laplace operator on compactly supported smooth forms.
{\sl Relative boundary conditions}
for the Laplace operator are defined as
\begin{gather*}
 \omega |_{\partial \mathcal{O},\mathrm{tan}}=0, \quad (\delta \omega)|_{\partial \mathcal{O},\mathrm{tan}}=0.
\end{gather*}
 {\sl Absolute boundary conditions} are defined to be
\begin{gather*}
 (\omega)|_{\partial \mathcal{O}, \mathrm{nor}}=0, \quad (\der \omega)|_{\partial \mathcal{O}, \mathrm{nor}}=0.
\end{gather*}
Note that if $\omega$ satisfies relative boundary conditions, then $* \omega$ satisfies absolute boundary conditions. Here $*$ is the Hodge star operator.

We will denote by $\Delta_{p,\mathrm{rel}}$ and $\Delta_{p,\mathrm{abs}}$ the self-adjoint extensions of unbounded operators in
$L^2(M,\Lambda^p T^* M)$ of $\Delta_p$ resulting from the respective boundary conditions. Since $* \Delta_{p,\mathrm{rel}} = \Delta_{n-p,\mathrm{abs}} *$
the Hodge star operator allows us to pass from relative to absolute boundary conditions.
The relative Laplacian $\Delta_{\mathrm{rel}}$ acting on differential forms can be written as the square of a self-adjoint operator $Q_{\mathrm{rel}} = \overline \delta + \overline d_0$ (see for example \cite{conner1954green,gaffney1954heat,borisov1988relative}).
Here $\overline \delta$ is the closure of the operator $\delta : \CnI(M,\Lambda T^* M) \to \CnI(M,\Lambda T^* M)$, and  $\overline d_0$
is the closure of the restriction of $d|_{\CnI(M_{\mathrm{int}} ,\Lambda T^* M_{\mathrm{int}}) } : \CnI(M_{\mathrm{int}},\Lambda T^* M_{\mathrm{int}}) \to \CnI(M_{\mathrm{int}},\Lambda T^* M_{\mathrm{int}})$. Here $M_{\mathrm{int}}= M \backslash \partial \mathcal{O}$ is the interior of $M$.

If $p=0$ the relative boundary conditions correspond to Dirichlet boundary conditions imposed on $\partial \mathcal{O}$, and absolute boundary conditions correspond to Neumann boundary conditions.

The Hilbert space $L^2(M,\Lambda^p T^* M)$ decomposes orthogonally into three invariant subspaces for $\Delta_{p,\mathrm{rel}}$ as follows (see \cite{conner1954green} and also \cite{gaffney1954heat})
$$
 L^2(M;\Lambda^p T^* M) = \overline{\delta  \CnI(M;\Lambda^{p+1} T^* M)} \oplus \overline{\der \CnI(M_{\mathrm{int}};\Lambda^{p-1} T^* M_{\mathrm{int}})} \oplus \mathcal{H}^p_{\mathrm{rel}}(M),
$$
where $\mathcal{H}^p_{\mathrm{rel}} = \ker{\Delta_{p,\mathrm{rel}}}$ is the space of relative $L^2$-harmonic $p$-forms, i.e. the space of square integrable forms that are closed, co-closed
and satisfy relative boundary conditions. 

The case $p=1$ is of particular interest in scattering theory of the electromagnetic field. Here the physics of the electromagnetic field in radiation gauge in the absence of charges and currents with the obstacles being perfect conductors is described by the operator $\Delta_{1,\mathrm{rel}}$ on co-closed forms. To be more precise, the electromagnetic vector potential
of a scattering wave in the frequency domain will satisfy relative boundary conditions and will be co-closed. It will also be a generalized eigenfunction of $\Delta_{\mathrm{rel}}$
as expressed by the Helmholtz equation $(\Delta_{\mathrm{rel}}-\lambda^2) A = 0$. The detailed spectral resolution and the scattering theory of $\Delta_{\mathrm{rel}}$ therefore describes scattering of electromagnetic waves in geometric backgrounds with perfectly conducting obstacles.

The spaces $\mathcal{H}^p_{\mathrm{rel}}$ are finite dimensional and directly related to the singular relative cohomology groups with compact support and coefficients in $\Reell$ as follows. If $d \geq 3$ then we have natural isomorphisms
\begin{gather*}
 \mathcal{H}^d_{\mathrm{rel}}(M) = \{ 0 \},\\
 \mathcal{H}^p_{\mathrm{rel}}(M) \cong H^p_0(M, \partial \calO) \cong H^p_0(X \backslash \calO), \quad \textrm{if } {p \not= d.}
\end{gather*}
Similarly, for the absolute boundary conditions one obtains for $d \geq 3$
\begin{gather*}
 \mathcal{H}^0_{\mathrm{abs}}(M) = \{ 0 \},\\
 \mathcal{H}^p_{\mathrm{abs}}(M) \cong H^p(M), \quad \textrm{if } {p \not= 0.}
\end{gather*}
These statements follow from a more general theorem by Melrose for scattering manifolds (as a consequence of Theorem 4 in case $\calO = \emptyset$ in \cite{melrose1994spectral})  
 and Carron who analysed the asymptotically flat case in great detail.
In particular, the statement above can be inferred using the exact sequence of Theorem 4.4 combined with Lemma 5.4 in \cite{carron2003l2}.
In dimension $d=2$ we have 
\begin{gather*}
 \mathcal{H}^0_{\mathrm{rel}}(M) =  \mathcal{H}^2_{\mathrm{rel}}(M) = \{ 0 \},\\
 \mathcal{H}^1_{\mathrm{rel}}(M) \cong  \mathrm{Im} \left( H^1_0(M, \partial \calO) \to H^1(M, \partial \calO)\right) \cong H^1(M, \partial \calO),
\end{gather*}
which follows from Proposition 5.5 in \cite{carron2003l2}. Moreover, the dual statement is 
\begin{gather*}
 \mathcal{H}^0_{\mathrm{abs}}(M) =  \mathcal{H}^2_{\mathrm{abs}}(M) = \{ 0 \},\\
 \mathcal{H}^1_{\mathrm{abs}}(M) \cong H^1_0(M).
\end{gather*}
The dimensions of these spaces, the $L^2$-Betti numbers, are therefore computable using the Mayer-Vietoris sequence. Note that it follows from the long exact sequence in cohomology that for manifolds Euclidean at infinity we always have $H^p_0(M, \partial \calO) \cong H^p(M, \partial \calO)$ if $1 < p < d$. For a more detailed description of the above natural isomorphisms see for example \cite{carron2003l2}.

\begin{example}
 If $X=\Reell^d$ and $\calO$ consists of $N$ non-intersecting balls, one obtains for $d>2$ that $\mathcal{H}^1_{\mathrm{rel}}(M) \cong \Reell^N$. These are the only non-trivial spaces of harmonic forms satisfying relative boundary conditions. In the case $d=2$, $N>0$ one has $\mathcal{H}^1_{\mathrm{rel}}(M) \cong \Reell^{N-1}$. 
 \end{example}
\begin{example} \label{worm}
A  wormhole $X$  in $\Reell^3$ is obtained by removing two non-intersecting balls and gluing the resulting spheres. In this case one obtains $\mathcal{H}^1_{\mathrm{rel}}(M) \cong \Reell$ and $\mathcal{H}^2_{\mathrm{rel}}(M) \cong \Reell$ as the only non-trivial spaces of square integrable harmonic forms.
\end{example}
\begin{example} \label{fulltorus}
Another interesting example is when $\calO$ is a full torus. In this case we also have $\mathcal{H}^1_{\mathrm{rel}}(M) \cong \Reell$ and 
$\mathcal{H}^2_{\mathrm{rel}}(M) \cong \Reell$ as the only non-trivial spaces of relative harmonic forms. 
\end{example}

In terms of $L^2$-Betti numbers the examples \ref{worm} and \ref{fulltorus} cannot be distinguished. 
We will see later that a certain refinement taking into account the decay properties of the harmonic forms distinguishes these spaces.

Choose an orthonormal basis $(u_j)_{j=1,\ldots,N}$ in $\ker_{L^2}(\Delta_{p,\mathrm{rel}})$ consisting of eigensections. If $P$ is the orthogonal projection onto $\ker_{L^2}(\Delta_{p,\mathrm{rel}})$ we have
$$
P = \sum_{j=1}^N \langle \cdot, u_j \rangle u_j.
$$
Each eigenfunction $u_j$ admits a multipole expansion 
$$
 u_j = \sum_{\nu} a_{\nu,j} \frac{1}{r^{\ell_{\nu}+d-2}} \Phi_\nu,
$$
if $(\Phi_\nu)$ is an orthonormal basis consisting of spherical harmonics of degree $\ell_\nu$, c.f. Appendix \ref{Amulti}.
For $\Phi \in L^2(\sphere;\Lambda^p \R^d)$ define
$$
 a_j(\Phi) := \sum_{\nu} \overline{a_{\nu,j}} \langle \Phi, \Phi_\nu \rangle,
$$
whenever the sum converges absolutely. In particular the sum is finite when $\Phi$ is a finite linear combination of spherical harmonics.
For each $\ell$ we can also define the matrices
$$
 a_{kj}^{(\ell)} =  \sum_{\nu, \ell_\nu=\ell} a_k(\Phi_\nu) \overline{a_j(\Phi_\nu)}. 
$$
The $a_{kj}^{(\ell)}$ do not depend on the choice of orthonormal basis $(\Phi_\nu)$ but they depend on the choice of orthonormal basis in $\ker_{L^2}(\Delta_{p,\mathrm{rel}})$.
However, the maps 
\begin{align}\label{Pl}
P^{(\ell)} = \sum\limits_{j,k=1}^N a_{kj}^{(\ell)} \langle \cdot, u_j \rangle  u_k: L^2(M;\Lambda^p T^*M) \to L^2(M;\Lambda^p T^*M)
\end{align}
are invariantly defined and self-adjoint.\\
Suppose $u$ is a harmonic form with a multipole expansion on $X \setminus K$ of the form
$$
  u(r \theta) = \sum_\nu \left( a_\nu \frac{1}{r^{d-2+\ell_\nu}} \Phi_\nu(\theta) + b_\nu r^{\ell_\nu} \Phi_\nu(\theta) \right),
 $$
in case $d=3$ or 
 $$
  u(r \theta) =\sum_{\nu,\ell_\nu=0} \left( a_\nu \log(r)  \Phi_\nu(\theta)+  b_\nu  \Phi_\nu(\theta) \right) +  \sum_{\nu,\ell_\nu>0} \left( a_\nu \frac{1}{r^{d-2+\ell_\nu}} \Phi_\nu(\theta) + b_\nu r^{\ell_\nu} \Phi_\nu(\theta) \right),
 $$
 in case $d=2$. We then define
$$
 a_u(\Phi) := \sum_{\nu} \overline{a_\nu} \langle \Phi, \Phi_\nu \rangle.
$$
Note that $a_j(\Phi) = a_{u_j}(\Phi)$.

Whereas $(u_j)$ gives the discrete part of the spectrum, the continuous part of the spectrum is described by the generalised eigenfunction
$E_\lambda(\Phi)$ that are indexed by $\Phi \in L^2(\sphere;\Lambda^p \R^d)$. For fixed $\lambda>0$ these generalised eigenfunctions are completely determined by their asymptotic behaviour
\begin{gather} 
 E_\lambda(\Phi) = \frac{\ee^{-\rmi \lambda r} \ee^{\frac{i\pi(d-1)}{4}} }{r^{\frac{d-1}{2}}} \Phi + \frac{\ee^{\rmi \lambda r} \ee^{-\frac{i\pi(d-1)}{4}} }{r^{\frac{d-1}{2}}} \Psi_\lambda(\Phi) + O\left(\frac{1}{r^{\frac{d+1}{2}}} \right),  \quad \textrm{for} \,\,\,r \to \infty,
\end{gather}
where $\Psi_\lambda(\Phi) = \tau S_\lambda \Phi$ and $\tau: L^2(\sphere;\Lambda^p \R^d) \to L^2(\sphere;\Lambda^p \R^d) $ is the pull-back of the antipodal map. The map $S_\lambda : L^2(\sphere) \to L^2(\sphere)$ is called the scattering matrix,
and $ A_\lambda= S_\lambda - \mathrm{id}$ is called the scattering amplitude.

\subsection{Statement of the main theorems} \label{theostate}

Suppose that $f,g:  Z \to W$ are functions that take values in a locally convex topological vector space and $h: Z \to \R$ . As usual we write
$f = g + O_W(h)$ if for every continuous semi-norm $p$ on $W$ there is a constant $C_p$ such that 
$p(f(\lambda)-g(\lambda)) \leq C_p | h(\lambda)|$ for all $\lambda \in Z$.  

\begin{theorem} \label{maintheorem1}
Let  $C_{d,\ell}$ be defined by 
$$
C_{d,\ell} = (-\rmi)^{\ell}\sqrt{2 \pi} \frac{1}{2^{\ell + \frac{d}{2}-1}} \frac{1}{\Gamma(\ell + \frac{d}{2})},
$$
and suppose that $\Phi \in C^\infty(\sphere; \Lambda^p \R^d)$ is a spherical harmonic of degree $\ell$, then the generalised eigenfunctions have for small $|\lambda|$
and bounded $|\arg \lambda|$ the following expansions
\begin{itemize}
 \item For $d=3$,
 \begin{gather*}
 E_\lambda(\Phi) =- (d-2+2\ell) C_{d,\ell} \lambda^{\ell +\frac{d-5}{2}} \sum_{j=1}^N a_j(\Phi) u_j 
 \\+ \rmi (d-2+2\ell) C_{d,\ell} \lambda^{\ell +\frac{d-3}{2}} \sum_{j,k=1}^N 
 a_{kj}^{(1)}a_j(\Phi) u_k + O_{C^\infty(M)}( \lambda^{\ell +\frac{d-1}{2}} ).
\end{gather*}
\item For $d=4$,
  \begin{gather*}
 E_\lambda(\Phi) = -(d-2+2\ell) C_{d,\ell} \lambda^{\ell +\frac{d-5}{2}} \sum_{j=1}^N a_j(\Phi) u_j 
 \\+ \frac{1}{4}(d-2+2\ell) C_{d,\ell} \lambda^{\ell +\frac{d-1}{2}} (-\log \lambda) \sum_{j,k=1}^N 
 a_{kj}^{(1)}a_j(\Phi)  u_k + O_{C^\infty(M)}( \lambda^{\ell +\frac{d-1}{2}} ).
\end{gather*}
\item For  $d \geq 5$,
  $$
   E_\lambda(\Phi) = -(d-2+2\ell) C_{d,\ell} \lambda^{\ell +\frac{d-5}{2}} \sum_{j=1}^N a_j(\Phi) u_j + O_{C^\infty(M)}( \lambda^{\ell +\frac{d-1}{2}} ).
  $$
\end{itemize}
In any dimension, if $\partial \calO = \emptyset$ or $p \not=1$, then $P^{(1)}=0$ and therefore $a_{kj}^{(1)}=0$ in the previous expansions.
\end{theorem}

This shows that all $L^2$-eigenfunctions appear as expansion coefficients of generalised eigenfunctions. Note that in even dimensions the functions are defined on a logarithmic cover of the complex plane and the estimates are understood as functions in an arbitrary but fixed sector of this cover (see Section \ref{hahnapp}). Hence the need for the restriction to bounded $\arg \lambda$.

\begin{theorem} \label{maintheorem2}
If $d$ is odd and $d\geq 3$ the resolvent $(\Delta_{rel} - \lambda^2)^{-1}$ (as an operator
from $L^2_{comp}$ to $H^2_{loc}$) has  for small $| \lambda |$  an expansion of the form
\begin{gather}
 R_\lambda = - \frac{P}{\lambda^2} + \rmi \frac{B_{-1}}{\lambda} + B(\lambda),
\end{gather}
where $B(\lambda)$ is holomorphic near zero.
 If $d=3$ then $B_{-1} = P^{(1)}$, and in particular we have that $B_{-1}=0$ if $\partial \calO = \emptyset$.
 If $d$ is odd and $d\geq 5$ then $B_{-1}=0$. 
\end{theorem}
The situation in even dimensions is different. In this case the resolvent $(\Delta_{rel} - \lambda^2)^{-1}$ (as an operator
from $L^2_{comp}$ to $H^2_{loc}$) is Hahn meromorphic at zero, i.e. it has a convergent expansion in terms of powers of $\lambda^2$ and $\-\log \lambda$
(see Appendix \ref{hahnapp} for the precise definition of this notion).
\begin{theorem} \label{maintheorem3}
If $d$ is even and $d \geq 4$ then the resolvent, as an operator
from $L^2_{comp}$ to $H^2_{loc}$, takes  for small $| \lambda |$ and bounded $|\arg \lambda|$ the form
$$
  R_\lambda=-\frac{P}{\lambda^2} + B_{-1} (-\log \lambda) + B(\lambda),
 $$
 where $B(\lambda)$ is Hahn-holomorphic and $B_{-1}= \frac{1}{4} P^{(1)}$ if $d=4$, and $B_{-1}=0$ if $d\geq 6$.
\end{theorem}

We now summarise the results for the two dimensional case. 

\begin{theorem} \label{maintheorem22}
 Suppose that $d=2$ and either $p=0$ or $p=2$. Then the resolvent, as an operator
from $L^2_{comp}$ to $H^2_{loc}$, takes for small $| \lambda |$ and bounded $|\arg \lambda|$ the form
$$
   R_\lambda=B_{-1} (-\log \lambda) + B(\lambda),
 $$
 where $B(\lambda)$ is Hahn-holomorphic. If $p=0$ and $\partial \calO=\emptyset$ then $B_{-1} = \langle \cdot , 1 \rangle 1 $, where $1$ is a constant function one.
 If $\partial \calO \not=\emptyset$ then $B_{-1}=0$. 
 In case $p=2$ we have $B_{-1} = \langle \cdot , *1 \rangle *1$, where $*1$ is the volume form.
\end{theorem}

The results in the case of one-forms in dimension two are rather complicated and require the definition of certain natural functions.
First note that if $\Psi$ is a linear function on $\R^2$ then $\Phi=d \Psi$ is a harmonic one form of degree zero. It turns out that there is a harmonic function
$u(\Psi) \in C^\infty(M)$ that satisfies relative boundary conditions such that
$$
 u(\Psi)(r,\theta) = \Psi(r,\theta) + O(1), \quad r \to \infty.
$$
By the maximum principle $u(\Psi)$ is uniquely determined up to a constant and therefore $\varphi(\Phi) = d u(\Psi)$ is well-defined. 
Note that $\varphi(\Phi) \in C^\infty(M;T^*M)$ is a one-form that satisfies relative boundary conditions and,  since the multipole expansion (see Appendix \ref{Amulti}) can be differentiated, we have
$$
 \varphi(\Phi) = \Phi + O\left(\frac{1}{r} \right).
$$
 Let $\Phi_0$ be the constant function $\Phi_0 = \frac{1}{2\pi}$.
In case there is a boundary, i.e. $\partial \calO \not= \emptyset$, there exists a unique harmonic function $g(\Phi_0)$ satisfying Dirichlet boundary conditions
 such that
 $$
  g(\Phi_0) =   \left(\log \frac{r}{2}\right) \Phi_0 + \beta \Phi_0 + O\left(\frac{1}{r}\right)
 $$ 
for $r$ sufficiently large. We then have $\psi(\Phi_0)=\der g(\Phi_0)$ is closed and co-closed, satisfies relative boundary conditions, and 
 $$
  \psi(\Phi_0) =  \frac{\der r}{r} \Phi_0 + O\left(\frac{1}{r^2}\right),
 $$
 where we have again used that the multipole expansion may be differentiated.

\begin{theorem} \label{maind2ge}
 Suppose that $p=1$ and $d=2$.  Let $\Phi$ be a spherical harmonic of degree $\ell$. Let $\psi = \psi(\Phi_0)$ in case $\partial \calO \not=\emptyset$ and define 
 $\psi=0$ otherwise.
 Then, for $| \lambda |$ small and bounded $|\arg \lambda|$ we have
 \begin{itemize}
  \item if $\ell=0$ we have $E_\lambda(\Phi)= \sqrt{2 \pi} \varphi(\Phi) \lambda^{\frac{1}{2}} + O_{C^\infty(M)}(\frac{\lambda^{1/2}}{-\log \lambda}).$
  \item if $\ell \geq 1$ we have
  \begin{gather*}
 E_\lambda(\Phi) = - 2\ell C_{2,\ell} \lambda^{\ell -\frac{3}{2}} \sum_{j=1}^N a_j(\Phi) u_j - 2\ell C_{2,\ell} \lambda^{\ell -\frac{3}{2}}  \frac{1}{-\log \lambda + \frac{\rmi \pi}{2} + \beta - \gamma} a_{\psi}(\Phi) \psi \\
+ 2 \ell C_{2,\ell} \lambda^{\ell +\frac{1}{2}} (-\log \lambda) \left( \frac{1}{4}\sum_{j,k=1}^N 
 a_{kj}^{(2)}a_j(\Phi)  u_k + \sum_{\ell_\nu=0} a_{\varphi(\Phi_\nu)}(\Phi) \varphi(\Phi_\nu)  \right) + O_{C^\infty(M)}( \lambda^{\ell +\frac{1}{2}}).
\end{gather*}
 \end{itemize}
Note that $a_j(\Phi)=0$ if $\ell=1$.
\end{theorem}

\begin{theorem} \label{maintheorem5}
 If $d=2$ and $p=1$ the resolvent, as an operator
from $L^2_{comp}$ to $H^2_{loc}$, has an expansion  for small $| \lambda |$ and bounded $|\arg \lambda|$ of the form
 $$
   R_\lambda = -\frac{P}{\lambda^2} - \frac{1}{\lambda^2} \frac{1}{-\log \lambda + \frac{\rmi \pi}{2} + \beta - \gamma} Q + B_{-1} (-\log \lambda) + B(\lambda),
 $$
 where $B(\lambda)$ is Hahn holomorphic and $\gamma$ is the Euler-Mascheroni constant. 
 Here $Q=0$ in case $\partial \calO=\emptyset$, and $Q =  \langle \cdot, \psi(\Phi_0) \rangle \psi(\Phi_0)$ if $\partial \calO \not= \emptyset$. Moreover, we have 
 $$
  B_{-1}= \frac{1}{4} P^{(2)} +  \sum\limits_{\ell_\nu=0}  \langle \cdot, \varphi(\Phi_\nu) \rangle \varphi(\Phi_\nu).
 $$
\end{theorem}

The form $\psi(\Phi_0)$ is, by construction, a cohomology class that generates the image of the map $H^0(\mathbb{S}^{d-1}) \to H^1(M,\partial \calO)$. This image is not detected by $L^2$-cohomology theory in the two dimensional case and the above shows that this cohomology class  appears as a zero-energy resonant state instead.

Each of the expansions of the generalised eigenfunctions can be used to derive an expansion of the scattering matrix and the scattering amplitude.
Section \ref{ScatterSectionExp} describes the detailed expansion depending on the dimension. The leading order behavior is independent of the dimension and can be summarised into the following theorem.

\begin{theorem} \label{scatterexp}
If $d \geq 3$ and $\Phi$ is a spherical harmonic of degree $\ell$, then
 \begin{gather*}
 \langle A_\lambda \Phi, \Phi_\nu \rangle = \left( -\frac{\rmi}{2} (d-2+2\ell) (d-2+2\ell_\nu)C_{d,\ell} \overline{C_{d,\ell_\nu} }\sum_{j=1}^N a_j(\Phi) \overline{a_j(\Phi_\nu)} \right) \lambda^{\ell +\ell_\nu + d-4} + r(\lambda),
  \end{gather*}
  where for small $| \lambda |$ and bounded $|\arg \lambda|$ we have
  \begin{itemize}
  \item  $r(\lambda)=O( \lambda^{\ell +\ell_\nu + d-3})$ if $d=3$, 
   \item $r(\lambda)=O( \lambda^{\ell +\ell_\nu + d-2}) (-\log \lambda)$ if $d=4$,
   \item $r(\lambda)=O( \lambda^{\ell +\ell_\nu + d-2})$ if $d>4$.
  \end{itemize}
  If $P^{(\ell)}=0$, then $ \langle A_\lambda \Phi, \Phi_\nu \rangle=O(\lambda^{\ell+\ell_\nu +d-2})$, in particular $\| A_\lambda \|_{L^2\to H^s} = O(\lambda^{d-2})$ for any $s \in \R$ and $| \lambda |$ small and bounded $|\arg \lambda|$.
\end{theorem}

The two dimensional case is more involved due to the existence of a zero resonant state when $\partial \calO \not=\emptyset$. In this case we have $\| A_\lambda \|_{L^2\to H^s} = O(\frac{1}{-\log \lambda})$ for any $s \in \R$. Precise expansions depend on the form degree and the presence of an obstacle. 

\begin{theorem}  \label{maintheorem6}
If $d =2$ and $\Phi$ is a spherical harmonic of degree $\ell$, then, for $| \lambda |$ small and bounded $|\arg \lambda|$, we have
\begin{itemize}
 \item if $p=0$ or $p=2$  then $\| A_\lambda \|_{L^2\to H^s} = O(\frac{\lambda}{-\log \lambda})$ and $$ \langle A_\lambda \Phi, \Phi_\nu \rangle_{L^2(\sphere)}=O(\lambda^{\ell + \ell_\nu}).$$
 \item if $p=1$, using the notation of Theorem \ref{maind2ge},
 \begin{align*}
\quad & \langle A_\lambda \Phi, \Phi_\nu \rangle_{L^2(\sphere)} \\&= 
 - 2 \rmi \ell \, \ell_\nu \, C_{2,\ell}\,  \overline{C_{2,\ell_\nu}} \lambda^{\ell +\ell_\nu -2}  \left(\left(\sum_{j=1}^N a_j(\Phi) \overline{a_j(\Phi_\nu)}\right)  +  \frac{1}{-\log \lambda + \frac{\rmi \pi}{2} + \beta - \gamma} a_{\psi}(\Phi) \overline{a_{\psi}(\Phi_\nu)} \right) \\&  + O(\lambda^{\ell + \ell_\nu}(-\log \lambda)).
  \end{align*}
\end{itemize}

\end{theorem}

The expansions of the generalised eigenfunctions encode the finer structure on the space $\mathcal{H}^p_{rel}(M)$ given by the order of vanishing at infinity.
Indeed, the space $\mathcal{H}^p_{rel}(M)$ carries a natural filtration $$\mathcal{H}^{p,m+1}_{rel}(M) \subset \mathcal{H}^{p,m}_{rel}(M) \subset \mathcal{H}^p_{rel}(M),$$
where $\mathcal{H}^{p,m}_{rel}(M)$, defined for $m \geq 1$, is the space of $L^2$-harmonic forms satisfying relative boundary conditions whose multipole expansion only has nonzero terms of order $\ell \geq m$.
In parts this filtration has topological significance. 
\begin{theorem} \label{maincohomo}
 If $d \geq 3$ and $0<p \leq d$ then $\mathcal{H}^{p,1}_{rel}(M) / \mathcal{H}^{p,2}_{rel}(M)$
  isomorphic to the kernel of the map 
 $H_0^p(M,\partial \calO) \to H^p(M,\partial \calO)$. 
 In particular $\mathcal{H}^{p,1}_{rel}(M) = \mathcal{H}^{p,2}_{rel}(M)$ if $\partial \calO = \emptyset$ or $1<p \leq d$, and 
  $\dim \mathcal{H}^{p,1}_{rel}(M) / \mathcal{H}^{p,2}_{rel}(M) =1$ if  $\partial \calO \not= \emptyset$ and $p=1$.
\end{theorem}
By the long exact sequence in cohomology the kernel of $H_0^p(M,\partial \calO) \to H^p(M,\partial \calO)$ is the same as the image of the map
$H^p(\sphere) \to H_0^p(M,\partial \calO)$. This map is given by the limit of a the generalised eigenfunction (see Section \ref{cohoscatter}).
The fact that these spaces are isomorphic can probably also be inferred from the general framework \cite{melrose1993atiyah}.
As explained before this is not true in dimension two where this image shows up as a zero resonance state.

Finally, we give a short proof of the relative Birman-Krein formula in our setting with particular emphasis on the low energy behaviour. The expansions of the scattering amplitudes therefore directly translate into the asymptotic properties of the spectral shift function $\xi \in L^1_\loc(\R)$ (see Section \ref{BMformsec} for a definition) which is expressed as
$$
 \xi(\mu) = \left \{ \begin{matrix} 0 & \mu < 0,\\ \left(\beta_p + \beta_{res} \right) + \eta(\mu) & \mu \geq 0, \end{matrix} \right.
$$
where 
$$
 \eta(\mu) =  \frac{1}{2 \pi \rmi }  \int\limits_0^{\sqrt{\mu}}\tr\left( S^*(\lambda) S'(\lambda)\right) \der \lambda,
$$
the integer $\beta_p$ is the $L^2$-Betti number, and $\beta_{res}$ equals one in case $d=2, p=1, \partial \calO \not= \emptyset$ and zero otherwise.
The jump of the spectral shift function at zero was also computed by Carron in \cite{carron2004saut} using a different method.
The expansions of $A_\lambda$ can be used to prove refined expansions for the spectral shift function at zero. An example is the following theorem.
\begin{theorem}  \label{maintheorem7}
Suppose $d \geq 3$. Then we have for $0\leq \mu \leq 1$ the estimate
$$
 \xi(\mu) = \beta_p + \alpha_p \mu^{\frac{d-2}{2}} + \left\{ \begin{matrix}  O(\mu^{\frac{d-1}{2}})  & d \textrm{ odd} \\ O(\frac{1}{-\log \mu}) & d \textrm{ even} \end{matrix} \right.,
$$
where in case $p=1$ we have  $\alpha_1 =  -\frac{2^{1-d} d^2}{\Gamma \left(\frac{d-2}{2}\right)^2} \tr\left(P^{(1)} \right)$, in case $p>1$ we have 
$\alpha_p=0$, and finally $a_0 = a_1$.
\end{theorem}

Similar expansions can be derived in dimension two. The details of the expansions of the spectral shift function and their applications will be discussed elsewhere.

\subsection{Plan of the paper}

The paper is organised as follows. Section \ref{basicscattersection} gives a review of basic spectral theory for the Hodge Laplacian on $p$-forms with boundary conditions. It develops the theory of generalised eigenfunctions based on Bessel function expansions, and it states the main spectral decomposition results that relate the generalised eigenfunctions, the resolvent, and the scattering matrix. Section \ref{expansionsection} establishes the main technical lemmata to derive low energy expansions of the resolvent and the generalised eigenfunctions. Section \ref{ScatterSectionExp} employs the obtained expansions for the resolvent and the eigenfunction expansions to derive expansions for the scattering amplitude.
In Section \ref{cohoscatter} we establish the relation between cohomology and the low energy expansions of the generalised eigenfunctions. Finally, we give a proof of the Birman-Krein formula in our setting in Section \ref{BMformsec}. Since the main theorems require lengthy arguments in parts bootstrapping we have decided to summarise them in the introduction and explain in Section \ref{sevenandfinal} in detail how they are obtained from the main body of the text.

\subsection{Possible generalisations}
For the purposes of this article we have focused on the important case of compact perturbations of Euclidean space. This is also the case that is most relevant in physics.
There are two natural generalisations of this. One is to consider compact perturbations of globally symmetric spaces. In this case the theory of Hahn meromorphic functions can still be applied with a difference being that the bottom of the continuous spectrum is generally not at zero any more and therefore the cohomological interpretation will be lost.
Another generalisation is to consider manifolds that are exact cones outside a compact set and even more generally scattering manifolds as introduced by Melrose (\cite{melrose1994spectral}).  Large parts of our analysis  carry over to that setting but the absence of a canonical basis in cotangent space complicates things on a notational level. Finally one can obtain results about short range perturbations of the metric by approximating them by manifolds that are Euclidian at infinity.

\section{Stationary scattering theory and the spectral resolution} \label{basicscattersection}

In this section we describe the spectral resolution of the operator $\Delta_{p,\mathrm{rel}}$ in our setting.
The main results presented here are well known for functions and standard in stationary black box scattering theory, 
and they can be derived for $p$-forms by a straightforward modification of the arguments and statements. For the convenience of the reader and the sake of completeness we prove the main spectral decomposition results for $p$-forms that are required for our analysis. We also give a construction of generalised eigenfunctions based on Bessel functions that we could not find in the literature.
This section generalises, with the obvious modifications, to  the case of the operator $\Delta_{p,\mathrm{rel}} +V$ where $V \in \CnI(M;\mathrm{End}(\Lambda^p T^*M))$ is a compactly supported symmetric potential. 
In this paper we focus only on the case of the Laplace operator and in order to keep the notation as simple as possible we omit the potential. For general background on the theory of black-box scattering for functions and current developments we refer to the recent monograph \cite{DZ}.
We will denote the kernel of the self-adjoint operator $\Delta_{p,\mathrm{rel}}$ by $\mathrm{ker}_{L^2}(\Delta_{p,\mathrm{rel}})$. This will distinguish it notationally from the kernel
$\mathrm{ker}_{C^\infty}(\Delta_{p,\mathrm{rel}})$ of the differential operator $\Delta_{p,\mathrm{rel}}$ acting on smooth forms satisfying relative boundary conditions
but without imposing the condition of square-integrability. 
The resolvent
$$
 R_\lambda:=(\Delta_{p,\mathrm{rel}} - \lambda^2)^{-1}
$$
is a holomorphic family of $L^2$-bounded operators for $\Im(\lambda)>0$. It is well known that the resolvent has a meromorphic extension to a family of bounded operators from $H^s_{\comp}(M) \to H^{s+2}_{\loc}(M)$ with finite rank negative Laurent coefficients to a larger Riemann surface $\mathcal{R}$. In the case the dimension $d$ is odd, we have $\mathcal{R}= \C$. In the case $d$ is even $\mathcal{R}$ is a logarithmic cover of the complex plane with branch
cut at $\rmi \mathbb{R}_-$. In either case $\mathcal{R}$ contains the set  $\R \setminus \rmi \overline{\mathbb{R}_-}$.
It is also known that the resolvent is holomorphic near $\R \setminus \{0\}$ (absence of embedded eigenvalues). The singularity structure at zero will be discussed in detail in Section \ref{expansionsection}.

\subsection{Coordinates in the Euclidean part}
Since $M$ is Euclidean at infinity there is a compact set $K$ such that $M \setminus K$ is isometric to $\R^d \setminus B_R(0)$.
On $M \setminus K$ we have a natural coordinate system. We will use both Cartesian coordinates $x \in \R^d$ and spherical coordinates
$(r,\omega) \in (R,\infty) \times \mathbb{S}^{d-1}$, where $r=|x|$ and $\omega= \frac{x}{|x|}$, where it is understood. We choose a smooth function $\chi \in C^\infty(M)$
supported in $M \setminus K$ such that $1-\chi$ is compactly supported. Using the Cartesian coordinates and the orthonormal frame $(\der x^1,\ldots,\der x^d)$
we trivialise the bundle $T^*(M \setminus K)$ and thereby identify forms in $C^\infty(M \setminus K; \Lambda^p T^*M)$ with vector-valued functions
in $C^\infty(M \setminus K; \Lambda^p \R^d)$.

\subsection{Incoming and outgoing forms}
Suppose that $f \in C^\infty(M;\Lambda^p T^* M)$, then if $\lambda \in \R$ is non-zero and $g = (\Delta_{p}  - \lambda^2) f$ is compactly supported
then $f$ is called {\sl outgoing} for $\lambda$ if $f = R_\lambda h$, where $h$ is compactly supported. The section $f$ is called {\sl incoming} for $\lambda$ if it is outgoing for $-\lambda$.
It is easy to see that if $f\chi \in \CnI(M)$ then $f$ is outgoing if and only if $(1 - \chi f)$ is outgoing. It follows that the definition depends only on the behavior of $f$ at infinity.
Moreover, the notion does not depend on the precise structure of the resolvent and is also independent of the compact part $M \setminus K$. This means that
$f$ is outgoing on $M$ is equivalent to $f |_{M \setminus K}$ being outgoing on $\R^d$.  One can use this to see that an outgoing $f$ has an asymptotic expansion
$$
 f = \frac{\ee^{\rmi \lambda r}}{r^{\frac{d-1}{2}}} \Phi + O\left(\frac{1}{{r^{\frac{d+1}{2}}}} \right), \quad  r \to \infty,
$$
where $\Phi \in C^\infty(\sphere;\Lambda^p\R^d)$ is the restriction of an entire function on $\C^d$ to the sphere. The expansion can be differentiated in $r$, c.f. Appendix \ref{Abf} for details. We refer to Appendix \ref{outgoingapp} for proofs of the above claims in our setting.

\subsection{Generalised plane waves}

Given $\omega \in \mathbb{S}^{d-1}$ and $v \in \Lambda^p \R^d$ we define the distorted plane wave $$e_\lambda(\omega,v) \in C^\infty(M; \Lambda^p T^* M)$$ by
$$
 e_\lambda(\omega,v)(x) = \chi(x) v \mathrm{e}^{-\rmi \lambda \omega \cdot x} - R_\lambda (\Delta_{p}  - \lambda^2)  \chi(x) v \ee^{-\rmi \lambda \omega \cdot x}.
$$
By construction $e_\lambda(\omega,v)$ is a meromorphic function on $\C \setminus \rmi \overline{\R_-}$ with
$$
 (\Delta_{p} - \lambda^2) e_\lambda(\omega,v) =0
$$
that satisfies relative boundary conditions but is generally not in $L^2(M; \Lambda^p T^* M)$.

\subsection{Generalised eigenfunctions and the scattering matrix}
Similarly, given $\Phi \in L^2(\mathbb{S}^{d-1},\Lambda^p \R^d)$ one can define the distorted spherical waves $E_\lambda(\Phi)$ by
$$
 E_\lambda(\Phi)(x) =\left(\frac{\lambda}{2\pi}\right)^{\frac{d-1}{2}}\int_{\mathbb{S}^{d-1}} e_\lambda(\omega, \Phi(\omega))(x) \der \omega.
$$
These distorted spherical waves are generalised eigenfunctions and can also be expressed directly in terms of Bessel and Hankel functions. In order to describe this it is convenient to introduce the following notation.
On the sphere $\sphere$ we have an orthonormal basis $(\phi_\nu)_{\nu \in I}$ in $L^2(\sphere)$ consisting of eigenfunctions of the Laplacian with eigenvalues $\ell_\nu (\ell_\nu+d-2)$. Here $I$ is an index set used to enumerate the spherical harmonics. 
These spherical harmonics can be obtained by restricting homogeneous harmonic polynomials to the sphere.
Given $g \in L^2(\sphere)$ we can therefore write $g = \sum\limits_\nu a_\nu \phi_\nu$, where $a_{\nu}(g)=\langle g, \phi_{\nu}\rangle_{L^2(\mathbb{S}^{d-1})}$ and convergence is in $L^2(\sphere)$. We denote by $\mathcal{H}_\ell(\sphere)$ the spherical harmonics of degree $\ell$, so we have the Hilbert space direct sum $L^2(\sphere) =  \bigoplus\limits_{\ell=0}^\infty \mathcal{H}_\ell(\sphere)$.
\begin{definition}
For $g\in L^2(\mathbb{S}^{d-1})$, we define $\tilde j_{\lambda}(g) \in C^{\infty}(\R^d \setminus \{0\})$ by 
 \begin{align*} 
\tilde j_\lambda(g)(r \theta) &=2\lambda^{\frac{d-1}{2}}\sum\limits_{\nu} a_{\nu}(g) \phi_{\nu}(\theta)j_{d,\ell_{\nu}}(\lambda r) (-\rmi)^{\ell_{\nu}},
 \end{align*}
 where $j_{d,\ell}$ is the spherical Bessel function in dimension $d$ (see Appendix \ref{Abf}).
\end{definition}
In the same way we have in the  vector-valued case the Hilbert space direct sum $L^2(\sphere; \Lambda^p\mathbb{R}^d) =  \bigoplus\limits_{\ell=0}^\infty \mathcal{H}_\ell^p(\sphere)$, where $\mathcal{H}_\ell^p(\sphere) \subset C^\infty(\Lambda^p\mathbb{R}^d)$ denotes the vector-space of vector-valued spherical harmonics of degree $\ell$ and form-degree $p$.

If $\Phi \in L^2(\sphere; \Lambda^p\mathbb{R}^d)$, then $\tilde{j}_{\lambda}(\Phi)$ is in $C^\infty(\R^d,\Lambda^p\mathbb{R}^d)$ and solves $(\Delta_p-\lambda^2)\tilde{j}_{\lambda}(\Phi)=0$. 
Using the properties of spherical Bessel functions it is not difficult to show that $\tilde j_\lambda(\Phi)$ is a holomorphic function in $\lambda$ taking values in $C^\infty(\R^d)$.

\begin{proposition} \label{besseljefunction}
We have that 
\begin{align}
E_{\lambda}(\Phi)=\chi\tilde{j}_{\lambda}(\Phi)-R_{\lambda}(\Delta_p -\lambda^2)(\chi\tilde{j}_{\lambda}(\Phi))=
\chi\tilde{j}_{\lambda}(\Phi)-R_{\lambda}[\Delta_p,\chi](\tilde{j}_{\lambda}(\Phi)).
\end{align}
\end{proposition} 
\begin{proof} This follows from the equality (see Appendix \ref{Abf})
\begin{align}
(2\pi)^{-\frac{d-1}{2}} \int\limits_{\mathbb{S}^{d-1}}\exp(-i\lambda x\cdot\omega)g(\omega)\der \omega= 2\sum\limits_{\nu} a_{\nu}\phi_{\nu}\left(\frac{x}{ r}\right)j_{d,l_{\nu}}(\lambda r) 
(-\rmi)^{l_{\nu}}=\lambda^{\frac{1-d}{2}}\tilde{j}_\lambda(g),\\ 
\textrm{with}\,\, a_{\nu}=\langle g, \phi_{\nu}\rangle_{L^2(\mathbb{S}^{d-1})}. \nonumber
\end{align}
\end{proof}

The generalised eigenforms $E_{\lambda}(\Phi)$, by construction, depend meromorphically on $\lambda$ and are holomorphic in $\lambda$ in an open neighbourhood of 
$\R \setminus \{0\}$.

\begin{definition} \label{hankelsums}
 We define $\tilde h^{(1)}_{\lambda}(g)$, and  $\tilde h^{(2)}_{\lambda}(g)$ by 
 \begin{align*}
  \tilde h^{(1)}_\lambda(g)(r \theta) &= \lambda^{\frac{d-1}{2}} \sum\limits_{\nu} a_{\nu}(g) \phi_{\nu}(\theta)h^{(1)}_{\ell_{\nu}}(\lambda r)(-\rmi)^{\ell_{\nu}},\\
  \tilde h^{(2)}_\lambda(g)(r \theta) &=\lambda^{\frac{d-1}{2}} \sum\limits_{\nu} a_{\nu}(g) \phi_{\nu}(\theta)h^{(2)}_{\ell_{\nu}}(\lambda r) (-\rmi)^{\ell_{\nu}},
 \end{align*}
 whenever the sums converge in $C^{\infty}(\mathbb{R}^{d}\setminus \{0\})$ where $h^{(1)}_\ell$, and $h^{(2)}_\ell$ are the spherical Hankel functions in dimension $d$
 (see Appendix \ref{Abf}).
\end{definition}

The above definition does not depend on the choice of spherical harmonics.

We have
\begin{proposition} \label{EoutofHankel}
 For every $\lambda \in \R \setminus \{0\}$ and $\Phi \in L^2(\sphere, \Lambda^p\mathbb{R}^d)$ there exists a unique $A_\lambda(\Phi) \in C^\infty(\sphere, \Lambda^p\mathbb{R}^d)$ such that
 \begin{gather*}
  E_{\lambda}(\Phi)|_{M \setminus K} = \tilde{j}_{\lambda}(\Phi) + \tilde h^{(1)}_\lambda(A_\lambda \Phi).
 \end{gather*}
\end{proposition}
\begin{proof}
 If we define $g=E_{\lambda}(\Phi) - \chi \tilde{j}_{\lambda}(\Phi)$ then $g$ is outgoing for $\lambda \not= 0$ and smooth. It follows from Lemma \ref{repoutgoing} that on $M \setminus K$ we have
 $g =  \tilde h^{(1)}_\lambda(A_\lambda(\Phi))$ for a unique $A_\lambda(\Phi) \in C^\infty(\sphere, \Lambda^p\mathbb{R}^d)$.
\end{proof}

If $\Phi$ is smooth one can use the well-known asymptotics of the Bessel and Hankel functions to see that for fixed $\lambda >0$ 
$$
 E_\lambda(\Phi) = \frac{\ee^{-\rmi \lambda r} \ee^{\frac{i\pi(d-1)}{4}} }{r^{\frac{d-1}{2}}} \Phi +  \frac{\ee^{\rmi \lambda r} \ee^{-\frac{i\pi(d-1)}{4}}}{r^{\frac{d-1}{2}}} \left(\tau(\Phi) + \tau (A_\lambda \Phi) \right) + O\left(\frac{1}{{r^{\frac{d+1}{2}}}} \right),  \quad \textrm{for} \quad r \to \infty,
$$
where $\tau: C^\infty(\sphere, \Lambda^p\mathbb{R}^d) \to C^\infty(\sphere, \Lambda^p\mathbb{R}^d), f(\theta)\mapsto f(-\theta)$ is the pull-back of the antipodal map. See for example  \cite[Section 5.1]{melrose1995geometric} where this is used to define the scattering matrix.
This asymptotic expansion may be differentiated, c.f., Appendix \ref{Abf}. 
Together with Rellich's uniqueness theorem this gives the well-known statement that, given $\lambda >0$ for every $\Phi\in C^\infty(\sphere,\Lambda^p \R^d)$ there exists 
a unique $\Psi_\lambda \in C^\infty(\sphere,\Lambda^p \R^d)$ and a unique solution $E_\lambda(\Phi)$ of $( \Delta_p  - \lambda^2) E_\lambda(\Phi)=0$ such that
\begin{gather} \label{superequation}
 E_\lambda(\Phi) = \frac{\ee^{-\rmi \lambda r} \ee^{\frac{i\pi(d-1)}{4}} }{r^{\frac{d-1}{2}}} \Phi + \frac{\ee^{\rmi \lambda r} \ee^{-\frac{i\pi(d-1)}{4}} }{r^{\frac{d-1}{2}}} \Psi_\lambda + O\left(\frac{1}{{r^{\frac{d+1}{2}}}} \right),  \quad \textrm{for} \quad r \to \infty,
\end{gather}
and by comparison with the above we get $\Psi_\lambda = \tau\left(\Phi + A_\lambda \Phi \right)$. 
\begin{definition}
The scattering matrix $S_\lambda: L^2(\sphere, \Lambda^p\mathbb{R}^d) \to  L^2(\sphere, \Lambda^p\mathbb{R}^d)$ is defined by $S_{\lambda}= \mathrm{id} + A_\lambda$. 
\end{definition}
If $\Phi \in C^\infty(\sphere,\Lambda^p \R^d)$ then $\der r \wedge \Phi \in C^\infty(\sphere,\Lambda^{p+1} \R^d)$ and $\iota_{\der r}  \Phi \in C^\infty(\sphere,\Lambda^{p+1} \R^d)$, where $\iota_{\der r}$ is interior multiplication of differential forms by $\der r$. 
\begin{proposition} \label{difftheo}
We have the following equalities,
\begin{gather*} 
 \der E_\lambda(\Phi) = -\rmi \lambda E_\lambda(\der r \wedge \Phi),\quad  \delta E_\lambda(\Phi) = \rmi \lambda E_\lambda(\iota_{\der r}\Phi).
\end{gather*} 
Moreover, we also have $\der r \wedge S_\lambda \Phi = S_\lambda \der r \wedge \Phi$ and $\iota_{\der r} S_\lambda \Phi = S_\lambda \iota_{\der r} \Phi$.
\end{proposition}
\begin{proof}
 By analyticity it is sufficient to prove the equalities for $\lambda$ in a neighbourhood of the real line. Computing the leading order term from
 \eqref{superequation} gives for fixed $\lambda>0$
 $$
  \der E_\lambda(\Phi) = -\rmi \lambda \left( \frac{\ee^{-\rmi \lambda r} \ee^{\frac{\rmi\pi(d-1)}{4}} }{r^{\frac{d-1}{2}}} \der r \wedge \Phi - \frac{\ee^{\rmi \lambda r} \ee^{-\frac{\rmi \pi(d-1)}{4}} }{r^{\frac{d-1}{2}}} \der r \wedge \Psi_\lambda \right) + O\left(\frac{1}{{r^{\frac{d+1}{2}}}} \right),  \; \textrm{for} \quad r \to \infty.
 $$
 Now one simply compares the leading order coefficients and uses Rellich's theorem to conclude that  $\der E_\lambda(\Phi) = -\rmi \lambda E_\lambda(\der r \wedge \Phi)$
 and $\tau S_{\lambda} \der r \wedge \Phi =- \der r \wedge \tau S_{\lambda} \Phi$. Note that $\tau \der r = -\der r$.
 The formula for $\delta E_\lambda(\Phi)$ is proved in exactly the same way.
\end{proof}

\subsection{Functional Equations}
From the uniqueness statement in Prop. \ref{EoutofHankel}  one deduces that in case the dimension is odd that
\begin{gather} \label{funcE1}
 E_{-\lambda}(\Phi) = (\rmi)^{d-1 }E_{\lambda}(\tau \; S_{-\lambda} \Phi),\quad
 S_\lambda \; \tau \; S_{-\lambda} =  \tau,
\end{gather}
and in case the dimension is even we have
\begin{gather} \label{funcE2}
 E_{-\lambda}(\Phi) = (\rmi)^{d-1} E_{\lambda}(\tau ( 2\;\mathrm{id} - S_{-\lambda}) \Phi),\quad
  S_\lambda \; \tau \; (2 \; \mathrm{id} - S_{-\lambda}) = \tau.
\end{gather}
We have used the formulae \eqref{rotationhankel1} and \eqref{rotationhankel2} for the analytic continuation of the Hankel functions.
In the even dimensional case $E_\lambda(\Phi)$ and $S_\lambda$ are Hahn meromorphic and only defined on a logarithmic cover of the complex plane the notation
$-\lambda$ for $\lambda>0$ needs an explanation. Throughout the paper, if $z \not=0$ is an element of the logarithmic cover of the complex plane, we will define $-z = \mathrm{e}^{\rmi \pi} z$ which corresponds to a counterclockwise rotation by $\pi$. Some care is needed with this notation, however, since then $-(-z)$ is on a different sheet than $z$. The complex conjugate of $z= r \ee^{\rmi \phi}$ in the logarithmic cover is defined by $\overline{z} = r \ee^{-\rmi \phi}$. For $z>0$ the complex conjugate
$\overline{-z}$ of $-z$ is then also on another branch than $-z$, namely $\overline{-z} = \ee^{-\rmi \pi} z$.
The functional relation above in the case of even dimensions seems to have been widely misstated in the literature (see \cite{MR3430366} for a careful analysis and a clarification of this formula).

Green's theorem applied to the identity $$\langle (\Delta-\lambda^2) E_\lambda(\Phi),E_\lambda(\Psi) \rangle - \langle E_\lambda(\Phi),  (\Delta-\lambda^2) E_\lambda(\Psi) \rangle=0$$ for $\lambda>0$ and analytic continuation shows that
$$
 S_{\bar \lambda}^{*} S_\lambda = \mathrm{id},
$$
in particular $S_\lambda$ is unitary for positive real $\lambda$.
Comparison gives
$$
 A_{\overline\lambda}^* = (-1)^{d-1} \tau \;A_{-\lambda} \; \tau.
$$

\subsection{Analytic properties of the scattering amplitude}
If $\Phi \in L^2(\sphere, \Lambda^p\mathbb{R}^d)$ then outside the support of $\chi$, we have
$$
 \tilde h^{(1)}_\lambda(A_\lambda \Phi) = R_\lambda [\Delta_p, \chi]  \tilde{j}_{\lambda}(\Phi).
$$
Now choose cutoff functions $\eta_1,\eta_2 \in C^\infty(M)$, supported in $X \backslash K$, and $\eta_1=1$ in a neighbourhood of the support of $\eta_2$. It follows $1-\eta_2$ is compactly supported. Let $B_R$ denote a ball of radius $R$. The following Lemma is equivalent to a well known representation of the scattering amplitude by the resolvent (see for example \cite[Prop. 2.1]{petkov2001semi}).
\begin{lemma} \label{aoutofresolvent}
 For large enough $R\gg 0$  and $\lambda \in \R$, we have
  $$
 -(2\rmi \lambda)e^{\frac{(d-1)\pi \rmi}{4}}\mathrm{Vol}(\sphere) \left( \frac{2 \pi}{\lambda} \right)^{\frac{d-1}{2}} (A_\lambda \Phi)(\omega)=\langle [\Delta_p, \eta_2]   R_\lambda [\Delta_p, \chi]  \tilde{j}_{\lambda}(\Phi), \eta_1 \ee^{- \rmi \lambda \omega x} \rangle_{L^2(B_R)} .
  $$
\end{lemma}
\begin{proof}
Note that
$$
 \langle (\Delta_p-\lambda^2) \eta_2 \tilde h^{(1)}_\lambda(\Psi) , \eta_1 \ee^{- \rmi \lambda \omega x} \rangle_{L^2(B_R)}
$$
is independent of $R$ for sufficiently large $R$  as $(\Delta_p-\lambda^2) \eta_2 \tilde h^{(1)}_\lambda(\Psi) = [\Delta_p,\eta_2]  \tilde h^{(1)}_\lambda(\Psi)$ is compactly supported.
Integration by parts gives only boundary terms since the derivative of $[\Delta_p,\eta_1]  \ee^{- \rmi \lambda \omega x}$ has support where $\eta_2$ vanishes.
Therefore the integral is given by
$$
- \int_{|x|=R} \ee^{\rmi \lambda \omega x} \frac{\partial}{\partial r} \tilde{h}^{(1)}_\lambda(\Psi) \der x +\int_{|x|=R}\left(\frac{\partial}{\partial r}\ee^{\rmi \lambda \omega x}\right) \tilde{h}^{(1)}_\lambda(\Psi) \der x,
$$
and is equal to the constant term in its large $R$ expansion. Using \eqref{appendix_bexp} one can compute this constant term as
$$
-(2 \rmi \lambda)e^{\frac{(d-1)\pi \rmi}{4}} \mathrm{Vol}(\sphere) \left( \frac{2\pi}{\lambda} \right)^{\frac{d-1}{2}}(A_{\lambda}\Phi)(\omega).
$$
\end{proof}

From this one concludes that $S_\lambda$ admits a meromorphic extension to $\C \setminus (-\rmi [0,\infty))$ with finite rank negative Laurent coefficients.
Since $S_{\lambda}$ is unitary for real $\lambda$ this implies immediately that $S_{\lambda}$ is holomorphic in a neighborhood of $\R \setminus  \{0\}$.  Depending on the dimension one can make a more precise statement about the behavior of $A_\lambda$ near $\lambda=0$.

\begin{corollary} \label{Hahnmamplitude}
 If $d \geq 3$ is odd, then $A(\lambda)$ is a holomorphic family of bounded operators in $\mathcal{B}(L^2,H^s)$ for any $s \in \R$. If $d \geq 2$ is even, then
 $A(\lambda)$ is a Hahn-holomorphic family of bounded operators in $\mathcal{B}(L^2,H^s)$ for any $s \in \R$
\end{corollary}
\begin{proof}
 The resolvent is (Hahn) meromorphic near zero as an operator from $L^2_\compp$ to $H^2_\loc$ (\cite{muller2014theory}). Differentiating the formula in $\omega$ this shows that $A_\lambda$
 is (Hahn)-meromorphic as an operator from $L^2(\sphere)$ to $C^k(\sphere)$. Since $A_\lambda$ is bounded as a map from $L^2 \to L^2$ the singular terms in the Hahn expansion must all vanish. Thus, $A_\lambda$  must be Hahn-holomorphic (see  \cite[Section 3]{muller2014theory}) with values in the operators from $L^2$ to $H^s$ for any $s \in \R$.
\end{proof}

\subsection{The spectral decomposition}
The generalised eigenfunctions $E_{\lambda}(\Phi)$ can be viewed as distributions in $\lambda$ with values in the space of Schwartz functions $\mathcal{S}(M;\Lambda^p T^*M)$.
In particular, if $g \in \CnI(\R_+)$ then $\int_\R g(\lambda) E_{\lambda}(\Phi) \der \lambda$ is square integrable. Therefore,
$\langle E_{\lambda}(\Phi),E_{\mu}(\Psi)\rangle_{L^2(M)}$ can be viewed as a bidistribution in $\mathcal{D}'(\R_+ \times \R_+)$.
This last inner product can be computed by taking the limit
$$
 \lim_{R \to \infty} \frac{1}{\lambda - \mu}\left( \langle  \Delta_p E_{\lambda}(\Phi), \chi_R  E_{\mu}(\Psi)\rangle_{L^2(M)} - \langle E_{\lambda}(\Phi), \chi_R \Delta_p E_{\mu}(\Psi)\rangle_{L^2(M)} \right),
$$
and using Green's identity. Here $\chi_R$ denotes the indicator function of a compact region whose boundary is in $M \setminus K$ and is identified with the sphere
of radius $R$. One obtains the distributional identity
\begin{align}
\langle E_{\lambda}(\Phi),E_{\mu}(\Psi)\rangle_{L^2(M)}=(4\pi\lambda)\delta(\lambda^2-\mu^2)\langle\Phi, \Psi\rangle_{L^2(\mathbb{S}^{d-1})}.
\end{align}
We have the following estimates on $\tilde j_\lambda(\Phi_\nu)$ and $E_\lambda$. Our main result in fact improves the estimate for $E_\lambda$ but the Lemma below serves as a first step in a bootstrapping process.
\begin{lemma} \label{expansionlemma}
 For any $K>0$ we have for $|\lambda|<K$ that
 $$\tilde j_\lambda(\Phi_\nu) = O_{C^\infty(\R^d \setminus \{0\})}\left(\frac{1}{\Gamma(\ell_\nu+\frac{d}{2})}\lambda^{\ell_\nu + \frac{d-1}{2}}\right),$$
 uniformly in $\nu$. Moreover, for $|\lambda|<K$ and $\Im{\lambda} \geq 0$, we have
 $$
  E_\lambda(\Phi_\nu) =O_{C^\infty(M)}\left(\frac{1}{\Gamma(\ell_\nu+\frac{d}{2})}\lambda^{\ell_\nu + \frac{d-5}{2}}\right).
 $$
\end{lemma}
\begin{proof}
  By the estimate
  $$
   | J_{\mu}\left(z\right)|\leq\frac{|\tfrac{1}{2}z|^{\mu}e^{|\Im z|}}{\Gamma\left.
(\mu+1\right)}, \quad \mu > -\frac{1}{2},
  $$
(see \cite[10.14.4]{olver2010nist}), the family $\lambda^{-\ell_\nu - \frac{d-1}{2}} \tilde j_\lambda(\Phi_\nu)$ is bounded in $L^2_\loc(\R^d \setminus \{0\})$.
Since $\Delta \tilde j_\lambda(\Phi_\nu) = \lambda^2 \tilde j_\lambda(\Phi_\nu)$ this shows that the family is bounded in $H^s_\loc(\R^d \setminus \{0\})$
for any $s \in 2 \mathbb{N}$. Hence, it is bounded in $C^\infty(\R^d \setminus \{0\})$. 
Now note that $R_\lambda = O(\lambda^{-2})$ if we consider $R_\lambda$ as a map from $H^s_\compp(M) \to H_\loc^{s+2}(M)$. This follows from the fact that
the resolvent is Hahn-meromorphic, has a singularity of order at most two at zero (see a more detailed analysis in Section \ref{expansionsection}), and is analytic near the real line and in the upper half plane.
Prop. \ref{besseljefunction} then implies the second estimate.
\end{proof}

\begin{lemma} \label{Abound}
 There exists a constant $R>0$ such that for $\lambda$ in any compact subset of $\mathcal{R}$ we have the bound
 $$
  | \langle A_\lambda \Phi_\nu, \Phi_\mu \rangle | \leq O\left(R^{\ell_\mu} \frac{\lambda^{\ell_\nu + \ell_\mu + d-4}}{\Gamma(\ell_\nu +\frac{d}{2}) \Gamma(\ell_\mu +\frac{d-2}{2})}\right).
 $$
\end{lemma}
\begin{proof}
 Lemma \ref{expansionlemma} gives that 
 \begin{align}\label{o}
  E_\lambda(\Phi_\nu) =O_{C^\infty(\R^d \setminus \{0\})}\left(\frac{1}{\Gamma(\ell_\nu+\frac{d}{2})}\lambda^{\ell_\nu + \frac{d-5}{2}}\right).
 \end{align}
 and recall from Prop. \ref{EoutofHankel} 
 \begin{gather*}
  E_{\lambda}(\Phi_{\nu})|_{M \setminus K} = \tilde{j}_{\lambda}(\Phi_{\nu}) + \tilde h^{(1)}_\lambda(A_\lambda \Phi_{\nu}).
 \end{gather*}
 The leading order terms in the expansion \eqref{o} are from the $\tilde{h}^{(1)}(A_{\lambda}\Phi_{\nu})$ terms because the $\tilde{j}_{\lambda}(\Phi_{\nu})$ terms are all of lower order, again from Lemma \ref{expansionlemma}. Therefore combining Lemma \ref{expansionlemma} and Prop. \ref{EoutofHankel}, one obtains. 
  $$
  | \langle A_\lambda \Phi_\nu, \Phi_\mu \rangle | | \lambda^{\frac{d-1}{2}} h^{(1)}_{d,\ell_\mu}(r \lambda)| = O \left(\frac{1}{\Gamma(\ell_\nu+\frac{d}{2})}\lambda^{\ell_\nu + \frac{d-5}{2}}\right),
  $$
where $r \gg 0$ is sufficiently large so that $\tilde K \subset B_r(0)$. The Lemma is then implied by the asymptotics \eqref{sphankelexpnu}.
\end{proof}

Furthermore, if $f\in C^\infty_0(M,\Lambda^p T^*M)$ then for fixed $\lambda>0$ the function $R_{-\lambda} f$ is incoming. By Lemma \ref{repoutgoing}
it can be expanded into a series of Hankel functions and has asymptotic behaviour
$$
 R_{-\lambda} f = -\frac{e^{-\rmi \lambda r} e^{\rmi \pi \frac{d-1}{4}}}{r^{\frac{d-1}{2}}} \Psi + O\left(\frac{1}{{r^{\frac{d+1}{2}}}} \right),
$$
as $r \to \infty$ for some $\Psi \in C^\infty(\sphere,\Lambda^p\R^d)$. 

We have
$$
(R_\lambda - R_{-\lambda}) f  = \frac{\ee^{-\rmi \lambda r} \ee^{\frac{i\pi(d-1)}{4}} }{r^{\frac{d-1}{2}}} \Psi +  \frac{\ee^{\rmi \lambda r} \ee^{-\frac{i\pi(d-1)}{4}}}{r^{\frac{d-1}{2}}} \tilde \Psi + O\left(\frac{1}{{r^{\frac{d+1}{2}}}} \right),  \quad \textrm{for} \quad r \to \infty
$$
for some $\tilde \Psi \in C^\infty(\sphere,\Lambda^p\R^d)$
and $(\Delta- \lambda^2) (R_\lambda - R_{-\lambda}) f=0$.
Since, by Eq. \eqref{superequation}, $E_\lambda(\Psi)$ and the solution $(R_\lambda - R_{-\lambda}) f$ have the same incoming leading asymptotic term they must coincide by Rellich's uniqueness theorem.
 Integration by parts gives
$$
 \langle f , E_{\overline{\lambda}}(\Phi_\nu) \rangle_{L^2(M)} = \langle (\Delta- \lambda^2) R_{-\lambda} f , E_{\overline{\lambda}}(\Phi_\nu) \rangle_{L^2(M)}= -2 \rmi \lambda \langle \Psi, \Phi_\nu \rangle_{L^2(\sphere)}.
$$
We can expand $\Psi$ into spherical harmonics
$\Psi =  \sum_{\nu}\langle \Psi, \Phi_\nu \rangle \Phi_\nu$, and
 the sum converges in $C^\infty(\sphere,\Lambda^p\R^d)$ as $\Psi$ is smooth. Thus,
$$
 \Psi = \frac{\rmi}{2\lambda} \sum_\nu  \langle f , E_{\overline{\lambda}}(\Phi_\nu) \rangle_{L^2(M)}   \Phi_\nu
$$
and we thus conclude
$$
 (R_\lambda - R_{-\lambda}) f = E_\lambda(\Psi)=\frac{\rmi}{2\lambda}\sum\limits_\nu E_{\lambda}(\Phi_\nu)\langle f, E_{\overline{\lambda}}(\Phi_\nu)\rangle.
$$
It follows immediately from the definition of $E_\lambda(\Phi)$ and the mapping properties of the resolvent that the map $C^\infty(\sphere,\Lambda^p\R^d) \to C^\infty(M,\Lambda^p T^*M), \Phi \mapsto E_\lambda(\Phi)$ is continuous. Therefore, the sum will converge in $C^\infty(M,\Lambda^p T^*M)$.
By Lemma \ref{expansionlemma} the convergence in $C^\infty(M,\Lambda^p T^*M)$ is also true for complex $\lambda$ uniformly in compact subsets of the complex plane for $\Im{\lambda} \geq 0$. Since both sides of the equation are (Hahn-)-meromorphic the equality continues to hold for
for $\Im{\lambda} > 0$.
We can now use Stone's theorem to compute the spectral measure and the complete spectral decomposition of $\Delta_{p,\mathrm{rel}}$.
This is essentially the analog of Stone's formula  (see for example \cite[4.20]{DZ}) in the black-box setting.

\begin{theorem}\label{StoneF}
If $f\in C^\infty_0(M,\Lambda^p T^*M)$ then for any $\lambda>0$
\begin{align} \label{diffres}
(R_\lambda-R_{-\lambda})f=\frac{\rmi}{2\lambda}\sum\limits_\nu E_{\lambda}(\Phi_\nu)\langle f, E_{\overline{\lambda}}(\Phi_\nu)\rangle,
\end{align} 
in $C^\infty(M)$.
For the spectral measure $dB_{\lambda}$ on the real line corresponding to the continuous spectrum we have for any $g,f\in C^\infty_0(M,\Lambda^p T^*M)$
\begin{align}
\langle dB_{\lambda}f,g\rangle=\frac{1}{2\pi} \chi_{[0,\infty)}(\lambda) \sum\limits_{\nu} \langle f, E_{\lambda}(\Phi_\nu)\rangle \langle  E_{\lambda}(\Phi_\nu) ,g\rangle \,\der \lambda,
\end{align}
so that for any bounded Borel function $h: \R \to \R$ we have
\begin{align}
\langle h(\Delta_{p,\mathrm{rel}}) f,g\rangle =h(0) \sum_{j=1}^N \langle f, u_j \rangle \langle u_j , g \rangle \nonumber \\+ \frac{1}{2\pi}  \sum\limits_{\nu}  \int_0^\infty h(\lambda^2) \langle f, E_{\lambda}(\Phi_\nu)\rangle \langle  E_{\lambda}(\Phi_\nu),g \rangle \,\der \lambda
\end{align}
where $u_{j}$'s are normalised eigenfunctions of $\Delta_{p,rel}$ with zero eigenvalue.
\end{theorem}

\begin{rem} \label{awesomeremark}
 The same arguments as before can be applied to the generalised eigenfunctions $E_{-\lambda}(\Phi)$ and as a result one also has for the spectral measure
 \begin{align}
\langle dB_{\lambda}f,g\rangle=\frac{1}{2\pi}  \chi_{[0,\infty)}(\lambda) \sum\limits_{\nu} \langle f, E_{-\lambda}(\Phi_\nu)\rangle \langle E_{-\lambda}(\Phi_\nu),g \rangle \,\der \lambda.
\end{align}
This could also be deduced more directly from the functional equations \eqref{funcE1} and \eqref{funcE2} and unitarity of the scattering matrix.
\end{rem}

\begin{theorem}\label{kernelconv}
 If $h$ is a Borel function with $h = O((1+\lambda^2)^{-q})$ for all $q \in \mathbb{N}$ we have that $h(\Delta_{p,\mathrm{rel}})$ has smooth integral kernel
 $k_h \in C^\infty(M \times M; \Lambda^p T^*M \boxtimes (\Lambda^p T^*M)^*)$ and 
 \begin{gather}
 k_h(x,y) = h(0) \sum_{j=1}^N u_j(x) \otimes (u_j(y))^* \nonumber \\+ \frac{1}{2\pi}  \sum\limits_{\nu}  \int_0^\infty h(\lambda^2) E_{\lambda}(\Phi_\nu)(x) \otimes E_{\lambda}(y)(\Phi_\nu)^*\,\der \lambda,
\end{gather}
where the sum converges in $C^\infty(M \times M; \Lambda^p T^*M \boxtimes (\Lambda^p T^*M)^*)$.
\end{theorem}
\begin{proof}
Note that by functional calculus $\Delta_{p,\mathrm{rel}}^{s_1}h(\Delta_{p,\mathrm{rel}}) \Delta_{p,\mathrm{rel}}^{s_2}$ is bounded as an operator in $L^2(M;\Lambda^p T^*M)$ for all $s_1,s_2 \in \R$.
Hence, $h(\Delta_{p,\mathrm{rel}})$ continuously maps $H^s_\compp(M,\Lambda^p T^*M)$ to $H^{s+q}_\loc(M,\Lambda^p T^*M)$ for all $s \in \R$ and $q \in \R$ and therefore has smooth integral  kernel $k_h$ in $C^\infty(M \times M; \Lambda^p T^*M \boxtimes (\Lambda^p T^*M)^*)$. 
Denote by $h_n(\Delta_{p,\mathrm{rel}})$ the approximation of $h(\Delta_{p,\mathrm{rel}})$ defined by truncating the infinite sum, i.e.
\begin{align}
\langle h_n(\Delta_{p,\mathrm{rel}}) f,g\rangle =h(0) \sum_{j=1}^N \langle f, u_j \rangle \langle u_j , g \rangle \nonumber \\+ \frac{1}{2\pi}  \sum\limits_{\nu, \ell_\nu \leq n}  \int_0^\infty h(\lambda^2) \langle f, E_{\lambda}(\Phi_\nu)\rangle \langle  E_{\lambda}(\Phi_\nu),g \rangle \,\der\lambda.
\end{align}
To show the statement it is sufficient to show it for $h \geq 0$ since the general case can be deduced by decomposing $h$ into positive and negative parts.
In this case $$0 \leq \Delta_{p,\mathrm{rel}}^{s_1}h_n(\Delta_{p,\mathrm{rel}}) \Delta_{p,\mathrm{rel}}^{s_2} \leq \Delta_{p,\mathrm{rel}}^{s_1}h(\Delta_{p,\mathrm{rel}}) \Delta_{p,\mathrm{rel}}^{s_2}$$ as operators in $L^2(M)$.
For any $\chi_1,\chi_2 \in C^\infty_0(M)$ we then obtain the estimate
\begin{gather*}
 | \langle  h_n(\Delta_{p,\mathrm{rel}}^{\frac{1}{2}}) (\chi_1 v),  \chi_2 w \rangle | \leq  | \langle  h(\Delta_{p,\mathrm{rel}}^{\frac{1}{2}}) (\chi_1 v),  \chi_1 v \rangle |^{\frac{1}{2}} | \langle  h(\Delta_{p,\mathrm{rel}}^{\frac{1}{2}}) (\chi_2 w),  \chi_2 w \rangle |^{\frac{1}{2}} \\  \leq C_s \| \chi_1 v \|_{H^{-s}}  \| \chi_2 w \|_{H^{-s}}.
\end{gather*}
Hence, $h_n(\Delta_{p,\mathrm{rel}}^{\frac{1}{2}})$ has smooth integral kernel $k_{h_n}$ and the sequence $k_{h_n}$ is bounded in $C^\infty(M \times M; \Lambda^p T^*M \boxtimes (\Lambda^p T^*M)^*)$.
By Theorem \ref{StoneF} the sequence of $k_{h_n}$ converges weakly to $k_h$ as $n \to \infty$. Since the sequence $k_{h_n}$ is also bounded in the space $C^\infty(M \times M; \Lambda^p T^*M \boxtimes (\Lambda^p T^*M)^*)$
the Theorem of Arzela-Ascoli implies that it converges in $C^\infty(M \times M; \Lambda^p T^*M \boxtimes (\Lambda^p T^*M)^*)$.
\end{proof}

\section{Expansions near zero} \label{expansionsection}

The generalised eigenfunctions $E_\lambda$ are related via Prop. \ref{besseljefunction} to the resolvent. In this section we use the singularity structure of the resolvent near zero to analyse the behaviour of $E_\lambda$ for small $\lambda$. If $\Phi \in \mathcal{H}^p_\ell(\sphere)$ is a vector-valued spherical harmonic of degree $\ell$ one has for $|\lambda| < 1$:
\begin{gather} \label{besseltildexp}
 \tilde{j}_{\lambda}(\Phi)(r \theta) = C_{d,\ell} \lambda^{\ell + \frac{d-1}{2}}  r^\ell \Phi(\theta)+ O_{C^\infty(\R^d \setminus \{0\})}(\lambda^{\ell + \frac{d+3}{2}}),
\end{gather}
where
$$
C_{d,\ell} = (-\rmi)^{\ell}\sqrt{2 \pi} \frac{1}{2^{\ell + \frac{d}{2}-1}} \frac{1}{\Gamma(\ell + \frac{d}{2})} .$$
\begin{lemma} \label{superlemma}
 Assume $d\geq 3$ and suppose that $u \in C^\infty(M;\Lambda^p T^*M)$ satisfies $\Delta_p  u =0$ so that on $M \setminus K$ we have the multipole expansion
 $$
  u(r \theta) = \sum_\nu \left( a_\nu \frac{1}{r^{d-2+\ell_\nu}} \Phi_\nu(\theta) + b_\nu r^{\ell_\nu} \Phi_\nu(\theta) \right).
 $$
For fixed $\Phi_\nu$ we have for $|\lambda| <1$,
 $$
  \langle (\Delta_p  - \lambda^2) (\chi\tilde{j}_{\lambda}(\Phi_\nu)),u \rangle_{L^2(M)} = -(d-2+2\ell_\nu) C_{d,\ell_\nu} \lambda^{\ell_\nu +\frac{d-1}{2}} \overline{a_\nu} + O( \lambda^{\ell_\nu +\frac{d+3}{2}} ).
 $$
\end{lemma}
\begin{proof}
 Note that $(\Delta_p  - \lambda^2) (\chi\tilde{j}_{\lambda}(\Phi_\nu))$ is compactly supported.
 Let $M_R$ be the complement of the region $\{ (r \theta) \in M \setminus K \mid r \geq R \}$ in $M$ and denote by $\chi_{M_R}$ its indicator function.
 The boundary of $M_R$ is the set $r=R$.
 For sufficiently large $R>R_1$ so that $\chi=1$ in a neighbourhood of the complement of $M_{R_1}$,  we have
 $$
    \langle  (\Delta_p  - \lambda^2) (\chi\tilde{j}_{\lambda}(\Phi_\nu)) ,u\rangle_{L^2(M)} =\langle  \chi_{M_R} (\Delta_p  - \lambda^2) (\chi\tilde{j}_{\lambda}(\Phi_\nu)) ,u\rangle_{L^2(M)}.
 $$
Integration by parts gives
 \begin{gather*}
   \langle (\Delta_p  - \lambda^2) (\chi\tilde{j}_{\lambda}(\Phi_\nu)),u \rangle_{L^2(M)} = - \lambda^2  \langle  \chi_{M_R} (\chi\tilde{j}_{\lambda}(\Phi_\nu)) ,u\rangle_{L^2(M)} \\-
   \int_{\partial M_R} \overline{u(x)} \partial_{n}( (\chi\tilde{j}_{\lambda}(\Phi_\nu))(x) \der \sigma(x) +  \int_{\partial M_R} \overline{\partial_{n}(u)(x)}  (\chi\tilde{j}_{\lambda}(\Phi_\nu))(x) \der \sigma(x),
 \end{gather*}
 where $d\sigma$ is the surface measure on $\partial M_R$. Using \eqref{besseltildexp} and the fact that $\chi\cdot \chi_{M_R}$ is supported in a compact annulus one has
 \begin{gather*}
   \langle (\Delta_p  - \lambda^2) (\chi\tilde{j}_{\lambda}(\Phi_\nu)) ,u\rangle_{L^2(M)} = C_{d,\ell_\nu} \\ \cdot \left( \int_{\sphere} \overline{(\partial_r u)(R \theta)} \lambda^{\ell_\nu + \frac{d-1}{2}} R^{\ell_\nu} R^{d-1} \Phi_\nu(\theta) \der \theta  - \int_{\sphere} \overline{u(R \theta)} \ell_\nu \lambda^{\ell_\nu + \frac{d-1}{2}} R^{\ell_\nu-1} R^{d-1} \Phi_\nu(\theta) \der \theta \right)\\
   + O( \lambda^{\ell_\nu +\frac{d+3}{2}} ).
 \end{gather*}
 We can now use the  multipole expansion  and orthonormality of $(\Phi_\nu)$ to obtain the claimed formula. This formula holds for any $R>R_1$
 and the right hand side is actually constant in $R$. 
\end{proof}

The same proof with obvious modifications in dimension two gives  the following.

\begin{lemma} \label{superlemmad2}
 Assume $d = 2$ and suppose that $u \in C^\infty(M;\Lambda^p T^*M)$ satisfies $\Delta_p  u =0$ so that we have the multipole expansion
 $$
  u(r \theta) =\sum_{\nu,\ell_\nu=0} \left( a_\nu \log(r)  \Phi_\nu(\theta)+  b_\nu  \Phi_\nu(\theta) \right) +  \sum_{\nu,\ell_\nu>0} \left( a_\nu \frac{1}{r^{d-2+\ell_\nu}} \Phi_\nu(\theta) + b_\nu r^{\ell_\nu} \Phi_\nu(\theta) \right).
 $$
For fixed $\Phi_\nu$ we have for $|\lambda| <1$ and $\ell_\nu \not = 0$:
 $$
  \langle (\Delta_p  - \lambda^2) (\chi\tilde{j}_{\lambda}(\Phi_\nu),u \rangle_{L^2(M)} = -(d-2+2\ell_\nu) C_{d,\ell_\nu} \lambda^{\ell_\nu +\frac{d-1}{2}} \overline{a_\nu} + O( \lambda^{\ell_\nu +\frac{d+3}{2}} ).
 $$
 In case $\ell_\nu = 0$ we get for $|\lambda| <1$:
 $$
  \langle (\Delta_p  - \lambda^2) (\chi\tilde{j}_{\lambda}(\Phi_\nu),u \rangle_{L^2(M)} = C_{d,0} \lambda^{\frac{d-1}{2}} \overline{a_\nu} + O( \lambda^{\frac{d+3}{2}} ).
 $$

\end{lemma}

\begin{lemma}
 Suppose that $d>2$, $u \in C^\infty(M,\Lambda^p T^*M)$ is closed and co-closed and has a multipole expansion of the form
 $$
  u(r,\theta) = \sum_{\nu,\ell_\nu\geq 0} a_{\nu} \frac{1}{r^{\ell_\nu + d-2}} \Phi_\nu(\theta)
 $$
 for sufficiently large values of $r$.
Then we can conclude $a_\nu=0$ whenever $\ell_\nu=0$.  
\end{lemma}
\begin{proof}
 We write the multipole expansion of $u$ in a slightly different way as
 $$
  u(r,\theta) =  \frac{1}{r^{d-2}} \omega + \sum_{\nu,\ell_\nu>0} a_{\nu,k} \frac{1}{r^{\ell_\nu + d-2}} \phi_\nu(\theta) e_k,
 $$
 where $(e_k)$ is the standard basis in $\Lambda^p \R^d$, $(\phi_\nu)$ a basis of spherical harmonics in $L^2(\sphere)$, and $\omega=\sum_{k} a_k e_k$ is a constant differential form. It follows that 
 \begin{gather*}
  \der u(r,\theta) =  \frac{2-d}{r^{d-1}} \der r \wedge \omega + O\left(\frac{1}{r^{d}}\right)
 \end{gather*}
for sufficiently large $r$, where we have used that $\der  \phi_\nu$ is of order $O(\frac{1}{r})$ since the inner product on one forms is given by the inverse metric
 $g^{-1} = \partial_r \otimes \partial_r + r^{-2} g^{-1}_\sphere$ on the cotangent bundle. Therefore, $\der r \wedge \omega =0$. Because $u$ is also co-closed
 the same computation applied to $* u$ which gives $\der r \wedge * \omega=0$. This implies that Clifford multiplication of $\omega$ by $\der r$ yields zero. Since Clifford multiplication by a non-zero covector is invertible, this implies $\omega=0$.
 \end{proof}

If $u$ is harmonic and satisfies relative boundary conditions with the above multipole expansion then we can integrate by parts and obtain the following.

\begin{corollary} \label{supercor}
 If $d>2$ and $u \in \ker_{C^\infty}(\Delta_p)$ satisfies relative boundary conditions and has a multipole expansion of the form
 $$
  u(r,\theta) = \sum_{\nu,\ell_\nu\geq 0} a_{\nu} \frac{1}{r^{\ell_\nu + d-2}} \Phi_\nu(\theta)
 $$
 when $r$ is sufficiently large. Then
 we can conclude $a_\nu=0$ whenever $\ell_\nu=0$.
 \end{corollary} 

Assume $u=u_j$ is in $\ker_{L^2}(\Delta_{p,rel})$.
 Since $\Delta_{p,rel}$ is the square of the self-adjoint operator $Q_{rel}$ this implies that $Q_{rel} u_j=0$ and therefore $u_j$ must be closed and co-closed. This implies the following corollary.
\begin{corollary} \label{supercor2}
 We have $a_j(\Phi_\nu)=0$ when $\ell_\nu=0$ and hence $P^{(0)}=0$.
\end{corollary}

Since generalised eigenfunction are Hahn-holomorphic they have Hahn-series expansions whose first terms are harmonic and satisfy relative boundary conditions.
The following Lemma  clarifies how the multipole expansions of these harmonic forms appear from Prop. \ref{EoutofHankel}.
\begin{lemma} \label{superlemma2}
 Let $0<R_1<R_2$ and $[R_1,R_2]$ be a fixed interval.
 Suppose that $\Psi_\lambda$ is a (Hahn)-holomorphic family of spherical harmonics of degree $\ell$ such that 
 $\tilde h^{(1)}_\lambda(\Psi_\lambda)(r,\theta)=O(\lambda^m)$ 
  as $\lambda \to 0$ uniformly in $(r,\theta)$ for $r \in [R_1,R_2], \theta \in \sphere$. Assume $\ell + \frac{d-2}{2}>0$,
  then $\Psi_\lambda = O(\lambda^{\ell + \frac{d-3}{2}+m})$ and
  $$
   \lim_{\lambda \to 0} \lambda^{-m} \tilde h^{(1)}_\lambda(\Psi_\lambda)(r,\theta)= \frac{\Phi(\theta)}{r^{\ell + d-2}},
  $$
  where 
  $$\Phi =   - \rmi \frac{1}{\sqrt{\pi}} 2^{\ell+\frac{d-3}{2}} \Gamma(\ell +\frac{d-2}{2}) \left(\lim\limits_{\lambda \to 0} \lambda^{-\ell - \frac{d-3}{2}-m} \Psi_\lambda \right).$$
  \end{lemma}
 \begin{proof}
  This follows from the asymptotic behavior of the Hankel function \eqref{sphankelexpzero}, which is in Appendix \ref{Abf}. Namely, as $\lambda \to 0$ we have
  $$
   \tilde h^{(1)}_\lambda(\Psi_\lambda)(r,\theta) = - \rmi \frac{1}{\sqrt{\pi}} 2^{\ell+\frac{d-3}{2}} \Gamma(\ell +\frac{d-2}{2}) \lambda^{\frac{3-d}{2}-\ell} r^{-\ell-d +2} \Psi_\lambda(\theta) + O( \lambda^{\frac{7-d}{2}-\ell})
  $$
 \end{proof} 

In $\R^d$ the operator $(\Delta_p - \lambda^2)$ has integral kernel 
$$
 r_\lambda(x,y) = \frac{\rmi}{4} \left( \frac{\lambda}{2 \pi |x-y|} \right)^{\frac{d-2}{2}} \mathrm{H}^{(1)}_{\frac{d-2}{2}}(\lambda |x-y|)) \mathbf{1}
$$
if the bundle of differential forms has been trivialised with respect to the standard basis in $\R^d$ and $\mathbf{1}$ denotes the identity matrix.

\subsection{Analysis when $d$ is odd } \label{doddanal}
In this case it follows from the explicit formula that the free resolvent kernel is meromorphic with a simple pole at $0$ if $d=1$, 
and is entire in case $d >1$.
By general arguments using a gluing construction one concludes that on $M$ we have that 
$R_\lambda = (\Delta_p  - \lambda^2)^{-1}$ is meromorphic also near zero with finite rank negative Laurent coefficients. 
Then Lemma \ref{aoutofresolvent}
implies that also $A_\lambda$ is meromorphic near zero. Since $S_\lambda$ is unitary this implies that $A_\lambda$ is regular at zero.
By general resolvent bounds for self-adjoint operators,
$R_\lambda$ can have a pole of order at most two at zero. Hence, in odd dimensions the resolvent (as an operator
from $L^2_{comp}$ to $H^2_{loc}$) has an expansion  for $| \lambda |$ small of the form
\begin{gather} \label{grppaasdasd}
 R_\lambda = - \frac{B_{-2}}{\lambda^2} + \rmi \frac{B_{-1}}{\lambda} + B(\lambda),
\end{gather}
where $B(\lambda)$ is holomorphic near zero and $B_{-2},B_{-1}: L^2_{comp} \to H^2_{loc}$ are of finite rank. By Stone's formula and the spectral decomposition 
$B_{-2}$ is the orthogonal projection onto $\ker_{L^2} (\Delta_{p,rel})$, i.e. $B_{-2}=P$.
By Prop. \ref{besseljefunction} we have
\begin{align} \label{gefunctionrep}
E_{\lambda}(\Phi)=\chi\tilde{j}_{\lambda}(\Phi)-R_{\lambda}(\Delta_p  -\lambda^2)(\chi\tilde{j}_{\lambda}(\Phi)),
\end{align}
and we can therefore obtain the Laurent series of $E$ about $\lambda=0$ by expanding the Bessel functions and using the resolvent expansion.
We have $E_0(\Phi)=0$ if  $\ell+\frac{d-5}{2}>0$. 
By Lemma \ref{expansionlemma}, using that $\Delta_p E_\lambda(\Phi_\nu) = \lambda^2 E_\lambda(\Phi_\nu)$, one has for $| \lambda |$ small,
$$
 E_\lambda(\Phi_\nu) = O_{C^\infty(M)}\left(\frac{1}{\Gamma(\ell_\nu+\frac{d}{2})}\lambda^{\ell_\nu + \frac{d-5}{2}}\right),
$$
uniformly in $\nu$.
Therefore, comparing the resolvent expansion with \eqref{diffres}, we obtain that the functions $E_\lambda(\Phi)$ are regular at zero and 
\begin{align}\label{belowfive}
 B_{-1} = \frac{1}{4}\sum\limits_{\ell_\nu \leq \frac{5-d}{2}} E_{0}(\Phi_\nu) \langle \cdot, E_{0}(\Phi_\nu)\rangle.
\end{align}
In particular, $B_{-1}$ is symmetric. In case $d>5$ we conclude that $B_{-1}=0$.
In order to compute $B_{-1}$ we would like to treat the cases $d=3$ and $d=5$ separately.

\subsubsection{Resolvent expansion and generalised eigenforms in dimension three }

By Prop. \ref{EoutofHankel} we have for fixed large $r$ that
 $$
  E_\lambda(\Phi)(r,\theta) = \tilde{j}_{\lambda}(\Phi)(r,\theta) + \tilde h^{(1)}_\lambda(A_\lambda \Phi)(r,\theta).
 $$
 Since $E_\lambda(\Phi)$ is regular at zero so must be $\tilde h^{(1)}_\lambda(A_\lambda \Phi)(r,\theta)$ and 
 $$
  E_0(\Phi)(r,\theta) = \lim_{\lambda \to 0} \tilde h^{(1)}_\lambda(A_\lambda \Phi)(r,\theta).
 $$
Consider $|\lambda r| \ll 1$, if $\Phi \in \mathcal{H}^p_\ell(\sphere)$ is a vector-valued spherical harmonic of degree $\ell$ one has uniform asymptotics for the Hankel function in $(r,\theta)$ for $r \in [R_1,R_2], \theta \in \sphere$, in powers of  $\lambda$ (given in Appendix). Combining this with the fact $A_\lambda$ is holomorphic and taking the limit $\lambda\rightarrow 0$, one sees that $E_0(\Phi)$ has a multipole expansion of the form
$$
  E_0(\Phi) = \sum_{\nu} e_{\nu}(\Phi) \frac{1}{r^{\ell_{\nu}+1}} \Phi_\nu.
$$
By construction $E_0(\Phi)$ is harmonic and satisfies relative boundary conditions. It follows from Corollary \ref{supercor} that $E_0(\Phi)$ is closed and co-closed and that $e_\nu(\Phi) =0$ whenever $\ell_\nu=0$.

If $\Phi$ is a spherical harmonic of degree $\ell=0$ then if $|\lambda|<1$, one has
$$
 \tilde j_\lambda(\Phi)(r, \theta) = 2 \lambda \Phi(\theta) + O_{C^\infty(\R^d\setminus \{0\})}(\lambda^3).$$
Using Lemma \ref{superlemma} we obtain from \eqref{grppaasdasd} and \eqref{gefunctionrep} that
$$
 E_0(\Phi) = \lim_{\lambda \to 0} -\rmi \frac{1}{\lambda} B_{-1} (\Delta_p  -\lambda^2) \chi   \tilde j_\lambda(\Phi) =  \frac{\rmi}{2} \sum_{\nu,\ell_\nu=0} E_0(\Phi_\nu) \overline{e_\nu(\Phi)}=0.
$$ 
If $\Phi$ is a spherical harmonic of degree $\ell=1$ then similarly for  $| \lambda |$ small
$$ 
 \tilde j_\lambda(\Phi)(r, \theta) = -\frac{2 \rmi}{3} \lambda^2 r \Phi(\theta) + O_{C^\infty(\R^d \setminus \{0\})}(\lambda^4)
$$
and therefore, using Lemma \ref{superlemma}, one gets
$$
 E_0(\Phi) =  \lim_{\lambda \to 0}\lambda^{-2 }P \left( (\Delta_p -\lambda^2) \chi  \tilde j_\lambda(\Phi) \right) = 2 \rmi \sum_{j=1}^N  a_{j}(\Phi) u_j.
$$ 
In particular it follows that in this case $E_0(\Phi) \in L^2(M)$. If $\Phi$ is a spherical harmonic of degree higher than $1$ then by the same reasoning one gets $E_0(\Phi)=0$. We have therefore proved the following proposition.
\begin{proposition} \label{prop37}
 If $d=3$ then
 \begin{itemize}
  \item $E_0(\Phi_\nu)=0$ if $\ell_\nu \not= 1$,
  \item  $E_0(\Phi_\nu)= 2 \rmi \sum_{j=1}^N   a_{j}(\Phi_\nu) u_j \in L^2$ if $\ell_\nu=1$.
 \end{itemize}
 Moreover, we have
 $$
  B_{-1} =  P^{(1)}.
 $$
\end{proposition}

\subsubsection{Resolvent expansion and generalised eigenforms in dimension five}

In the case $\Phi \in \mathcal{H}^p_\ell(\sphere)$ is a spherical harmonic of degree $\ell=0$ then for $|\lambda|<1$ we have
$$
  \tilde j_\lambda(\Phi)(r, \theta) = \frac{2}{3} \lambda^2 \Phi(\theta) + O_{C^\infty(\R^d \setminus \{0\})}(\lambda^4)
$$
and therefore,
$$
 E_0(\Phi) = \lim_{\lambda \to 0}\lambda^{-2 }P \left( (\Delta_p -\lambda^2)  \tilde j_\lambda(\Phi) \right) = -2 \sum_{j=1}^N   a_{j}(\Phi) u_j.
$$
This vanishes, by Corollary \ref{supercor2}.
In the case $\Phi$ is a spherical harmonic of degree higher than $0$ we obtain $E_0(\Phi)=0$. 
Hence, we have the following 
\begin{proposition} \label{prop38}
 If $d=5$ then $E_0(\Phi)=0$ and hence $B_{-1}=0$.
\end{proposition}

\subsubsection{Expansion of $E_\lambda(\Phi)$ in odd dimensions}
Assume that $\Phi$ is a spherical harmonic of degree $\ell$, then for $|\lambda|<1$ we have
$$
 \tilde{j}_{\lambda}(\Phi)(r \theta) = C_{d,\ell} \lambda^{\ell + \frac{d-1}{2}} r^\ell \Phi(\theta) + O_{C^\infty(\R^d \setminus \{0\})}(\lambda^{\ell + \frac{d+3}{2}}).
$$
Therefore, using Lemma \ref{superlemma}, we get in dimensions $d \geq 5$, using $B_{-1}=0$,
\begin{gather} \label{expd5th1}
 E_\lambda(\Phi) = -(d-2+2\ell) C_{d,\ell} \lambda^{\ell +\frac{d-5}{2}} \sum_{j=1}^N a_j(\Phi) u_j + O_{C^\infty(M)}( \lambda^{\ell +\frac{d-1}{2}} ).
\end{gather}
In case $d=3$ we have obtain from Lemma \ref{superlemma} that
\begin{gather}
 E_\lambda(\Phi) = -(d-2+2\ell) C_{d,\ell} \lambda^{\ell +\frac{d-5}{2}} \sum_{j=1}^N a_j(\Phi) u_j \nonumber
 \\+ \rmi (d-2+2\ell) C_{d,\ell} \lambda^{\ell +\frac{d-3}{2}} \sum_{j,k=1}^N 
 a_{kj}^{(1)}a_j(\Phi) u_k + O_{C^\infty(M)}( \lambda^{\ell +\frac{d-1}{2}} ).  \label{expd3th1}
\end{gather}
The mixed terms in dimension $d=3$ and $\ell=1$ cancel out giving only even powers, which is consistent with the right hand side being an even function. 
In case $\ell=0$ we have $a_j(\Phi)=0$ by Corollary \ref{supercor2}. Hence, in that case $E_\lambda(\Phi) = O_{C^\infty(M)}( \lambda^{\frac{d-1}{2}} )$.
Therefore, for any $f \in C^\infty_0(M;\Lambda^pT^*M)$ we get
$$
\sum\limits_\nu E_{\lambda}(\Phi_\nu)\langle f, E_{\overline{\lambda}}(\Phi_\nu)\rangle = d^2 | C_{d,1}|^2 \lambda^{d-3} P^{(1)} f + O_{C^\infty(M)}(\lambda^{d-1}).
$$
Recall the definition of $P^{(\ell)}$ in \eqref{Pl}. 

\begin{theorem}
 Suppose that $d$ is odd and $d \geq 3$. Then for any $f \in C^\infty_0(M;\Lambda^pT^*M)$ and  for small $| \lambda |$ we have
 $$
-2 \rmi \lambda (R_\lambda - R_{-\lambda}) f=\sum\limits_\nu E_{\lambda}(\Phi_\nu)\langle f, E_{\overline{\lambda}}(\Phi_\nu)\rangle =d^2 | C_{d,1}|^2 \lambda^{d-3} P^{(1)} f + O_{C^\infty(M)}(\lambda^{d-1}).
$$
\end{theorem}

\subsection{Analysis when $d$ is even} \label{devenanal}
If the dimension $d$ is even and $d>0$ the free resolvent $R_{0,\lambda}$ takes the form
$$
 R_{0,\lambda} = U_\lambda + V_\lambda \log \lambda, 
$$
where $U_\lambda$ and $V_\lambda$ are holomorphic and even. There is a suitable function space allowing for expansions with $\log$-terms and we will make use of this space of Hahn meromorphic and Hahn holomorphic functions. We refer to Section \ref{hahnapp} for details of this. In our case it
follows that $R_{0,\lambda}$ is Hahn-meromorphic with respect to the group
$2 \mathbb{Z} \times \mathbb{Z}$, in particular only even powers of $\lambda$ appear in the expansions. More precisely, in dimension $d=2$ there is a singularity with a finite rank negative expansion coefficient. In even dimensions
$d \geq 4$ the free resolvent is Hahn holomorphic.
By the general gluing construction and the Hahn-meromorphic Fredholm theorem  (\cite[Theorem 4.1]{muller2014theory}) this implies that the resolvent
$R_\lambda = (\Delta_p - \lambda^2)^{-1}$ on $M$ is Hahn-meromorphic near zero and the joint span of the ranges of all negative expansion coefficients is finite dimensional. The most general expansion that still satisfies the resolvent bounds for self-adjoint operators is then
\begin{gather} \label{genexpansion}
 R_\lambda = - \frac{1}{\lambda^2} \sum_{k=0}^\infty B_{-2,k} (-\log(\lambda))^{-k} + \sum_{k=1}^L B_{-k} (-\log(\lambda))^{k} + B(\lambda)
\end{gather}
where $B(\lambda)$ is Hahn holomorphic. 
In particular, $\lambda^2 R_\lambda$ is bounded uniformly in $\lambda$ for $|\lambda|<1$ and bounded $|\arg \lambda|$ as a map from $L^2_\compp$ to $H^2_\loc$. As in the odd-dimensional case we can use Lemma \ref{expansionlemma} and $\Delta_p E_\lambda(\Phi_\nu) = \lambda^2 E_\lambda(\Phi_\nu)$ to conclude that for $| \lambda |<1$,
$$
 E_\lambda(\Phi_\nu) = O_{C^\infty(M)}\left(\frac{1}{\Gamma(\ell_\nu+\frac{d}{2})}\lambda^{\ell_\nu + \frac{d-5}{2}}\right),
$$
uniformly in $\nu$ in any fixed sector of the logarithmic cover. Therefore, as before we can compare expansion coefficients of the corresponding Hahn-series in equation \eqref{diffres}.

It follows from Stone's formula that $B_{-2,0}=P$ is the spectral projection onto the zero eigenspace.
In fact, it follows from the relation between resolvent and spectral measure that non-zero coefficients $B_{-2,k}$ for $k>0$ can only occur in the presence of a non-zero leading order term $B_{-2,1}$ and in dimension lower than $6$.

\begin{lemma} \label{abovelemma}
 If $d>4$ then $B_{-2,k}=0$ for any $k>0$.
 If $B_{-2,1}=0$ and $d = 2$ or $d=4$ then $B_{-2,k}=0$ for any $k>0$.
\end{lemma}
\begin{proof}
 Let $k \geq 1$.
 We will show that $B_{-2,k} =0$ is implied whenever either $d>4$ or when  $B_{-2,1},\ldots,B_{-2,k-1}=0$. First note that
 $$
   R_{\lambda}- R_{-\lambda} = (-\rmi \pi)^{k-1} B_{-2,k} \frac{1}{\lambda^2(-\log \lambda)^{k+1}} + O_{\mathcal{B}(L^2_\compp \to H^{2}_\loc)}\left (\frac{1}{\lambda^2(-\log \lambda)^{k+2}} \right).
 $$ 
 Suppose, by contradiction, that $B_{-2,k} \not=0$. Then, by Theorem \ref{StoneF}, \eqref{diffres}, the expansion of $E_\lambda(\Phi_\nu)$ for some $\nu$
 must have a non-zero top-order term of the form
 $$
  f \frac{1}{\lambda^{1/2}(-\log \lambda)^{\frac{k+1}{2}}} 
 $$
  and some non-zero function $f$.
From Prop. \ref{besseljefunction} we have the following a-priori estimate
 $$
  E_{\lambda}(\Phi) = O_{C^\infty(M)}\left(\lambda^{\frac{d-5}{2}} \frac{1}{(-\log \lambda)^k}\right).
 $$
 If $d>4$ then $E_{\lambda}(\Phi) = O(\lambda^{\frac{1}{2}})$, thus $f=0$ and therefore $B_{-2,k}=0$.
 If $d=4$ and $B_{-2,1}=0$ then we can assume $k\geq 2$ and hence $k>\frac{k+1}{2}$. This would imply once more $f=0$ and therefore $B_{-2,k}=0$.
\end{proof}

\begin{rem}
It is known at least since Murata's work \cite{murata1982asymptotic} that generalised
projections onto the resonant states can occur in the form
 $-B_{-2,1} \frac{1}{\lambda^2 (-\log \lambda)}$ in the case of potential scattering in dimension four. In our case such zero resonance states do not exist in dimension higher than two. Therefore, 
we obtain a much more refined result below. 
 \end{rem}

\begin{theorem} \label{resolventexpdg4}
 Suppose that $d \geq 4$. Then $R_\lambda$ (as an operator
from $L^2_{comp}$ to $H^2_{loc}$) is Hahn-meromorphic at $\lambda=0$ with expansion of the form
 $$
  -\frac{P}{\lambda^2} + B_{-1} (-\log \lambda) + B(\lambda),
 $$
 for $| \lambda |$ small and in a fixed sector $| \arg \lambda | \leq \Theta$,
 where $B(\lambda)$ is Hahn-holomorphic and $P,B_{-1}: L^2_{comp} \to H^2_{loc}$ are of finite rank. If $d>4$ then $B_{-1}=0$.
  If $d=4$ then $B_{-1}=\frac{1}{4} P^{(1)}$.
\end{theorem}
\begin{proof}
 We first show that $B_{-2,k}=0$ for any $k>0$. This is the case in any dimension greater $4$, so we only need to check that case $d=4$. We only need to show that $B_{-2,1}=0$. 
 Suppose by contradiction that $B_{-2,1} \not= 0$. By the same argument as in Lemma \ref{abovelemma} there must exist a $\nu$ such that
 $E_\lambda(\Phi_\nu)$ has non-zero top-order term of the form
 \begin{gather} \label{sillyterm}
  f_\nu \frac{1}{\lambda^{1/2}(-\log \lambda)}. 
 \end{gather}
 The coefficient $f_\nu$ is harmonic and satisfies relative boundary conditions. 
 By unitarity of the scattering matrix $S_\lambda$ is bounded near zero and therefore, $A_\lambda$ is Hahn-holomorphic.
 By Prop. \ref{EoutofHankel} the term \eqref{sillyterm} must appear in the expansion of $\tilde h^{(1)}_\lambda(A_\lambda \Phi_\nu)$.
 Inspection of the expansion of the Hankel function in the regime $|r\lambda| \ll 1$, $r\in [R_1,R_2]$ and $\lambda\rightarrow 0$ shows that $f_\nu$ has a multipole expansion of the form 
 $$
  f_\nu(r,\theta) = \sum_{\mu} a_{\mu} \frac{1}{r^{\ell_\mu + d-2}} \Phi_\mu. 
 $$
 
 Therefore, by Corollary \ref{supercor}, $a_\mu=0$ when $\ell_\mu=0$. On the other hand $B_{-2,1}$ is of the form $B_{-2,1} = \sum\limits_{\nu, \ell_\nu=0} f_\nu \langle \cdot, f_\nu \rangle$. We can now use Lemma \ref{superlemma} to see that
 \begin{gather}
  -\frac{1}{\lambda^2} B_{-2,0} (\Delta_p - \lambda^2) \chi \tilde j_{\lambda} (\Phi) = O_{C^\infty(M)}(\lambda^{1+\frac{d-5}{2}}),\\
  -\frac{1}{\lambda^2 (-\log\lambda)} B_{-2,1} (\Delta_p - \lambda^2) \chi \tilde j_{\lambda} (\Phi) = O_{C^\infty(M)}(\lambda^{1+\frac{d-5}{2}}).
 \end{gather}
Using Prop. \ref{besseljefunction} we obtain $E_\lambda(\Phi) = O_{C^\infty(M)}(\lambda^{\frac{1}{2}})$. We conclude that $f_\nu=0$ and therefore $B_{-2,1}=0$.

Now, using again \ref{StoneF}, \eqref{diffres} the bound $E_\lambda(\Phi) = O_{C^\infty(M)}(\lambda^{\frac{1}{2}})$ implies that $R_\lambda - R_{-\lambda} = O_{C^\infty(M)}(1)$
and therefore $B_{-k}=0$ whenever $k>1$. If $d>4$ then $E_\lambda(\Phi) = O_{C^\infty(M)}(\lambda^{\frac{3}{2}})$ and hence $B_{-1}=0$.\\
It remains to compute $B_{-1}$ in dimension $4$.
 Since $-\log \lambda + \log(-\lambda) = \rmi \pi$ we have
 $$
  R_\lambda - R_{-\lambda} =  \rmi \pi B_{-1}.
 $$
 Comparing with $\eqref{diffres}$ shows that
 $$
  B_{-1}= \frac{1}{2 \pi}\sum_{\nu, \ell_\nu=1} g_\nu \langle \cdot, g_\nu \rangle,
 $$
 where $g_\nu = \lim\limits_{\lambda \to 0} \lambda^{-\frac{1}{2}} E_\lambda(\Phi_\nu)$. This can be computed using Lemma \ref{superlemma},
 $$
  g_\nu = -\sum_{k=1}^N 4 C_{4,1} a_k(\Phi_\nu) u_k = \rmi \sqrt{\frac{\pi}{2}}  \sum_{k=1}^N a_k(\Phi_\nu) u_k,
 $$
 thus resulting in the claimed expression.
\end{proof}

\subsubsection{Expansion of $E_\lambda(\Phi)$ in even dimensions $d \geq 4$}

Assuming that $\Phi$ is a spherical harmonic of degree $\ell$, we have that (in the asymptotic regime previously described) 
$$
 \tilde{j}_{\lambda}(\Phi)(r \theta) = C_{d,\ell} \lambda^{\ell + \frac{d-1}{2}} r^\ell \Phi(\theta) + O_{C^\infty(\R^d \setminus \{0\})}(\lambda^{\ell + \frac{d+3}{2}}).
$$
Therefore, using Lemma \ref{superlemma}, we get in dimensions $d \geq 6$, using $B_{-1}=0$,
\begin{gather} \label{expd6th1}
 E_\lambda(\Phi) = -(d-2+2\ell) C_{d,\ell} \lambda^{\ell +\frac{d-5}{2}} \sum_{j=1}^N a_j(\Phi) u_j + O_{C^\infty(M)}( \lambda^{\ell +\frac{d-1}{2}} ).
\end{gather}
In case $d=4$ we have obtain from Lemma \ref{superlemma} that
\begin{gather}
 E_\lambda(\Phi) = -(d-2+2\ell) C_{d,\ell} \lambda^{\ell +\frac{d-5}{2}} \sum_{j=1}^N a_j(\Phi) u_j \nonumber 
 \\+ \frac{1}{4} (d-2+2\ell) C_{d,\ell} \lambda^{\ell +\frac{d-1}{2}} (-\log \lambda) \sum_{j,k=1}^N 
 a_{kj}^{(1)}a_j(\Phi)  u_k + O_{C^\infty(M)}( \lambda^{\ell +\frac{d-1}{2}} ).\label{expd4th1}
\end{gather}

In case $\ell=0$ we have $a_j(\Phi)=0$ by Corollary \ref{supercor2}. Hence, in that case $E_\lambda(\Phi) = O( \lambda^{\frac{d-1}{2}} )$.
Therefore, for any $f \in C^\infty_0(M;\Lambda^pT^*M)$ we get in case $d \geq 6$
$$
\sum\limits_\nu E_{\lambda}(\Phi_\nu)\langle f, E_{\overline{\lambda}}(\Phi_\nu)\rangle = d^2 | C_{d,1}|^2 \lambda^{d-3} P^{(1)} f + O_{C^\infty(M)}(\lambda^{d-1}).
$$
In dimension $d=4$ we obtain the two-term expansion
$$\sum\limits_\nu E_{\lambda}(\Phi_\nu)\langle f, E_{\overline{\lambda}}(\Phi_\nu)\rangle = d^2 | C_{d,1}|^2 \left(\lambda^{d-3} P^{(1)} f - \frac{1}{2}\lambda^{d-1}(-\log \lambda) \left( P^{(1)} \right)^2 f  \right)+ O_{C^\infty(M)}(\lambda^{d-1}).
$$
We have proved the following theorems.
\begin{theorem}
 Suppose that $d$ is even and $d \geq 6$. Then for any $f \in C^\infty_0(M;\Lambda^pT^*M)$ and  for $| \lambda |$ small in a fixed sector $| \arg \lambda | \leq \Theta$ we have
 $$
-2 \rmi \lambda (R_\lambda - R_{-\lambda}) f=\sum\limits_\nu E_{\lambda}(\Phi_\nu)\langle f, E_{\overline{\lambda}}(\Phi_\nu)\rangle =d^2 | C_{d,1}|^2 \lambda^{d-3} P^{(1)} f + O_{C^\infty(M)}(\lambda^{d-1}).
$$
\end{theorem}

\begin{theorem}
 Suppose that $d=4$, then for any $f \in C^\infty_0(M;\Lambda^pT^*M)$  and $| \lambda |$ small  in a fixed sector $| \arg \lambda | \leq \Theta$ we have
 \begin{gather*}
-2 \rmi \lambda (R_\lambda - R_{-\lambda}) f=\sum\limits_\nu E_{\lambda}(\Phi_\nu)\langle f, E_{\overline{\lambda}}(\Phi_\nu)\rangle \\=\frac{\pi}{2} \lambda P^{(1)} f -\frac{\pi}{4} \lambda^3 (-\log \lambda)\left( P^{(1)}\right)^2 f + O_{C^\infty(M)}(\lambda^{3}).
\end{gather*}
\end{theorem}

\subsection{Analysis when $d=2$} \label{d2anal}

Finally we treat that fairly special case of dimension two. 

\begin{lemma} \label{ruleoutlemma1}
 In case $p=0,d=2$ we have $P=0$ and $B_{-2,k}=0$ for all $k>0$.
\end{lemma}
\begin{proof}
 The fact $P=0$ follows immediately from the maximum principle which implies there are no $L^2$-harmonic functions on $M$.
 By Lemma \ref{abovelemma} it suffices to show that $B_{-2,1}=0$.
 Assume by contradiction $B_{-2,1}\not=0$. By \eqref{diffres} this means in the expansion of $E_\lambda(\Phi)$
 we must have a non-zero top-order term of the form
 $$
  f \frac{1}{\lambda^{1/2}(-\log \lambda)}
 $$
 for some $\Phi$. By Prop. \ref{difftheo}  $-\rmi E_\lambda(\der r \wedge \Phi)$ has a leading expansion term of the form
  $$
  \rmi \frac{1}{\lambda^{3/2}(-\log \lambda)} \der f.
 $$
 This leading term must vanish by \eqref{diffres} and because the resolvent has a pole of order at most two. Hence $\der f=0$.
 In case $\partial \calO \not=\emptyset$ this already implies $f=0$ as $f$, by construction, satisfies relative boundary conditions. We will now show that this is also the case if the boundary is empty. By Prop. \ref{EoutofHankel} and since $\tilde j_\lambda(\Phi)=O(\lambda^{1/2})$ this singularity must appear in the expansion of 
 $$
  \tilde h^{(1)}_\lambda(A_\lambda \Phi).
 $$
 Morever, since $f$ is constant it must appear in
 $$
  \tilde h^{(1)}_\lambda(\Psi_\lambda),
 $$
 where $\Psi_\lambda $ is a Hahn holomorphic family of spherical harmonic of degree $\ell=0$. We have used here that $A_\lambda$ is Hahn-holomorphic
 and thus bounded.
 This function is of the form
 $$
   \lambda^{1/2} H^{(1)}_0(\lambda r) \Psi_\lambda
 $$
and is therefore of order $O(\lambda^{1/2} \log \lambda)$. This shows that $f=0$.
\end{proof}

\begin{lemma} \label{abovelemma3}
 If $p=0$ and $d=2$ then the resolvent, for $| \lambda |$ small and in a fixed sector $| \arg \lambda | \leq \Theta$, has an expansion of the form
 $$
   B_{-1} (-\log \lambda) + B(\lambda),
 $$
 where $B(\lambda)$ is Hahn-holomorphic and $B_{-1}: L^2_{comp} \to H^2_{loc}$ is of rank at most one. If $\partial \calO \not= \emptyset$
 we have $B_{-1}=0$. In case $\partial \calO = \emptyset$ any element in the range of $B_{-1}$ is a multiple of the constant function $1$.
\end{lemma}
\begin{proof}
 By Lemma \ref{ruleoutlemma1} and the general form \eqref{genexpansion} the resolvent has a Hahn-expansion of the form
 $$
  R_\lambda = \sum_{k=1}^N B_{-k} (-\log \lambda)^k +  B(\lambda),
 $$
 where $B(\lambda)$ is Hahn-holomorphic and the $B_{-k}$ are of finite rank.
 A leading term of the form $B_{-N} (-\log \lambda)^N$ with $N > 1$ gives a leading order term $-\rmi \pi B_{-N} (-\log \lambda)^{N-1}$ in the expansion of $R_\lambda - R_{-\lambda}$. Therefore $B_{-N}$ is symmetric and this leading term must arise from a leading term of 
 the expansion of $E_\lambda(\Phi)$ of the form $\lambda^{1/2}  (-\log \lambda)^{\frac{N-1}{2}} f$. Since then $E_\lambda(\der r \wedge \Phi)$
 has a singularity of the form $-\rmi \lambda^{-1/2}  (-\log \lambda)^{\frac{N-1}{2}} df$ this implies that $df=0$ by \eqref{diffres} since any singularity in $R_\lambda - R_{-\lambda}$ for $p=1$ must be weaker than $\frac{1}{\lambda^2}$.
  In that case we therefore have
 $B_{-N}= c \langle \cdot , 1 \rangle 1$. 
 Application of \ref{superlemmad2} shows that this term does not contribute to $E_\lambda$. We therefore get
 $$
  E_\lambda(\Phi) = O_{C^\infty(M)}(\lambda^{1/2}),
 $$
 thus implying that $N=1$. In case $\partial \calO \not= \emptyset$ the Dirichlet boundary condition (relative in case $p=0$ is equivalent to Dirichlet boundary conditions) implies that $f=0$. Hence, in this case $B_{-1}=0$.
 \end{proof}

 In the case $p=0, d=2$ we denote  by $\Phi_0=\frac{1}{\sqrt{2 \pi}}$ the normalised constant function that spans the space of spherical harmonics of degree zero.
 In case $p=2, d=2$ we let $\Phi_0=*\frac{1}{\sqrt{2 \pi}}$.

 \begin{proposition}\label{prop02}
  Assume $p=0$, $d=2$. 
  \begin{itemize}
   \item Suppose  $\partial \calO \not= \emptyset$, then $B_{-1}=0$ and  $E_\lambda(\Phi) =O_{C^\infty(M)}(\frac{\lambda^{1/2}}{-\log \lambda})$ for $| \lambda |$ small in a fixed sector $| \arg \lambda | \leq \Theta$.
  If $\Phi$ is a spherical harmonic of degree $\ell$, the function
  $$
   G(\Phi) = \lim_{\lambda \to 0} \lambda^{-1/2} (-\log \lambda) E_\lambda(\Phi)
  $$
  is nonzero only if $\ell=0$. In this case $G(\Phi)$ is the unique harmonic function satisfying Dirichlet boundary conditions at $\partial \calO$
  such that
  $$
   G(\Phi) - \sqrt{2 \pi}\Phi \log \frac{r}{2} = O_{C^\infty(M)}(1).
  $$
  \item Suppose $\partial \calO = \emptyset$. Then $B_{-1}$ has rank one and its range is spanned by the constant function. Moreover, if $F$ is a harmonic function satisfying Dirichlet boundary conditions such that
  $$
   F - a  \log \frac{r}{2} = O_{C^\infty(M)}(1),
  $$ 
  then $a=0$.
   \end{itemize}
 \end{proposition}
  \begin{proof}  
  Assume that $B_{-1}=0$. Then there is no constant term in the expansion of $R_\lambda - R_{-\lambda}$, which implies the bound $E_\lambda(\Phi) =O_{C^\infty(M)}(\frac{\lambda^{1/2}}{-\log \lambda})$.
  Since the resolvent has no singular terms we also have $E_\lambda(\Phi) = O_{C^\infty(M)}(\lambda^{\ell + \frac{1}{2}})$ if $\Phi$ has degree $\ell$.
  Therefore $G(\Phi)=0$ if $\ell>0$.
  If $\ell=0$ and $\Phi \not=0$ we have
  $$
   \lambda^{-\frac{1}{2}}\langle E_\lambda(\Phi),\Phi \rangle_{L^2(\mathbb{S}^{1}_r)} = \sqrt{\frac{\pi}{2}} \left( 2  J_0(\lambda r) + a(\lambda) H^{(1)}_0(\lambda r) \right) \| \Phi \|^2_{L^2(\mathbb{S}^1)}
  $$
  where $a(\lambda)= \| \Phi \|^{-2}_{L^2(\mathbb{S}^1)} \langle A_\lambda \Phi, \Phi \rangle$. Since the left hand side converges to $0$ as $\lambda \to 0$
  we obtain from the asymptotics of the Hankel function \eqref{sphankelexpzerod2} by comparing the expansion coefficients
  $$
   a(\lambda) = \frac{\pi}{ \rmi (-\log \lambda)} + O\left(\frac{1}{(\log \lambda)^2}\right),
  $$
 and
  $$
   \lim_{\lambda \to 0}\lambda^{-\frac{1}{2}}(-\log \lambda)\langle E_\lambda(\Phi),\Phi \rangle_{L^2(S^{1}_r)} = \sqrt{2 \pi} \log \frac{r}{2} \| \Phi \|^{2}  + O\left(1\right).  
   $$
   We have shown that $B_{-1}=0$ implies the existence of a non-zero harmonic function $G(\Phi_0)$ with asymptotic behaviour 
   $$
    G(\Phi_0) = \sqrt{2 \pi} \Phi_0 \log \frac{r}{2}   + O\left(1\right), \quad r \to \infty.
   $$
   
   Next note that the existence of such a function rules out the existence of a harmonic function $f$ satisfying Dirichlet boundary conditions 
   such that
   $$
    f(r,\theta) = 1 + O\left(\frac{1}{r}\right),\quad r \to \infty.
   $$
  Indeed, if such a function existed, then we would have 
  $$
   0= \langle \Delta G(\Phi_0), f \rangle - \langle G(\Phi_0), \Delta f \rangle = \int_{\mathbb{S}^1} \Phi_0(\theta) \der \theta \not=0.
  $$
  Conversely, if the constant function satisfies Dirichlet boundary conditions then $B_{-1} \not= 0$. Hence, $B_{-1}=0$ if and only if 
  $\partial \calO \not= \emptyset$. Since the range of $B_{-1}$ consists of constant functions satisfying Dirichlet boundary conditions this proves the second statement.
  It only remains to show uniqueness of the harmonic function $G(\Phi)$ satisfying Dirichlet boundary conditions. If there were two such functions the difference would have a multipole expansion for large enough $r$. By the above the constant coefficient in the multipole expansion has to vanish. Thus, the difference is a harmonic function that satisfies Dirichlet boundary conditions and vanishes at infinity. The maximum principle implies that such a function vanishes.
   \end{proof}

The non-trivial $2$-form $*1$ satisfies absolute boundary conditions independent of whether $\partial \calO$ is non-empty. The same argument as in the previous proposition then gives the following statement.
 \begin{proposition} \label{plprop}
  Suppose $p=2$, $d=2$, then the resolvent has an expansion of the form 
 $$
   B_{-1} (-\log \lambda) + B(\lambda),
 $$
  for $| \lambda |$  small in a fixed sector $| \arg \lambda | \leq \Theta$,
 where $B(\lambda)$ is Hahn-holomorphic and $B_{-1}: L^2_{comp} \to H^2_{loc}$ is of rank one and its range is spanned by the volume form $*1$.
 In particular $B_{-1} \not=0$.
 \end{proposition}
 
  \begin{proposition}\label{propextra}
  Assume $p=0$, $d=2$ and let $\Phi \in \mathcal{H}^0_1(\sphere)$ be a spherical harmonic of degree $1$. Then $E_\lambda(\Phi) = O_{C^\infty(M)}(\lambda^{3/2})$
    for $| \lambda |$  small in a fixed sector $| \arg \lambda | \leq \Theta$,
and the function 
  $$
   G_1(\Phi) := \lim_{\lambda \to 0} \lambda^{-\frac{3}{2}} E_\lambda(\Phi)
  $$
  is harmonic, satisfies Dirichlet boundary conditions, and we have
  $$
   G_1(\Phi)(r,\theta) = -\rmi \sqrt{\frac{\pi}{2}} \Phi r + O(1),
  $$
for sufficiently large $r$. This function is uniquely determined by this property in case $\partial \calO \not= \emptyset$ and is uniquely determined modulo a constant in case 
  $\partial \calO = \emptyset$.
    \end{proposition}
\begin{proof}
 The fact that $E_\lambda(\Phi) = O_{C^\infty(M)}(\lambda^{3/2})$ follows from Prop. \ref{besseljefunction} and the expansion \eqref{besseltildexp}. 
 Note that $B_{-1} = c \langle \cdot, 1 \rangle 1$ for some $c \in \mathbb{C}$. Therefore, by Lemma \ref{superlemmad2}, the singularity $B_{-1} (-\log \lambda)$ in the  resolvent does not contribute to $E_\lambda(\Phi)$. Since $E_\lambda(\Phi)$ is Hahn-holomorphic therefore the limit  
 $G_1(\Phi) =\lim\limits_{\lambda \to 0} \lambda^{-\frac{3}{2}} E_\lambda(\Phi)$ exists.
 On the other hand, by Prop. \ref{EoutofHankel}, we have
 $$
  G_1(\Phi) = \lim_{\lambda \to 0} \lambda^{-\frac{3}{2}} \left(   \tilde{j}_{\lambda}(\Phi) + \tilde h^{(1)}_\lambda(A_\lambda \Phi) \right)
 $$
  By the Hankel function asymptotics \eqref{sphankelexpzero} in case $\ell_\nu>0$ the existence of the limit 
  $\lim\limits_{\lambda \to 0} \lambda^{-\frac{3}{2}} \tilde h^{(1)}_\lambda(A_\lambda \Phi)$ implies that 
 $\langle A_\lambda \Phi_\nu, \Phi \rangle = O(\lambda^{\ell_\nu +1})$ and $\lim\limits_{\lambda \to 0} \lambda^{-\frac{3}{2}} \tilde h^{(1)}_\lambda(A_\lambda \Phi) = O(\frac{1}{r^{\ell_\nu}})$ for $r$ large.
 (see Lemma \ref{superlemma2} for a similar argument). For $\ell_\nu=0$ we get in the same way, using \eqref{sphankelexpzerod2},
 $\langle A_\lambda \Phi_\nu, \Phi \rangle = O(\frac{\lambda}{-\log \lambda})$ and 
 $\lim\limits_{\lambda \to 0} \lambda^{-\frac{3}{2}} \tilde h^{(1)}_\lambda(A_\lambda \Phi)= O(1)$ for $r$ large. 
 Now the expansion of the Bessel function \eqref{besseltildexp} gives the expansion of $G_1(\Phi)$ as claimed in the theorem. The uniqueness statement follows from the maximum principle.
  \end{proof}
  
 Now we analyse the case $p=1$.  
 If $\Phi$ is in $\mathcal{H}^1_0(\sphere)$ then we can write uniquely $\Phi = \frac{1}{2}\Psi \der r + \tilde \Phi$, where $\Psi \in \mathcal{H}^0_1(\sphere)$ and $\tilde \Phi \in \mathcal{H}^1_2(\sphere)$. Then, $\der (\Psi r) =  \Phi$. 
 \begin{definition} In case $d=2$ let $\Phi \in \mathcal{H}^1_0(\sphere)$,  we define the one-form $\varphi(\Phi) \in C^\infty(M; T^*M)$ as
  $$
   \varphi(\Phi) = \rmi \; \sqrt{\frac{2}{\pi}} \der G_1(\Psi),
  $$
  where $\Psi$ is the unique spherical harmonic of degree zero such that $\der(\Psi r) =  \Phi$.
 \end{definition}
 By the above we have 
 $$
  \varphi(\Phi) = \Phi + O(\frac{1}{r}),
 $$
 as $r \to \infty$. Moreover, $\varphi(\Phi)$ is harmonic and satisfies relative boundary conditions.

 We have the following corollary.
 \begin{corollary} \label{ellzeroterm}
  Suppose that $p=1$ and $d=2$ and let $\Phi \in \mathcal{H}^1_0(\sphere)$, then $E_\lambda(\Phi) = O(\lambda^{\frac{1}{2}})$ for $| \lambda |$  small in a fixed sector $| \arg \lambda | \leq \Theta$, and 
  $$
   \lim_{\lambda \to 0} \lambda^{-\frac{1}{2}} E_\lambda(\Phi) =\sqrt{2 \pi} \varphi(\Phi).
  $$
 \end{corollary}
 \begin{proof} As explained above we can write $\Phi = \frac{1}{2}\Psi \der r + \tilde \Phi$.
  Since $E_\lambda(\tilde \Phi) = O(\lambda^{\frac{5}{2}} \log \lambda)$ and $\der E_\lambda(\Psi) = -\rmi \lambda E_\lambda(\Psi \der r)$
  we obtain $E_\lambda(\Phi) = O(\lambda^{\frac{1}{2}})$ and
  $$
    \lim_{\lambda \to 0} \lambda^{-\frac{1}{2}} E_\lambda(\Phi) =   \lim_{\lambda \to 0}  \frac{\rmi}{2}  \lambda^{-\frac{3}{2}} \der E_\lambda(\Psi) = 
    \frac{\rmi}{2}  \der G_1(\Psi)=\sqrt{2 \pi} \varphi(\Phi).
  $$
 \end{proof}
 
 We are now going to refine the asymptotic expansions.
 
  \begin{proposition}\label{prop03}
  Assume $p=0$, $d=2$ and $\partial \calO \not= \emptyset$ and let $\Phi_0 \in \mathcal{H}^0_0(\sphere)$ be the constant function $\frac{1}{\sqrt{2 \pi}}$. Then there exists a holomorphic function $q(z)$, defined near zero, such that for 
  $| \lambda | \ll 1$ we have $q(0) = 1$ and, for $| \lambda |$  small in a fixed sector $| \arg \lambda | \leq \Theta$, we get
  $$
   E_\lambda(\Phi_0) =\frac{\lambda^{1/2}}{-\log \lambda} q(-\frac{1}{\log \lambda}) G(\Phi_0) + O_{C^\infty(M)}(\lambda^{5/2} (-\log \lambda)^N)
  $$  
  for some $N>0$.
  \end{proposition}
 \begin{proof}
  The resolvent is Hahn-holomorphic with coefficient group $2 \mathbb{Z} \times \mathbb{Z}$. Therefore only even powers of $\lambda$ appear in its expansion.
  Since $E_\lambda(\Phi_0)$ is Hahn-holomorphic near $0$ we have by Proposition \ref{prop02} that
  $$
   E_\lambda(\Phi_0) = \sum_{k=0}^\infty a_k \lambda^{1/2} \left( \frac{1}{-\log \lambda}\right)^{k+1} + O_{C^\infty(M)}(\lambda^{5/2} (-\log \lambda)^N)
  $$ 
  for some $N>0$ where the series converges normally. The same proof as that of Prop. \ref{prop02} shows that the coefficients
  are harmonic functions satisfying Dirichlet boundary conditions with an expansion of the form
  $$
   a_k(r) = \alpha_k \sqrt{2 \pi} \Phi_0 \log(\frac{r}{2}) + O(1).
  $$
  By Prop. \ref{prop02} this implies $a_k(r) = G(\phi_0) \alpha_k$ and $\alpha_0=1$. Hence, the series
  $$
   q(z) = \sum_{k=0} \alpha_k z^k
  $$
  converges normally and therefore defines a holomorphic function near zero with the required properties.
 \end{proof}
 
  \begin{proposition}\label{prop04}
  Suppose $d=2$ and either $p=0$ and  $\partial \calO = \emptyset$, or $p=2$, if $\Phi \not= 0$ has degree $\ell$,
  then
  $$
   H(\Phi):= \lim_{\lambda \to 0} \lambda^{-\frac{1}{2}} E_\lambda(\Phi)
  $$
  is non-zero if and only if $\ell=0$. In case $\ell=0$ we have $H(\Phi) = \sqrt{2 \pi} \Phi$. Moreover, we have  for $| \lambda |$  small in a fixed sector $| \arg \lambda | \leq \Theta$ that
  $$
   E_\lambda(\Phi) - \lambda^{\frac{1}{2}} H(\Phi) = O_{C^\infty(M)}(\lambda^{\frac{3}{2}} )
  $$
  and $a(\lambda)= \langle A_\lambda \Phi_0, \Phi_0 \rangle =O(\lambda^2)$.
 \end{proposition}
\begin{proof}
 For $\ell>0$ the resolvent expansion (Propositions \ref{prop02} and \ref{plprop}) implies the bound $E_\lambda(\Phi) =O_{C^\infty(M)}(\lambda^{\ell + \frac{1}{2}})$,
 We used here that $B_{-1}$ does not yield a contribution in Lemma \ref{superlemmad2} since constant terms in the multipole expansion do not contribute.
  Hence, $H(\Phi)$ vanishes if $\ell>0$.
 We now consider the case $\ell=0$ and $p=0$.
  Since $B_{-1} \not= 0$ (Proposition \ref{prop02}) the expansion \eqref{diffres} implies that $H(\Phi)$ is a non-zero multiple of the constant function. 
 Comparing coefficients in $(\Delta -\lambda^2) E_\lambda(\Phi) =0$ shows that any coefficient in the Hahn expansion of $E_\lambda(\Phi)$ of order less than
 $\lambda^{5/2}$ is harmonic and satisfies the boundary conditions. Since there is no harmonic form with a leading non-zero $\log r$ term in its multipole expansion (\ref{prop02}) this implies $a(\lambda)=O(\lambda^{2})$. 
 Here we have used \eqref{sphankelexpzerod2}.
 Now just use the expansion of $E_\lambda(\Phi)$ and Prop. \ref{EoutofHankel} to obtain $H(\Phi) = \sqrt{2 \pi} \Phi$. Moreover, the expansion coefficients of $E_\lambda(\Phi) - \lambda^{\frac{1}{2}} H(\Phi)$ of order less than $\lambda^{2}$ are harmonic and decay at infinity. They must therefore vanish. Hence, in case
 $\ell=0$ we obtain $E_\lambda(\Phi) - \lambda^{\frac{1}{2}} H(\Phi) =  O_{C^\infty(M)}(\lambda^2)$.
 Since the argument above also applies to absolute boundary conditions an application of the Hodge star operator reduces the case $p=2$ to the case $p=0$.
\end{proof}

\begin{proposition} \label{connectionprop}
 Let $d=2$ and suppose either $p=0$ and $\partial \calO = \emptyset$, or $p=2$, then we have the equality:
 $$
  B_{-1} = \frac{1}{2\pi} \langle \cdot , H(\Phi_0) \rangle H(\Phi_0).
 $$
  \end{proposition}
\begin{proof}
This follows immediately from $\rmi \pi B_{-1} = \frac{\rmi}{2}   \langle \cdot , H(\Phi_0) \rangle H(\Phi_0)$, which is obtained from the expansion \eqref{diffres}
by comparing coefficients. 
\end{proof}

\begin{theorem} \label{resexpd2p1}
 Suppose that $d=2, p=1$ and $\partial \calO \not= \emptyset$, and let $g(\Phi_0)$ be the unique harmonic function satisfying relative boundary conditions
 such that
 $$
  g(\Phi_0) =   \log \frac{r}{2} \Phi_0 + \beta \Phi_0 + O\left(\frac{1}{r}\right)
 $$ 
for $r$ sufficiently large.  The function $\psi(\Phi_0)=\der g(\Phi_0)$ is then closed and co-closed, satisfies relative boundary conditions, and 
 $$
  \psi(\Phi_0) =  \frac{\der r}{r} \Phi_0 + O\left(\frac{1}{r^2}\right).
 $$
 Let $Q =  \langle \cdot, \psi(\Phi_0) \rangle \psi(\Phi_0)$ and $T = \sum\limits_{\ell_\nu=0}  \langle \cdot, \varphi(\Phi_\nu) \rangle \varphi(\Phi_\nu)$. Then,  for $| \lambda |$ small and in a fixed sector $|\arg \lambda | \leq \Theta$, the resolvent has an expansion of the form
 $$
   R_\lambda = -\frac{P}{\lambda^2} - \frac{1}{\lambda^2} \frac{1}{-\log \lambda + \frac{\rmi \pi}{2} + \beta - \gamma} Q + B_{-1} (-\log \lambda) + B(\lambda)
 $$
 where $B(\lambda)$ is Hahn holomorphic, $\gamma$ is the Euler-Mascheroni constant. We have that $B_{-1}$ is given by 
\begin{align}
B_{-1}=\frac{P^{(2)}}{4} + T.
\end{align}
\end{theorem}
\begin{proof}
 For $\Phi \in \mathcal{H}^1_\ell(\sphere)$ we have a unique decomposition
 $$
  \Phi = \alpha \Phi_0 \der r + \tilde \Phi \der r + \iota_{\der r}(\der r \wedge \Phi ),
 $$
 where $\alpha=\alpha(\Phi) = \langle \iota_{\der r} \Phi, \Phi_0 \rangle_{L^2(\sphere)}$, and $\tilde \Phi$ is orthogonal to $\Phi_0$. 
 We then have
  $$E_\lambda(\Phi) = \left( \alpha E_\lambda(\Phi_0 \der r) + E_\lambda(\tilde \Phi \der r)  - \frac{\rmi}{\lambda} \delta E_\lambda(\der r \wedge \Phi)  \right).$$
 By Proposition \ref{prop04}, as $\delta H(\der r \wedge \Phi) =0$, we have 
 $
   -\frac{\rmi}{\lambda} \delta E_\lambda(\der r \wedge \Phi) = O_{C^\infty(M)}(\lambda^{\frac{1}{2}}).
 $
By Prop \ref{prop03} we have
 \begin{gather} \label{initialequ}
  -\frac{\rmi}{\sqrt{2 \pi}} E_\lambda(\Phi_0 \der r) = \frac{q(-\frac{1}{\log \lambda})}{\lambda^{\frac{1}{2}} (-\log \lambda)} \psi(\Phi_0) + O_{C^\infty(M)}(\lambda^{\frac{5}{2}} \log \lambda^N).
 \end{gather}
 for some $N>0$.
 Since $E_{\lambda}(\tilde \Phi) = O_{C^\infty(M)}(\lambda^{3/2})$ the general form of the resolvent and \eqref{diffres} imply the resolvent has the form
 $$
  R_\lambda = -\frac{P}{\lambda^2} - \frac{1}{\lambda^2} h\left( \frac{1}{-\log \lambda + \frac{\rmi \pi}{2}} \right) Q  +B_{-1}(-\log \lambda) + O_{C^\infty(M)}(1),
 $$ 
 where $h(z)$ is a holomorphic function defined near zero that is determined by
 $$
  h\left( \frac{z}{1 +\frac{\rmi \pi}{2} z}  \right) -h\left( \frac{z}{1 -\frac{\rmi \pi}{2} z}  \right)= -\pi \rmi z^2 q(z) \overline{q(\overline{z})}.
 $$
Since for $\lambda \in e^{\rmi \pi/2} \R$ the resolvent is self-adjoint the function $h$ must be real-valued for real arguments, thus implying
 $\overline{h(z)} = h(\overline{z})$.
 Using this fact, we have
 \begin{gather}\label{Eexpansion}
  E_\lambda(\Phi_0 \der r) =  \frac{1}{\lambda^2} h\left( \frac{1}{-\log \lambda + \frac{\rmi \pi}{2}} \right) Q (\Delta - \lambda^2) (\chi \tilde j_\lambda(\Phi_0 \der r)) \\+ \nonumber \frac{1}{\lambda^2} P  (\Delta - \lambda^2) (\chi \tilde j_\lambda(\Phi_0 \der r)) + O_{C^\infty(M)}((-\log \lambda) \lambda^{3/2}).
 \end{gather} 
 Now we use Lemma \ref{superlemmad2} to conclude that the
 second term vanishes and 
 $$
   E_\lambda(\Phi_0 \der r) =  \rmi \sqrt{2 \pi} \frac{1}{\lambda^{\frac{1}{2}}} h\left( \frac{1}{-\log \lambda + \frac{\rmi \pi}{2}} \right) \psi(\Phi_0)+  O_{C^\infty(M)}((-\log \lambda) \lambda^{3/2}).
 $$
By Theorem \ref{StoneF} we obtain 
 \begin{gather*}
  \left( h\left( \frac{1}{-\log \lambda + \frac{\rmi \pi}{2}} \right) -h\left( \frac{1}{-\log \lambda - \frac{\rmi \pi}{2}} \right) \right) Q 
  \\= -\rmi \pi h\left( \frac{1}{-\log \lambda + \frac{\rmi \pi}{2}} \right) h\left( \frac{1}{-\log \lambda - \frac{\rmi \pi}{2}} \right) Q
 \end{gather*}
 This implies that the function $h$ satisfies the following equation
 $$
  h(\frac{t}{1 + t \frac{\rmi \pi}{2}}) -  h(\frac{t}{1 - t \frac{\rmi \pi}{2}}) = -\rmi \pi h(\frac{t}{1 + t \frac{\rmi \pi}{2}})h(\frac{t}{1 - t \frac{\rmi \pi}{2}}).
 $$
 where we substituted $t = \frac{1}{- \log \lambda}$. 
 By \eqref{initialequ} we have $h(t) = t + O(t^2)$. It follows that the function $\frac{1}{h}$ is meromorphic and one can use the functional equation above to see that
 $g(t) = \frac{1}{h(t)} - \frac{1}{t}$ defines a holomorphic function near zero that is invariant under the transformation $t \mapsto \frac{t}{1- \pi \rmi t}$. It follows that $g$ is constant.
 Hence, if $h(t) = t - \alpha t^2 + O(t^3)$ we have
 $$
  h(t) = \frac{t}{1 + \alpha t}, \quad h\left( \frac{1}{-\log \lambda + \frac{\rmi \pi}{2}} \right) = \frac{1}{- \log \lambda + \rmi \pi/2 + \alpha}.
 $$
 In order to relate $\alpha$ and $\beta$ note that it follows from the form of $h$ that
 $$
  \langle A_{\lambda}(\Phi_0 \der r), \Phi_0 \der r \rangle_{L^2(\sphere)} =  \frac{\pi}{ \rmi (-\log \lambda + \rmi \pi/2 + \alpha)} + O(\lambda).
 $$
Because $\der r \wedge$ commutes with $A_\lambda$ this implies
 $$
  a(\lambda) =  \frac{\pi}{ \rmi (-\log \lambda + \rmi \pi/2 + \alpha)} + O(\lambda).
 $$
 In the expansion of $E_\lambda(\Phi_0) = \tilde j_\lambda(\Phi_0) + \tilde h^{(1)}_\lambda(A_\lambda \Phi_0)$ for large $r$ one obtains
 $$
  \lim_{\lambda \to 0} \lambda^{-\frac{1}{2}} (-\log \lambda)\frac{1}{\sqrt{2 \pi}} E_\lambda(\Phi_0) = \log\left(\frac{r}{2}\right) + \alpha + \gamma + O(1/r),
 $$
 using the asymptotic properties of the Hankel function found in Appendix \ref{Abf}.
 This shows that $\beta =  \alpha + \gamma$. To find $B_{-1}$ we consider
\begin{align}\label{Bminus}
R_{\lambda}-R_{-\lambda}= -\frac{1}{\lambda^2}\left(\frac{Q}{-\log \lambda+\frac{i\pi}{2}+\alpha}-\frac{Q}{-\log \lambda-\frac{i\pi}{2}+\alpha}\right) +B_{-1}(\rmi \pi) + O(\frac{1}{-\log \lambda})
\end{align}
and compare coefficients with the expansion \eqref{diffres}. The term $\ell =0$ contributes
$$
 \frac{\rmi}{2 \lambda} \sum_{\ell_\nu=0} E_{\lambda}(\Phi_\nu)\langle f, E_{\overline{\lambda}}(\Phi_\nu)\rangle =  \pi \rmi \; T + O(\frac{1}{\log \lambda}).
$$
For $\ell=1$ we can use  Prop. \ref{besseljefunction}, Lemma \ref{superlemmad2} and the fact that $P^{(1)}=0$ to obtain
$$
  \frac{\rmi}{2 \lambda} \sum_{\ell_\nu=1} E_{\lambda}(\Phi_\nu)\langle f, E_{\overline{\lambda}}(\Phi_\nu)\rangle = -\frac{1}{\lambda^2}\left(\frac{Q}{-\log \lambda+\frac{i\pi}{2}+\alpha}-\frac{Q}{-\log \lambda-\frac{i\pi}{2}+\alpha}\right) + O(\lambda).
 $$
 In the same way we get for $\ell=2$
$$
  \frac{\rmi}{2 \lambda} \sum_{\ell_\nu=2} E_{\lambda}(\Phi_\nu)\langle f, E_{\overline{\lambda}}(\Phi_\nu)\rangle =  \frac{\rmi}{2}  | 4\, C_{2,2} |^2 P^{(2)} + O(\frac{1}{\log \lambda}).$$
  Finally, the terms with $\ell>2$ are of order $O(\lambda^2)$. Comparing coefficients and using $C_{2,2} = -\sqrt{\frac{\pi}{32}}$ we obtain
  $$
   B_{-1} = T + \frac{P^{(2)}}{4}.
  $$
\end{proof}

\begin{corollary} \label{cor327}
 Suppose that $d=2, p=1$ and $\partial \calO \not= \emptyset$.  Let $\Phi$ be a spherical harmonic of degree $\ell$. Then, for $| \lambda |$ small and in a fixed sector $|\arg \lambda | \leq \Theta$:
 \begin{itemize}
  \item if $\ell=0$ we have $E_\lambda(\Phi)= \sqrt{2 \pi} \varphi(\Phi) \lambda^{\frac{1}{2}} + O_{C^\infty(M)}(\frac{\lambda^{1/2}}{-\log \lambda}).$
  \item if $\ell=1$ we have
  \begin{gather*}
   E_\lambda(\Phi) =  \rmi \sqrt{2 \pi} \lambda^{-\frac{1}{2}}  \frac{1}{-\log \lambda + \frac{\rmi \pi}{2} + \beta - \gamma} \alpha(\Phi) \psi(\Phi_0)\\-\rmi \sqrt{2 \pi} \sum_{\ell_\nu=0} a_{\varphi(\Phi_\nu)}(\Phi) \varphi(\Phi_\nu) \lambda^\frac{3}{2} (-\log \lambda) +  O_{C^\infty(M)}(\lambda^{3/2}).
 \end{gather*}
  \item if $\ell>1$ we have 
  \begin{gather*}
 E_\lambda(\Phi) = - 2\ell C_{2,\ell} \lambda^{\ell -\frac{3}{2}} \sum_{j=1}^N a_j(\Phi) u_j - 2\ell C_{2,\ell} \lambda^{\ell -\frac{3}{2}}  \frac{1}{-\log \lambda + \frac{\rmi \pi}{2} + \beta - \gamma} a_{\psi(\Phi_0)}(\Phi) \psi(\Phi_0) \\
+ 2 \ell C_{2,\ell} \lambda^{\ell +\frac{1}{2}} (-\log \lambda) \left( \frac{1}{4}\sum_{j,k=1}^N 
 a_{kj}^{(2)}a_j(\Phi)  u_k + \sum_{\ell_\nu=0} a_{\varphi(\Phi_\nu)}(\Phi) \varphi(\Phi_\nu)  \right) + O_{C^\infty(M)}( \lambda^{\ell +\frac{1}{2}}).
\end{gather*}
 \end{itemize}
\end{corollary}
\begin{proof}
 The case $\ell=0$ follows immediately from Corollary \ref{ellzeroterm}. The other cases are computed using Prop. \ref{besseljefunction} and Lemma \ref{superlemma}.
\end{proof}

Similarly, in case there are no obstacles, we have the following.

\begin{proposition} \label{resexpd2p1no}
 Suppose that $d=2, p=1$ and $\partial \calO = \emptyset$ and let $T = \sum\limits_{\ell_\nu=0}  \langle \cdot, \varphi(\Phi_\nu) \rangle \varphi(\Phi_\nu)$. Then the resolvent has an expansion of the form
 $$
   R_\lambda = -\frac{P}{\lambda^2}  + B_{-1} (-\log \lambda) + B(\lambda)
 $$
 where $B(\lambda)$ is Hahn holomorphic, $\gamma$ is the Euler-Mascheroni constant.We have that $B_{-1}$ is given by 
\begin{align}
B_{-1}=\frac{P^{(2)}}{4} + T.
\end{align}
\end{proposition}
\begin{proof}
 Proposition \ref{prop04} together with 
 $$
  E_\lambda(\Phi) = \frac{\rmi}{\lambda} \left( \der E_\lambda(\iota_{dr} \Phi) -  \delta E_\lambda(\der r \wedge \Phi)\right)
 $$
 implies the bound $E_\lambda(\Phi) = O(\lambda^{1/2})$. The same method as in the proof of the previous proposition  then allows us to conclude that
 the resolvent has the claimed form. The computation of $B_{-1}=\frac{P^{(2)}}{4} + T$ is exactly the same as in the proof of the previous proposition
 with the simplification that the terms containing $Q$ are absent.
\end{proof}

\begin{corollary} \label{cor329}
 Suppose that $d=2, p=1$ and $\partial \calO = \emptyset$.  Let $\Phi$ be a spherical harmonic of degree $\ell$. Then
 \begin{itemize}
  \item if $\ell=0$ we have $E_\lambda(\Phi)= \sqrt{2 \pi} \varphi(\Phi) \lambda^{\frac{1}{2}} + O_{C^\infty(M)}(\frac{\lambda^{1/2}}{-\log \lambda}).$
  \item if $\ell=1$ we have
  \begin{gather*}
   E_\lambda(\Phi) = -\rmi \sqrt{2 \pi} \sum_{\ell_\nu=0} a_{\varphi(\Phi_\nu)}(\Phi) \varphi(\Phi_\nu) \lambda^\frac{3}{2} (-\log \lambda) +  O_{C^\infty(M)}(\lambda^{3/2}).
 \end{gather*}
  \item if $\ell>1$ we have 
  \begin{gather*}
 E_\lambda(\Phi) = - 2\ell C_{2,\ell} \lambda^{\ell -\frac{3}{2}} \sum_{j=1}^N a_j(\Phi) u_j  \\
+ 2 \ell C_{d,\ell} \lambda^{\ell +\frac{d-1}{2}} (-\log \lambda) \left( \frac{1}{4}\sum_{j,k=1}^N 
 a_{kj}^{(2)}a_j(\Phi)  u_k + \sum_{\ell_\nu=0} a_{\varphi(\Phi_\nu)}(\Phi) \varphi(\Phi_\nu)  \right) + O_{C^\infty(M)}( \lambda^{\ell +\frac{1}{2}}).
\end{gather*}
 \end{itemize}
\end{corollary}

\section{General bounds and expansion of the scattering amplitude} \label{ScatterSectionExp}

Summarising the results from the previous two sections we have the following asymptotic behaviour of the generalised eigenfunctions $E_\lambda(\Phi)$.

\begin{lemma} \label{lemmafourone}
Let $d \geq 2$ and
 suppose that $\Phi \in \mathcal{H}^p_\ell(\sphere)$ is a spherical harmonic of degree $\ell$ then as $\lambda \to 0$ we have
 \begin{gather*}
  E_\lambda(\Phi) = O_{C^\infty(M)}(\lambda^{\frac{d-1}{2}}) \textrm{ if } \ell =0,\\
  E_\lambda(\Phi) = O_{C^\infty(M)}(\lambda^{\ell+\frac{d-5}{2}}) \textrm{ if } \ell \geq 1.
 \end{gather*}
\end{lemma}

By unitarity of the scattering matrix the operator family $A_\lambda$ is holomorphic at zero in odd dimensions and Hahn-holomorphic in even dimensions, respectively. The expansion of Prop. \ref{EoutofHankel} together with the analytic properties of the Hankel function can be used to obtain much more detailed information about $A_\lambda$. 

%

\begin{theorem} \label{AEexpansion}
 Suppose that $\Phi$ is a spherical harmonic of degree $\ell$ and that for  $| \lambda|$ small
 $$
  E_\lambda(\Phi) = \lambda^{\ell +\frac{d-5}{2}} \sum_{\alpha,\beta} F_{\alpha,\beta}(\Phi) \lambda^\alpha (-\log \lambda)^{-\beta} + O_{C^\infty(M)}( \lambda^{\ell + \frac{d-1}{2}}),
 $$
 where $\sum_{\alpha,\beta} F_{\alpha,\beta}(\Phi) \lambda^\alpha (-\log \lambda)^{-\beta}$ is Hahn-holomorphic. If $(\alpha,\beta) < (2,0)$, the function $F_{\alpha,\beta}(\Phi)$ is harmonic and bounded, and in this case let
$F_{\alpha,\beta}^\nu(\Phi)$ be $\nu$-coefficient in its multipole expansion 
 $$
  F_{\alpha,\beta}(\Phi)(r,\theta) = \sum_\nu F_{\alpha,\beta}^\nu(\Phi) \frac{1}{r^{\ell_\nu + d-2}} \Phi_\nu
 $$
 for large $r\gg 0$.
We then have in case $\ell_\nu + \frac{d-2}{2} >0$
 \begin{gather*}
  \langle A_\lambda \Phi, \Phi_\nu \rangle_{L^2(\sphere)} \\= \frac{\rmi}{2} (-1)^{\ell_\nu} C_{d,\ell_\nu} \left(d -2+ 2 \ell_\nu \right) \left( \sum_{\alpha,\beta} F_{\alpha,\beta}^\nu(\Phi) \lambda^{\ell + \ell_\nu + d-4+\alpha} (-\log \lambda)^{-\beta} \right)+ O(\lambda^{\ell + \ell_\nu + d-2}).
 \end{gather*}
 If $d=2$ and $\ell_\nu=0$ we have for $|\lambda|$ small
  \begin{gather*}
  \langle A_\lambda \Phi, \Phi_\nu \rangle_{L^2(\sphere)} \\=  \rmi \sqrt{\frac{\pi}{2}} \left( \sum_{\alpha,\beta} F_{\alpha,\beta}^\nu(\Phi) \lambda^{\ell -2+\alpha} (-\log \lambda)^{-\beta-1} \right)+ O(\frac{\lambda^{\ell}}{-\log \lambda}) = \rmi \sqrt{\frac{\pi}{2}} F_{2,-1}^\nu(\Phi) \lambda^{\ell} + O(\frac{\lambda^{\ell}}{-\log \lambda}).
 \end{gather*}
\end{theorem}
\begin{proof}
 We have, by Prop. \ref{EoutofHankel} and the expansion \eqref{besseltildexp}, 
 $$ \lambda^{\ell +\frac{d-5}{2}} \sum_{\alpha,\beta} F_{\alpha,\beta}(\Phi) \lambda^\alpha (-\log \lambda)^{-\beta} = \tilde h^{(1)}_\lambda(A_\lambda \Phi) + O_{C^\infty(M)}( \lambda^{\ell + \frac{d-1}{2}}).
 $$
 Multiplication by $\overline{\Phi_\nu}$ and integration over the sphere of radius $R \gg 0$ gives
 \begin{gather*}
  \lambda^{\ell +\frac{d-5}{2}} \sum_{\alpha,\beta} F_{\alpha,\beta}^\nu(\Phi) \lambda^\alpha (-\log \lambda)^{-\beta} R^{-\ell_\nu -d+2} \\= 
  (-\rmi)^{\ell_\nu} h^{(1)}_{\ell_\nu}(\lambda R) \lambda^{\frac{d-1}{2}} \langle A_\lambda \Phi, \Phi_\nu \rangle + O( \lambda^{\ell + \frac{d-1}{2}}).
 \end{gather*}
 In case $\ell_\nu + \frac{d-2}{2} >0$ theorem now follows from the expansion
 $$
   \left( h^{(1)}_{\ell_\nu}(x) \right)^{-1} = \frac{1}{2} \rmi^{\ell_\nu+1} \left(d -2+ 2 \ell_\nu \right) C_{d,\ell_\nu} \, x^{\ell_\nu+d-2} + O(x^{\ell_\nu+d}), \quad x \to 0,
  $$
 which is valid with this error term if $\ell_\nu + \frac{d-2}{2} \geq 1$ (see e.g. \cite[10.8.1]{olver2010nist} in the even dimensional case and  \cite[10.53.1,10.53.2]{olver2010nist} in the case of odd dimensions).  Note that there are no non-zero terms with $\ell_\nu=0$ in dimension $3$.
 If $d=2$ and $\ell_\nu=0$ we have
 $$
   \left( h^{(1)}_{\ell_\nu}(x) \right)^{-1} = \rmi \sqrt{\frac{\pi}{2}} (-\log x)^{-1} + O(\frac{1}{(\log x)^2}).
  $$
  In this case $F_{\alpha,\beta}^\nu(\Phi)=0$ if $(\alpha,\beta)<(2,-1)$.
\end{proof}

%
This shows the bounds for the scattering amplitude stated in Theorem \ref{scatterexp}. More precisely, 
in dimensions $2$, $3$ and $5$ the expansions of $A_\lambda$ are therefore.\vspace{0.5cm}\\
{\underline {\bf Case $d=3$:}}

 \begin{gather*}
 \langle A_\lambda \Phi, \Phi_\nu \rangle = -\frac{\rmi}{2} (2\ell+1) (2\ell_\nu+1)C_{3,\ell} \overline{C_{3,\ell_\nu} }  \,  \lambda^{\ell +\ell_\nu -1}  \\ \times \left( \sum_{j=1}^N a_j(\Phi) \overline{a_j(\Phi_\nu)} - \rmi  \lambda \sum_{j,k=1}^N 
 a_{kj}^{(1)}a_j(\Phi)  \overline{a_j(\Phi_\nu)} \right) + O(\lambda^{\ell+\ell_\nu+1})
 \end{gather*}

\noindent
{\underline {\bf Case $d=4$:}}

 \begin{gather*}
 \langle A_\lambda \Phi, \Phi_\nu \rangle = -2\rmi  (\ell+1) (\ell_\nu+1)C_{4,\ell} \overline{C_{4,\ell_\nu} }  \,  \lambda^{\ell +\ell_\nu }  \\ \times \left( \sum_{j=1}^N a_j(\Phi) \overline{a_j(\Phi_\nu)} - \frac{1}{4} \lambda^2 (-\log \lambda) \sum_{j,k=1}^N 
 a_{kj}^{(1)}a_j(\Phi)  \overline{a_j(\Phi_\nu)} \right) + O(\lambda^{\ell+\ell_\nu+2})
 \end{gather*}

\noindent
{\underline {\bf Case $d=2, p=1$:}}

 \begin{gather*}
 \langle A_\lambda \Phi, \Phi_\nu \rangle =
 - 2 \rmi \ell \, \ell_\nu \, C_{2,\ell}\,  \overline{C_{2,\ell_\nu}} \lambda^{\ell +\ell_\nu -2}\\\times  \left (\left(\sum_{j=1}^N a_j(\Phi) \overline{a_j(\Phi_\nu)}\right)  +  \frac{1}{-\log \lambda + \frac{\rmi \pi}{2} + \beta - \gamma} a_{\psi(\Phi_0)}(\Phi) \overline{a_{\psi(\Phi_0)}(\Phi_\nu)} \right. \\- \left. (-\log \lambda) \lambda^2 \left( \frac{1}{4}\sum_{j,k=1}^N 
 a_{kj}^{(2)}a_j(\Phi)  \overline{a_k(\Phi)} + \sum_{\ell_\mu=0}  a_{\varphi(\Phi_\mu)}(\Phi) \overline{a_{\varphi(\Phi_\mu)}(\Phi_\nu)}  \right) \right)  + O(\lambda^{\ell + \ell_\nu}).
 \end{gather*}

Note that $A_\lambda$ is (as an operator) Hahn holomorphic (Corollary \ref{Hahnmamplitude}). The above  implies that 
the expansion coefficients of order less than $\lambda^{d-2}$ must vanish, which can be summarised into the following corollary. 
\begin{corollary}
Suppose that $d \geq 3$, then for any $s \in \R$ we have
$\| A_\lambda \|_{L^2 \to H^s} = O(\lambda^{d-2})$. In case $d=2$ we have $\| A_\lambda \|_{L^2 \to H^s} = O(\frac{1}{-\log \lambda})$.
\end{corollary}

\section{Scattering and Cohomology} \label{cohoscatter}

In order to describe the cohomology spaces of $X$ and $M$ is is convenient to introduce some additional spaces which we describe first.
Since $(X,g)$ is Euclidean at infinity, there exist compact sets $K \subset X$ and $\tilde K \subset \R^d$ such that $X \setminus K$ is isometric to $\R^d \setminus \tilde K$. We choose $R>0$ large enough so that the interior of $B_R = B_R(0)$ contains $\tilde K \cup \calO$. Then, $\R^d \setminus \mathrm{int}(B_R)$ is isometric to the complement of a compact subset $Y \subset X$ with $\partial Y = \sphere_R$. We then have that $N = Y \setminus \calO$ is a compact subset of $M$ and $\partial N = \partial Y \cup \partial \calO$.\\
Let $(r,\theta) \in (R,\infty) \times \sphere$ be spherical coordinates on $M \setminus Y$. The substitution $x=\frac{1}{r}$ gives coordinates
$(x,\theta)  \in (0,\frac{1}{R}) \times \sphere$. Then $(x,\theta)  \in [0,\frac{1}{R}) \times \sphere$ are coordinates endowing the radial compactification 
$\overline X$ of $X$ with the structure of a manifold with boundary $\partial \overline X \cong  \sphere$. Similarly, $\overline M = \overline X \setminus \calO$
is the radial compactification of $M$ and has boundary $\partial \overline M = \partial \overline X \cup \partial \calO$.

We will use the ring $\R$ for the cohomology groups without further reference, so we will simply write $H^p(M)$ for $H^p(M,\R)$
and $H^p(M,\partial \calO)$ for $H^p(M,\partial \calO,\R)$ for the relative cohomology groups. Hence all cohomology groups may be realised by deRham cohomology, i.e. by the complex of differential forms.
Note that the inclusions $Y \hookrightarrow X$, $N \hookrightarrow M$, $X \hookrightarrow \overline X$, and $M \hookrightarrow \overline M$ are homotopy equivalences and hence the induced maps in cohomology are isomorphisms. We therefore have natural isomorphisms
\begin{gather*}
 H^p(N) \cong H^p(M) \cong  H^p(\overline M),\\
 H^p(N, \calO) \cong H^p(M,\calO) \cong  H^p(\overline M,\calO).
\end{gather*}

Since $M$ is the interior of a manifold with boundary $\overline M$ we also have natural isomorphisms
\begin{gather*}
 H^p_0(M) \cong  H^p(\overline M, \partial \overline X),\\
 H^p_0(M, \partial \calO) \cong  H^p(\overline M, \partial \overline X \cup \partial \calO) \cong H^p_0(X \setminus \calO),
\end{gather*}
induced by the inclusion maps. Unless stated otherwise we will identify $\partial \overline X$ with the sphere $\sphere$.

We then have the following standard exact sequences:
\begin{equation}
 \begin{tikzcd}[row sep=5ex, column sep=4ex]   \label{sequ1}
    \arrow[r,""] & H^{p-1}(\sphere)  \arrow[r,""] & H^p_0(M, \partial \calO)  \arrow[r,""]  & H^p(M, \partial \calO)    
    \arrow[r,""] &  H^{p}(\sphere)  \arrow[r,""] & {},
 \end{tikzcd}
 \end{equation}
 
 and
 
 \begin{equation}
 \begin{tikzcd}[row sep=5ex, column sep=5ex] \label{sequ2}
    \arrow[r,""] & H^{p-1}(\sphere)  \arrow[r,""] & H^p_0(M)  \arrow[r,""]  & H^p(M)    
    \arrow[r,""] &  H^{p}(\sphere)  \arrow[r,""] & {}.
 \end{tikzcd}
 \end{equation}

Recall that the kernel of $\Delta_{p,rel}$ consists of the $L^2$-harmonic $p$-forms $\mathcal{H}^p_{rel}(M)$, and this space is isomorphic to the $L^2$-cohomology spaces. As explained in the introduction we have in case $d \geq 3$: 
\begin{gather*}
 \mathcal{H}^d_{rel}(M)=\{0\},\\
  \mathcal{H}^d_{rel}(M) \cong H^p_0(M,\partial \mathcal{O}) \cong H^p(\overline M,\partial \mathcal{O} \cup \partial \overline X),
\end{gather*}
where the isomorphism $\mathcal{H}^d_{rel}(M) \to H^p(\overline M,\partial \mathcal{O} \cup \partial \overline X)$ is given by understanding an $L^2$-harmonic form $u \in \mathcal{H}^p_{rel}(M)$ as a differential form on $\overline M$ whose restriction to $\mathcal{O} \cup \partial \overline X$ vanishes. The fact that $u$ is smooth up to $\partial \overline X$ follows from the fact that $u$ has a multipole expansion of the form 
$$
 u (r,\theta) = \sum_\nu a_\nu(u) \frac{1}{r^{\ell_\nu+d-2}} \Phi_\nu(\theta) =  \sum_\nu a_\nu(u) x^{\ell_\nu+d-2} \Phi_\nu(\theta).
$$
In dimension $d=2$ we have
\begin{gather*}
 \mathcal{H}^0_{rel}(M)=\mathcal{H}^2_{rel}(M)=\{0\},\\
  \mathcal{H}^1_{rel}(M) \cong H^1( M,\partial \mathcal{O}),
\end{gather*}
in which case the isomorphism $\mathcal{H}^1_{rel}(M) \to H^1( M,\partial \mathcal{O})$ is given by simply mapping the square integrable harmonic form $u$ to its equivalence class in $H^p( M,\partial \mathcal{O})$.
We define the map
\begin{gather*}
  \sigma_\ell: \mathcal{H}^{p}_{rel}(M)  \to \mathcal{H}^p_\ell(\sphere),\\
  u \mapsto \sigma_{\ell}(u)= \sum_{\nu, \ell_\nu=\ell} a_\nu(u) \Phi_\nu.
\end{gather*}

We say $u$ has order $m$ if $a_\nu(u)=0$ when $\ell_\nu < m$. Thus, $u \in\mathcal{H}^p_{rel}(M)$ is of order $m$ if and only if
$u = O(r^{-(m + d-2)})$ for sufficiently large $r$. Denote the vector space of elements $ u \in \mathcal{H}^p_{rel}(M)$ of order $m$ by $\mathcal{H}^{p,m}_{rel}(M)$.
Of course we have
$$
 \mathcal{H}^{p,m+1}_{rel}(M) \subset  \mathcal{H}^{p,m}_{rel}(M)
$$
and therefore this defines a filtration of $\mathcal{H}^{p}_{rel}(M)$.
Since the multipole expansion converges on compact sets it follows from unique continuation that 
$$
\mathcal{H}^{p,\infty}_{rel}(M)=\bigcap\limits_{m\geq0} \mathcal{H}^{p,m}_{rel}(M) = \{0\}.
$$
We have the following exact sequence
\begin{equation}
 \begin{tikzcd}\label{DIAGRAMA}
   0  \arrow[r,""] & \mathcal{H}^{p,m+1}_{rel}(M)  \arrow[r,""]  &  \mathcal{H}^{p,m}_{rel}(M) \arrow[r,"\sigma_m"] &  \mathcal{H}_m(\sphere).
 \end{tikzcd}
 \end{equation}
 We will denote that quotient $\mathcal{H}^{p,m}_{rel}(M)/\mathcal{H}^{p,m+1}_{rel}(M)$ by $\mathcal{K}^{p,m}_{rel}(M)$.
 We can also use the $L^2$-inner product to identify $\mathcal{K}^{p,m}_{rel}(M)$ with the orthogonal complement of $\mathcal{H}^{p,m+1}_{rel}(M)$
 in $\mathcal{H}^{p,m}_{rel}(M)$. All together we have
 $$
  \mathcal{H}^{p}_{rel}(M) \cong \bigoplus_{m=0}^\infty \mathcal{K}^{p,m}_{rel}(M).
 $$
 Note that the map $\sigma_\ell$ does however not in general commute with the projection onto the summands in the direct sum.
 Contributions to the multipole expansion proportional to $\frac{1}{r^{\ell_\nu + d-2}} \Phi_\nu(\theta)$ are square integrable
 if and only if $2 \ell_\nu + d >4$. Therefore, we have $\mathcal{K}^{p,0}_{rel}(M) = \{0\}$ and $\mathcal{K}^{p,1}_{rel}(M) = \{0\}$ if $d=2$.
 As a consequence of Corollary \ref{supercor2} one always has $\mathcal{K}^{p,0}_{rel}(M) = \{0\}$ for any $d \geq 2$.
 
 Given $\Phi \in \mathcal{H}^p_\ell(\sphere)$ define 
 $$
  F_\ell(\Phi):= \lim_{\lambda \to 0}  \lambda^{-\ell -\frac{d-5}{2}} E_\lambda(\Phi).
 $$
 It follows from the top term of the expansion of $E_\lambda(\Phi)$ that
 the limit exists in the locally convex topological vector space $C^\infty(M; \Lambda^p T^* M)$ and is in $\mathcal{H}^p_{rel}(M)$. Therefore $F_\ell$ is a linear map
 $F_\ell: \mathcal{H}^p_\ell(\sphere) \to \mathcal{H}^p_{rel}(M)$. The following proposition paraphrases the asymptotic
 $$
  E_\lambda(\Phi) =  -(d-2+2\ell) C_{d,\ell} \lambda^{\ell +\frac{d-5}{2}}  \sum_{j=1}^N a_j(\Phi) u_j + o(\lambda^{\ell +\frac{d-5}{2}})
 $$
 as $\lambda \to 0$.
 \begin{proposition}
  The map $F_\ell$ is the adjoint of $-(d-2+2\ell) \overline{C_{d,\ell}} \; \sigma_\ell$, i.e.
  $$
     \langle F_\ell(\Phi), u \rangle_{L^2(M)} = -(d-2+2\ell) C_{d,\ell} \langle \Phi, \sigma_\ell(u) \rangle_{L^2(\sphere)}. 
  $$
 \end{proposition}
 
 \begin{corollary} \label{Aldensfavourite}
  If $\mathcal{K}^{p,m}_{rel}={0}$ for all $m<\ell$, then the range of $F_\ell$ equals $\mathcal{K}^{p,\ell}_{rel}(M)$.
 \end{corollary}
 \begin{proof}
 By assumption $ \mathcal{H}^p_{rel}(M) =  \mathcal{H}^{p,\ell}_{rel}(M)$. The range of $F_\ell$ is then the orthogonal complement of the kernel of $\sigma_\ell$. 
 \end{proof}

 Our first observation is the following.
 \begin{proposition} \label{prop5.3}
  If $d \geq 3$ and $2 \leq p < d$, then  $\mathcal{K}^{p,1}_{rel}(M) =\{0\}$. Hence, in this case $ \mathcal{H}^{p}_{rel}(M) =  \mathcal{H}^{p,2}_{rel}(M)$.
 \end{proposition}
 \begin{proof}
  First note that it follows from the exactness of \eqref{sequ1} that the canonical map $H^p_0(M,\partial \mathcal{O}) \to H^p(M,\partial \mathcal{O})$
  is injective if $2 \leq p < d$. Suppose that $v \in \mathcal{H}^{p,1}_{rel}(M)$ and assume that $v$ is orthogonal to  $\mathcal{H}^{p,2}_{rel}(M)$. Hence,  by Corollary \ref{Aldensfavourite},
  $$
   v = F(\Phi)=\lim_{\lambda\to 0} \lambda^{\frac{3-d}{2}} E_\lambda(\Phi)
  $$
  for some $\Phi \in \mathcal{H}_1^p(\sphere)$. Now write $\Phi = \der r \wedge \Phi_- + \iota_{\der r} \Phi_+$ and note that $\Phi_\pm$ is a linear combination
  of spherical harmonics of degree $0$ and $2$:
  $$
   \Phi_\pm = \Phi_{\pm,0} +  \Phi_{\pm,2}. 
  $$
  Since the $L^2$-harmonic forms are closed and co-closed the expansions of Theorem \ref{maintheorem1} give
  $\der E_\lambda(\Phi_{\pm,2}) =O(\lambda^\frac{d+3}{2})$ and $\delta E_\lambda(\Phi_{\pm,2}) =O(\lambda^\frac{d+3}{2})$.
  By Lemma \ref{lemmafourone} the limits 
  $$
   v_\pm = \rmi \lim_{\lambda\to 0} \lambda^{\frac{1-d}{2}} E_\lambda(\Phi_{\pm,0})
  $$
  exist and by construction $v = \der v_- - \delta v_+$. Moreover, $v_\pm$ satisfy relative boundary conditions. By Proposition \ref{superprop} we have that
  $\der v_-$ and $\delta v_+$ are $L^2$-harmonic. Therefore $v_-$ represents a cohomology class the kernel of the map $H^p_0(M,\partial \calO) \to H^p(M,\partial \calO)$. Thus, $\der v_- =0$. Similarly,  applying the Hodge star operator, $*\delta v_+$ represents a cohomology class the kernel of the map 
  $H^{d-p}_0(M) \to H^{d-p}(M)$. Hence, again $\delta v_+=0$.
 \end{proof}

\begin{proposition} \label{superprop}
 Assume $d \geq 3$ and 
 suppose that $\Phi \in \mathcal{H}^p_0(\sphere)$ is a spherical harmonic of degree $0$, i.e. $\Phi$ is independent of $\theta$.  The limit
 $$
  G(\Phi) = \lim_{\lambda \to 0} \lambda^{-\frac{d-1}{2}} E_\lambda(\Phi)
 $$
 exists. Moreover, $G(\Phi)$ is harmonic, satisfies relative boundary conditions, and  we have
 $$
   G(\Phi)(r, \theta) = C_{d,0} \Phi + O\left(\frac{1}{r^{d-2}}\right)
 $$
for sufficiently large $r$. In particular, we have that $d G \in L^2$.
\end{proposition}
\begin{proof}
 The function $G(\Phi)$ exists because $E_\lambda(\Phi)$ is Hahn holomorphic and of order $\lambda^{\frac{d-1}{2}}$.
 By construction this function is harmonic and satisfies relative boundary conditions. The asymptotic behaviour
  follows from Lemma \ref{superlemma2}, as for large $\theta$ we have $E_\lambda(\Phi)(r,\theta) = \tilde j_\lambda(\Phi)(r,\theta) + \tilde h^{(1)}_\lambda(A_\lambda \Phi)$. We have 
 $$
  \lim_{\lambda \to 0}  \lambda^{-\frac{d-1}{2}} \tilde h^{(1)}_\lambda(\langle A_\lambda \Phi, \Phi_\nu \rangle   \Phi_{\nu}) = C_\nu\frac{1}{r^{\ell_\nu+d-2}} \Phi_\nu.
 $$

 Thus,  $ \lim\limits_{\lambda \to 0}   \lambda^{-\frac{d-1}{2}} \tilde j_\lambda(\Phi) = C_{d,0} \Phi$
 implies the asymptotic form.
\end{proof}

 \begin{proposition} \label{prop5.5}
   If $d \geq 3$, we have that $\mathcal{K}^{1,1}_{rel}(M)$ is canonically isomorphic to the kernel of the map $H_0^1(M,\partial \calO) \to H^1(M,\partial \calO)$, i.e.
   $\mathcal{K}^{1,1}_{rel}(M) =\{0 \}$ if $\mathcal{O}= \emptyset$ and
   $\mathcal{K}^{1,1}_{rel}(M) \cong \R$ if $\mathcal{O} \not= \emptyset$. In the latter case $\mathcal{K}^{1,1}_{rel}(M)$ is generated
   by $\der G(1)$.
 \end{proposition}
 \begin{proof}
  The map 
  $$\mathcal{H}_1^1(\sphere) \to \mathcal{K}^{1,1}_{rel}(M), \quad \Phi \mapsto  \rmi \lim_{\lambda\to 0} \lambda^{\frac{3-d}{2}} E_\lambda(\Phi_\pm)$$
  is onto. Let $v$ be the image of some $\Phi \in \mathcal{H}_1^1(\sphere)$.
  As in the proof of the previous proposition write $\Phi = \der r \wedge \Phi_- + \iota_{\der r} \Phi_+$ for $\Phi \in \mathcal{H}_1(\sphere)$.
  Then, $v_\pm=\rmi \lim_{\lambda\to 0} \lambda^{\frac{3-d}{2}} E_\lambda(\Phi_\pm)$ and as above $v = d v_- - \delta v_+$, where $v_\pm = G(\Psi_\pm)$ and $\Psi_\pm$ is the degree $0$
  part of $\Phi_\pm$. If $\partial \mathcal{O} = \emptyset$ then $v_-$ is proportional to the constant function and therefore $\der v_-=0$. If $\mathcal{O} \not= \emptyset$ then $d v_- \not= 0$. By construction $\der v_-$ satisfies harmonic boundary conditions and is $L^2$-harmonic. We will now show that
  $\delta v_+$ vanishes, thus completing the proof. By construction $* \delta v_+$ is $L^2$-harmonic, satisfies absolute boundary conditions and is exact. 
It represents therefore a class in the kernel of the map 
$H^{d-1}_0(M) \to H^{d-1}(M)$. It follows that $\delta v_+=0$.
 \end{proof}
 
\section{The Birman-Krein formula and expansions of the spectral shift function} \label{BMformsec}

The spectral shift function usually describes the trace of the difference of functions of perturbed and unperturbed operators in scattering theory.
In our setting these operators act on different Hilbert spaces so a suitable domain decomposition is needed. Let $K$ be a compact subset of $M$
such that $M \setminus K$ is isometric to $\R^d \setminus \tilde K$, with $\tilde K$  a compact subset of $\R^d$. Hence, $L^2(M \setminus K)$ is identified
with $L^2(\R^d \setminus \tilde K)$. Let $p$ be the orthogonal projection $L^2(M) \to L^2(M \setminus K)$, and let  $p_0$ be the orthogonal projection
$L^2(\R^d) \to L^2(\R^d \setminus \tilde K)$. The Birman-Krein formula then reads:

\begin{theorem} \label{Beerman}
 Let $f \in C^\infty_0(\R)$ be an even compactly supported smooth function. Then,
 \begin{gather*}
  \Tr \left( (1-p) f(\Delta^{1/2}) (1-p) \right) - \Tr \left( (1-p_0) f(\Delta_0^{1/2}) (1-p_0) \right) \\+ \Tr \left( p f(\Delta^{1/2})p - p_0 f(\Delta_0^{1/2}) p_0 \right) = (\beta_p + \beta_{\mathrm{res}}) f(0) \\
  +\frac{1}{2 \pi \rmi }\int_{0}^\infty f(\lambda) \Tr_{L^2(\sphere)}\left ( S^*(\lambda) S^{\;\prime}(\lambda) \right) \der \lambda.
 \end{gather*}
 Here $\beta_p = \dim \mathcal{H}^p_{rel}(M)$ is the dimension of the $L^2$-kernel of $\Delta_{p,rel}$, i.e. the relative $L^2$-Betti number of $M$. The number 
 $\beta_{\mathrm{res}}$ is zero unless $d=2, p=1, \partial \calO \not= \emptyset$, in which case we have $\beta_{\mathrm{res}}=1$.
 \end{theorem}

As similar statement involving domain decomposition was proved by Christiansen (\cite{christiansen1998spectral}) in the special case of functions and also in the more general setting of scattering manifolds in \cite{christiansen1999weyl}. The case of obstacle scattering in $\R^3$ for functions is discussed in great detail in Taylor's book (\cite[Chapter 9]{MR2743652}). 
We are considering differential forms and also allow the test function $f$ to be non-zero in any dimension. For the sake of completeness we provide a detailed proof in our setting\begin{proof}
Since both sides are distributions in $\mathcal{D}'(\R)$ it suffices to check this for a dense class of functions. We will thus assume here that $f$ is real analytic in some neighborhood of zero, depending on $f$, and real valued on the real line.
By Theorem \ref{kernelconv}, the operators $f(\Delta^{1/2})$ and $f(\Delta^{1/2}_0)$ have smooth integral kernels $k(x,y)$ and $k_0(x,y)$ respectively.
We define the family $(k_\nu)_\nu$ of smooth kernels $k_\nu \in C^\infty(M \times M; \Lambda^p T^*M \boxtimes (\Lambda^p T^*M)^*)$ by
\begin{gather*}
 k_\nu(x,y) =  \frac{1}{4\pi}   \int_{-\infty}^\infty f(\lambda) E_{\lambda}(\Phi_\nu)(x) \otimes E_{\lambda}(y)(\Phi_\nu)^*\,\der\lambda. 
\end{gather*}
In the same way we construct $k_{0,\nu} \in C^\infty(\R^d \times \R^d; \Lambda^p T^*\R^d \boxtimes (\Lambda^p T^* \R^d)^*)$ for $f(\Delta^{1/2}_0)$.
 By Theorem  \ref{kernelconv}   and Mercer's Theorem $(1-p) f(\Delta^{1/2}) (1-p)$ and $(1-p_0) f(\Delta_0^{1/2}) (1-p_0)$ are trace-class and their trace is given by  \begin{gather*}
  \Tr \left( (1-p) f(\Delta^{1/2}) (1-p) \right) = \int_K \tr\, k(x,x) \der x=\beta_p f(0)+\int_K \sum_\nu  \tr\, k_\nu(x,x) \der x,\\
  \Tr \left( (1-p_0) f(\Delta_0^{1/2}) (1-p_0) \right) = \int_{\tilde K}  \tr \, k_{0}(x,x) \der x=\sum_{\nu} \int_{\tilde K}\sum_\nu   \tr \, k_{0,\nu}(x,x) \der x,
 \end{gather*}
 where $\tr$ denotes the pointwise trace on the fibre $\mathrm{End}(\Lambda^p T^*_x M)$ of $\Lambda^p T^*M \boxtimes (\Lambda^p T^*M)^*$ at the point $(x,x)$.
 We have used that fact that $f$ is even here and Remark \ref{awesomeremark}.
  Now let $p_R$ be the indicator function of a large ball $B_R$ such that $\tilde K \subset B_R$. Then, again by Mercer's theorem:
 $$
  \Tr \left( p_R \left (  p f(\Delta^{1/2})p - p_0 f(\Delta_0^{1/2}) p_0  \right)p_R \right) = \int_{B_R \setminus \tilde K}  \tr\, (k(x,x)  - k_0(x,x)) \der x.
 $$ 
 By Corollary \ref{tracecor} the operator  $p f(\Delta^{1/2})p - p_0 f(\Delta_0^{1/2}) p_0$ is trace-class and by dominated convergence theorem applied to the trace we obtain
 $$
  \Tr \left( p f(\Delta^{1/2})p - p_0 f(\Delta_0^{1/2}) p_0 \right) = \lim_{R \to \infty}  \int_{B_R \setminus \tilde K}  \tr\, (k(x,x)  - k_0(x,x)) \der x.
 $$
Collecting everything we have
 \begin{gather*}
   -\beta_p f(0) +\Tr \left( (1-p) f(\Delta^{1/2}) (1-p) \right) - \Tr \left( (1-p_0) f(\Delta_0^{1/2}) (1-p_0) \right) \\+ \Tr \left( p f(\Delta^{1/2})p - p_0 f(\Delta_0^{1/2}) p_0 \right)=  
     \lim_{R \to \infty}  \sum_{\nu}\left(  \int_{M_R}  \tr\; k_\nu(x,x)  \der x -  \int_{B_R}  \tr\;  k_{0,\nu}(x,x)  \der x \right),
 \end{gather*}
 where $M_R$ is obtained from $M$ by removing the subset identified with $\R^d \setminus B_R$. 
  It is common to use the following trick to compute these integrals. Since $(\Delta - \lambda^2) E_\lambda(\Phi_\nu) =0$, differentiation in $\lambda$ yields $(\Delta - \lambda^2) E^{\;\prime}_\lambda(\Phi_\nu)  = 2 \lambda E_\lambda(\Phi_\nu)$, where $E^{\;\prime}_\lambda(\Phi_\nu) = \frac{d}{d\lambda} E_\lambda(\Phi_\nu)$. 
  Hence, integration by parts and the general bounds on $E_\lambda$ give
 \begin{gather*}
  \int_{M_R} \tr \;k_\nu(x,x) \der x= \lim_{\epsilon \to 0_+}\frac{1}{2 \pi} \int_{M_R}  \int_{\epsilon}^\infty f(\lambda)  | E_\lambda(\Phi_\nu) |^2 \der \lambda \der x \\ =  \lim_{\epsilon \to 0_+}\frac{1}{4 \pi }  \int_{M_R} \int_{\epsilon}^\infty \frac{1}{\lambda} f(\lambda)  \langle  (\Delta - \lambda^2)  E^{\;\prime}_\lambda(\Phi_\nu),E_\lambda(\Phi_\nu) \rangle  \der \lambda \der x \\= \lim_{\epsilon \to 0_+}
   \frac{1}{4 \pi}  \int_{\epsilon}^\infty f(\lambda) \frac{1}{\lambda} b_R(E^{\;\prime}_\lambda(\Phi_\nu),  E_\lambda(\Phi_\nu) ) \der \lambda.
 \end{gather*}
 Here $b_R(F,G)$ is the boundary pairing of forms $F$ and $G$ and defined by
 $$
  b_R(F,G) = \int_{\partial M_R} \langle F(x) ,\partial_n G(x) \rangle  - \langle \partial_n F(x),G(x) \rangle d\sigma(x),
 $$
 where $d\sigma$ is the surface measure of the sphere $\partial M_R$.
 We conclude that
 \begin{gather*}
 \int_{M_R} \left(
  \tr \;k_\nu(x,x) -\tr \;k_{0,\nu}(x,x) \right) \der x =  \lim_{\epsilon \to 0_+} \frac{1}{4 \pi} \int_{\epsilon}^\infty \frac{1}{\lambda} f(\lambda) \eta_{\nu,R}(\lambda) \der \lambda, \\
  \eta_{\nu,R}(\lambda) = b_R\left(\frac{d}{d \lambda}\left( \tilde j_\lambda(\Phi_\nu) + \tilde h^{(1)}_\lambda(A_\lambda \Phi_\nu) \right), \tilde j_{\overline \lambda}(\Phi_\nu) + \tilde h^{(1)}_{\overline \lambda}(A_\lambda \Phi_\nu) \right) \\-  b_R\left( \frac{d}{d \lambda}\left( \tilde j_\lambda(\Phi_\nu) \right) , \tilde j_{\overline \lambda}(\Phi_\nu)\right)=
   b_R\left(\frac{d}{d \lambda}\left( \tilde j_\lambda(\Phi_\nu) + \tilde h^{(1)}_\lambda(A_\lambda \Phi_\nu) \right), \tilde h^{(1)}_{\overline \lambda}(A_\lambda \Phi_\nu) \right)\\
   +b_R\left(\frac{d}{d \lambda}\left(\tilde h^{(1)}_\lambda(A_\lambda \Phi_\nu) \right), \tilde j_{\overline \lambda}(\Phi_\nu)  \right).
 \end{gather*}
 We have $\frac{d}{d\lambda} \left(\tilde h^{(1)}_\lambda(A_\lambda \Phi_\nu) \right) =  \tilde h^{(1)}_\lambda(A^{\prime}_\lambda \Phi_\nu) +\tilde h^{(1)\prime}_\lambda(A_\lambda \Phi_\nu)$.
 Unitarity of $S(\lambda)$ implies the identity $A(\lambda) + A^*(\overline{\lambda}) + A^*(\overline{\lambda}) A(\lambda) =0$, and therefore we have
 \begin{gather*}
   b_R\left(\tilde h^{(1)\prime}_\lambda(A_\lambda \Phi_\nu),  \tilde h^{(1)}_{\overline \lambda}(A_{\overline \lambda} \Phi_\nu)\right) + b_R\left(\tilde h^{(1)\prime}_\lambda( \Phi_\nu),  \tilde h^{(1)}_{\overline \lambda}(A_{\overline \lambda} \Phi_\nu)\right)  \\+b_R\left(\tilde h^{(1)\prime}_\lambda(A_\lambda \Phi_\nu),  \tilde h^{(1)}_{\overline \lambda}(\Phi_\nu)\right) =0.
 \end{gather*}
  Using $b_{R}( \tilde h^{(1)}_\lambda(\Phi_\nu), \tilde h^{(2)}_{\overline{\lambda}}(\Phi_\nu))=0$ and $\tilde j_\lambda(\Phi_\nu)=  \tilde h^{(1)}_\lambda(\Phi_\nu)+ \tilde h^{(2)}_\lambda(\Phi_\nu)$ one obtains
  \begin{gather*}
  \eta_{\nu,R}(\lambda)=\left( \langle A_\lambda^{\prime} \Phi_\nu, A_{\overline \lambda} \Phi_\nu \rangle + \langle A_\lambda^{\prime} \Phi_\nu,\Phi_\nu\rangle \right) b_{R}( \tilde h^{(1)}_\lambda(\Phi_\nu), \tilde h^{(1)}_{\overline \lambda}(\Phi_\nu))\\+
  \langle \Phi_\nu, A_{\overline \lambda} \Phi_\nu\rangle b_{R}( \tilde h^{(2) \prime}_\lambda(\Phi_\nu), \tilde h^{(1)}_{\overline \lambda}(\Phi_\nu))+
  \langle A_\lambda  \Phi_\nu, \Phi_\nu\rangle b_{R}(\tilde h^{(1) \prime}_\lambda(\Phi_\nu), \tilde h^{(2)}_{\overline \lambda}(\Phi_\nu))
   \end{gather*}
 The term $b_{R}( \tilde h^{(1)}_\lambda(\Phi_\nu), \tilde h^{(1)}_{\overline{\lambda}}(\Phi_\nu))$ is independent of $R$ and is actually given in terms of a Wronskian between Hankel functions. One obtains
 $$
 b_{R}( \tilde h^{(1)}_\lambda(\Phi_\nu), \tilde h^{(1)}_{\overline{\lambda}}(\Phi_\nu)) = -2 \rmi \lambda.
 $$
The terms 
 $ \langle  A_\lambda \Phi_\nu, \Phi_\nu\rangle b_{R}(\tilde h^{(1) \prime}_\lambda(\Phi_\nu), \tilde h^{(2)}_{\overline \lambda}(\Phi_\nu))$ and $ \langle \Phi_\nu, A_{\overline \lambda} \Phi_\nu\rangle b_{R}(\tilde h^{(2)}_\lambda(\Phi_\nu), \tilde h^{(1) \prime}_{\overline \lambda}(\Phi_\nu))$ are complex conjugates of each other for positive $\lambda$. As long as $\lambda>0$ their sum is therefore
 $$
 2 \Re\left(\langle  A_\lambda \Phi_\nu, \Phi_\nu\rangle b_{R}(\tilde h^{(1)}_\lambda(\Phi_\nu), \tilde h^{(2) \prime}_{\overline \lambda}(\Phi_\nu))\right) = 2 \Re(g(\lambda)).
 $$ Using the functional equation for $A_\lambda$ and 
 one shows that the function
 $$
   g(\lambda)=\langle A_\lambda  \Phi_\nu, \Phi_\nu\rangle b_{R}(\tilde h^{(1) \prime}_\lambda(\Phi_\nu), \tilde h^{(2)}_{\overline \lambda}(\Phi_\nu))
 $$
 satisfies $\overline{g(-\lambda)} = - g(\lambda)$. Thus, $\Re(g(\lambda))$ is odd.
  Using Lemma \ref{superhankel} below and the bound $| \langle A_{{\lambda}} \Phi_\nu , \Phi_\nu \rangle | = O(\frac{\lambda^{2\ell_\nu+d-4}}{(\ell_\nu !)^2})$ (Lemma \ref{Abound}),  one obtains
 \begin{gather*}
    \lim_{\epsilon \to 0_+} \frac{1}{4 \pi} \int_{\epsilon}^\infty \frac{1}{\lambda} f(\lambda) 2 \Re\left( \langle A_\lambda  \Phi_\nu, \Phi_\nu\rangle b_{R}(\tilde h^{(1) \prime}_\lambda(\Phi_\nu), \tilde h^{(2)}_{\overline \lambda}(\Phi_\nu))\right)  \der \lambda \\
   =\Re \lim_{\epsilon \to 0_+} \frac{1}{4 \pi} \int_{\R_\epsilon} \frac{1}{\lambda} f(\lambda) \langle A_\lambda  \Phi_\nu, \Phi_\nu\rangle b_{R}(\tilde h^{(1) \prime}_\lambda(\Phi_\nu), \tilde h^{(2)}_{\overline \lambda}(\Phi_\nu))  \der \lambda \\
   = \Re(b_\nu)  f(0) + \Re(c_\nu) f(0) R^{-2\ell-d+4} + O((1+\ell_\nu^2)R^{-1} e^{c \ell_\nu}. \frac{\lambda^{2\ell_\nu+d-4}}{(\ell_\nu !)^2}),
 \end{gather*}
  where $R_\epsilon = \R \setminus [-\epsilon, \epsilon]$.
 Here $b_\nu$ are non-zero only in finitely many cases (in fact only when $d=2,\ell=1$), and 
 $c_\nu$ computes to $\frac{2}{2 \ell_\nu + d -4} \sum_{j=1}^N |a_j(\Phi_\nu)|^2$ in case $2 \ell_\nu + d -4>0$ using the expansion of the scattering amplitude from Theorem \ref{scatterexp}. Since the $u_j$ are square integrable one obtains the bound $\sum_\nu  \sum_{j=1}^N R^{-2\ell} |a_j(\Phi_\nu)|^2 < \infty$.
 All together we get the estimate
 $$
   \sum_\nu \lim_{\epsilon \to 0_+} \frac{1}{4 \pi} \int_{\R_\epsilon} \frac{1}{\lambda} f(\lambda) \langle A_{\lambda}  \Phi_\nu, \Phi_\nu\rangle b_{R}(\tilde h^{(1) \prime}_\lambda(\Phi_\nu), \tilde h^{(2)}_{\overline{\lambda}}(\Phi_\nu))  \der \lambda = f(0) \sum_\nu b_\nu + O(R^{-1}).
 $$
Unless $d=2, \ell=1,p=1, \partial \calO \not= \emptyset$ the bounds on the scattering amplitude imply that $b_\nu=0$ and this implies the theorem in these cases. It remains to compute the contribution from $b_\nu$ when $d=2, \ell=1,p=1, \partial \calO \not= \emptyset$. By the bounds on $\langle A_\lambda \Phi_\nu, \Phi_\nu \rangle$ we obtain a contribution only when $\ell_\nu=1$, and in this case $\langle A_\lambda \Phi_\nu, \Phi_\nu \rangle= -\rmi \, \pi (-\log \lambda)^{-1} + o( (-\log \lambda)^{-1})$. Lemma \ref{superhankel} then gives a contribution of  $\sum_\nu b_\nu = 1$. This concludes the proof.
\end{proof}

It is easy to see that the function $b_{R}(\tilde h^{(1) \prime}_\lambda(\Phi_\nu), \tilde h^{(2)}_{\overline{\lambda}}(\Phi_\nu))$ depends only on $\ell_\nu$ and $\lambda R$.
We can therefore define $H_\ell$ by  $H_\ell(\lambda R)= b_{R}(\tilde h^{(1) \prime}_\lambda(\Phi_\nu), \tilde h^{(2)}_{\overline{\lambda}}(\Phi_\nu))$. 

\begin{lemma} \label{superhankel}
 Let as before $H_\ell(\lambda R):= b_{R}(\tilde h^{(1) \prime}_\lambda(\Phi_\nu), \tilde h^{(2)}_{\overline{\lambda}}(\Phi_\nu))$. 
 Suppose that $f \in C^\infty_0(\R)$ is supported in $(-T,T)$ and extends holomorphically near zero to a function analytic in a neighborhood of the closed ball $\overline{B_\delta(0)}$.
 Let $K:=[-T,T] \times [0, \delta_1]$ be any rectangle with $\delta_1>0$.
 Then for every $k \in \mathbb{N}$ there exists a constant $C_k>0$, independent of $\nu$ such that for any $R>\delta^{-1}$ and any $g$ that is holomorphic in the interior of $K$ and continuous on $K$ we have the following estimates for $R>1$;
   \begin{itemize}
  \item if $d=2$ and $\ell_\nu=0$  then 
  $$
 | \lim_{\epsilon \to 0_+} \int_{\R_\epsilon} \frac{1}{\lambda} f(\lambda) g(\lambda) H_\ell(\lambda R) \der \lambda - (-2 \, \rmi \,g(0))  f(0) | \leq  \frac{C_k}{R^k} \sup\limits_{x \in K} |g(x)| $$
  \item if $d=2$ and $\ell_\nu=1$ and $g(\lambda) = \frac{a}{-\log \lambda}+ o(\frac{1}{-\log \lambda})$ for $|\lambda|<1/2$ then 
  $$
  | \lim_{\epsilon \to 0_+} \int_{\R_\epsilon} \frac{1}{\lambda} f(\lambda) g(\lambda) H_\ell(\lambda R) \der \lambda -(4 \, \rmi \, a) f(0) | \leq \frac{C_k}{R^k} \sup\limits_{x \in K} |g(x)|  $$
  \item  if $d=3$ and $\ell_\nu=0$ then 
   $$
  | \lim_{\epsilon \to 0_+} \int_{\R_\epsilon}\frac{1}{\lambda} f(\lambda) g(\lambda) H_\ell(\lambda R) \der \lambda - (-\pi g(0)) f(0)| \leq\frac{C_k}{R^k} \sup\limits_{x \in K} |g(x)|  .
 $$
 \item if $2\ell +(d-4)>0$ and $g(\lambda) = a \lambda^{2 \ell +  d - 4}  +o(\lambda^{-2 \ell -  d + 4}) $  for $|\lambda|<1$ then
  \begin{gather*}
|\lim_{\epsilon \to 0_+} \int_{\R_\epsilon} \frac{1}{\lambda} f(\lambda) g(\lambda) H_\ell(\lambda R) \der \lambda - a \gamma_{d,\ell} R^{-2\ell- d+4}| \\ \leq \frac{C_k (1+\ell)^2}{R^k} \sup\limits_{x \in K} |g(x)|  e^{2(1+\frac{d}{2})^2 R^{-1} \delta^{-1}  (1+\ell)^2},
 \end{gather*}
 where $\gamma_{d,\ell} =  \rmi\; 2^{2\ell+d-3} \Gamma(\ell +\frac{d-2}{2}) \Gamma(\ell +\frac{d-5}{2})$.  
 \end{itemize}
\end{lemma}
\begin{proof}
We choose a compactly supported almost analytic extension 
$\tilde f \in C^\infty_0(\R^2)$ of $f$, i.e. $\delbar f = O (\mathrm{Im}(\lambda)^N)$ for any $N>0$. Since $f$ was assumed to be analytic near zero we can arrange this almost analytic extension to be analytic in $B_\delta$ and supported in the interior of $K$.
Then, by Stokes' formula
 \begin{gather*}
  \lim_{\epsilon \to 0_+} \int_{\R_\epsilon} \frac{1}{\lambda}f(\lambda) g(\lambda) H_\ell(\lambda R) \der \lambda -  \lim_{\epsilon \to 0_+} \int_{C(\epsilon)} \frac{1}{\lambda}f(\lambda) g(\lambda) H_\ell(\lambda R) \der \lambda \\= 2 \rmi \int_{K \setminus B_\delta} (\delbar \tilde f)(x,y) g(x+ \rmi y) \frac{1}{x+ \rmi y} H_\ell((x+\rmi y) R) \der x \der y,
 \end{gather*}
 where $C(\epsilon)$ is the semi-circle in the upper half plane centered at zero of radius $\epsilon$. 
 Note that $\delbar f$ vanishes on $B_\delta(0)$ and we therefore can integrate over the complement of $B_\delta(0)$.
 The function $H_\ell$ can be expressed explicitly as
 \begin{gather*}
  H_\ell(\lambda)=
  \frac{1}{8} \pi  \lambda  \left( \right.\lambda 
   H_{\frac{d}{2}+\ell-2}^{(1)}(\lambda ){}^2-2
   (H_{\frac{d}{2}+\ell-1}^{(1)}(\lambda )+\lambda 
   H_{\frac{d}{2}+\ell}^{(1)}(\lambda ))
   H_{\frac{d}{2}+\ell-2}^{(1)}(\lambda )\\+2 \lambda 
   H_{\frac{d}{2}+\ell-1}^{(1)}(\lambda ){}^2+\lambda 
   H_{\frac{d}{2}+\ell}^{(1)}(\lambda ){}^2-\lambda 
   H_{\frac{d}{2}+\ell-3}^{(1)}(\lambda )
   H_{\frac{d}{2}+\ell-1}^{(1)}(\lambda
   )\\+H_{\frac{d}{2}+\ell-1}^{(1)}(\lambda ) (2
   H_{\frac{d}{2}+\ell}^{(1)}(\lambda )-\lambda 
   H_{\frac{d}{2}+\ell+1}^{(1)}(\lambda )) \left. \right)
 \end{gather*}
 For $\Im{\lambda}\geq0$ and $|\lambda|>0$ we have the following asymptotics for the Hankel function (\cite[10.17.13-10.17.15]{olver2010nist})
 $$
  H^{(1)}_{k}(\lambda) = \left(\frac{2}{\pi \lambda}\right)^{\frac{1}{2}}e^{
\rmi(\lambda -\frac{k}{2} \pi- \frac{\pi}{4})}\left(1+R_{1}^{+}(k,\lambda)\right),
 $$
 with
 $$
  | R_{1}^{+}(k,\lambda) | \leq  | k^2-\frac{1}{4} | |\lambda|^{-1} e^{(| k^2-\frac{1}{4} | |\lambda|^{-1})}.
 $$
 This gives the uniform bound for $| R \lambda | >1$ and $|\lambda|>\delta$
 \begin{gather*}
  | H_{\ell}(R \lambda) | \leq C (1+\ell)^2 e^{-2 R \Im{\lambda}} e^{2(1+\frac{d}{2})^2R^{-1} | \lambda| ^{-1}  (1+\ell)^2} \\ \leq C (1+\ell)^2  e^{- 2 R \Im{\lambda}} e^{2(1+\frac{d}{2})^2 R^{-1} \delta^{-1}  (1+\ell)^2},
  \end{gather*}
  where $C$ depends on $d$ only.
  Now there is a constant $\tilde C$ such that $\delbar \tilde f (x+\rmi y) \leq \tilde C \, y^{k}$. Integrating this gives for $R>1$ the bound
   \begin{gather*}
  | \lim_{\epsilon \to 0_+} \int_{\R_\epsilon} \frac{1}{\lambda}f(\lambda) g(\lambda) H_\ell(\lambda R) \der \lambda -  \lim_{\epsilon \to 0_+} \int_{C_\epsilon} \frac{1}{\lambda}f(\lambda) g(\lambda) H_\ell(\lambda R) \der \lambda | \\ \leq C_k R^{-k} (1+\ell)^2 \sup\limits_{x \in K} |g(x)|  e^{2(1+\frac{d}{2})^2 R^{-1} \delta^{-1}  (1+\ell)^2}.
 \end{gather*}
 Using the asymptotics of the Hankel function in case $2\ell +(d-4)>0$ as $x \to 0$ in the upper half plane
 $$
  H_\ell(x) = O( x^{-2 \ell -  d + 4} ),
 $$
 and in this case the integral over the circle converges to zero as $\epsilon \to 0$.
 In case $d=3, \ell =0$ one can compute $H_\ell(x)=\rmi e^{2 \rmi x}$ and therefore $H_\ell(x)= \rmi + O(x)$. If $d=2, \ell=0$ one has
 $H_\ell(x)=-\frac{2}{\pi} + O(x^2)$, and if $d=2, \ell=1$ we have $H_\ell(x)= -\frac{4}{\pi} \log{x} +O(1)$. This gives the claimed values.
\end{proof}

The Birman-Krein formula can also be stated, using integration by parts, as
  \begin{gather}
  \Tr \left( (1-p) f(\Delta^{1/2}) (1-p) \right) - \Tr \left( (1-p_0) f(\Delta_0^{1/2}) (1-p_0) \right) \nonumber \\+ \Tr \left( p f(\Delta^{1/2})p - p_0 f(\Delta_0^{1/2}) p_0 \right) = 
  -\int f'(\lambda) \xi(\lambda^2) \der \lambda, \label{Birmanv2}
 \end{gather}
where $\xi \in L^1_\loc(\R)$ is the spectral shift function defined by
$$
 \xi(\mu) = \left \{ \begin{matrix} 0 & \mu < 0,\\ \left(\beta_p + \beta_{res} \right) +  \frac{1}{2 \pi \rmi }  \int_0^{\sqrt{\mu}}\tr\left( S^*(\lambda) S'(\lambda)\right) \der \lambda& \mu \geq 0. \end{matrix} \right.
$$

\begin{rem}
 Our proof is along the same lines as similar computations involving the Maass-Selberg relations, with the additional complication of interchanging the 
 summation over $\nu$ and the limit $R \to \infty$ that poses a problem when the test function has zero contained in its support. 
 In case of potential scattering in dimension three one can also use this method and the Lemma above to compute the contribution $\frac{1}{2}$ of a possible zero 
 resonance state. We note here that the Maass-Selberg trick was also used by Parnovski in \cite{parnovski2000scattering} in the context of manifolds with conical ends to compute the asymptotics of the spectral function and hence the spectral shift function. His method also applies to our situation with obvious changes and one therefore has a Weyl law for the spectral shift function. Therefore version \eqref{Birmanv2} of the Birman-Krein formula also holds for even Schwartz functions. The Weyl law for the scattering phase in case $p=0$ was first proved for obstacle scattering by Majda and Ralston \cite{Majda_1978} for convex domains and finally for smooth domains by Melrose \cite{melrose1988weyl}.
 \end{rem}

\section{Proofs of the main theorems} \label{sevenandfinal}
For the purposes of presentation we have stated our main theorems into Section \ref{theostate}. In this section we summarise how they follow from the statements in the body of the text.
\vspace{0.2cm}\\
\noindent
{\sl Proof of Theorem \ref{maintheorem1}:}
The expansions were shown in \eqref{expd5th1}, \eqref{expd3th1}, \eqref{expd6th1}, and \eqref{expd4th1}. That $P^{(1)}=0$ unless $p=1$ and $\partial \calO \not=\emptyset$ follows immediately from Prop \ref{prop5.3} and Prop \ref{prop5.5}. 
\vspace{0.2cm}\\
{\sl Proof of Theorem \ref{maintheorem2}:} The resolvent expansion was shown in Section \ref{doddanal}. In particular $B_{-1}$ was computed to be $0$ in dimension greater $5$, as a result of the analysis of equation \eqref{belowfive}. Propositions \ref{prop37} and \ref{prop38} show the result in dimension $3$ and $5$.
\vspace{0.2cm}\\
{\sl Proof of Theorem \ref{maintheorem3}:} The resolvent expansion was shown in Section \ref{devenanal}. The theorem follow directly from \ref{resolventexpdg4}.
\vspace{0.2cm}\\
{\sl Proof of Theorem \ref{maintheorem22}:} 
The theorem follows from the discussion in Section \ref{d2anal}. The particular form of the resolvent is contained in Lemma \ref{abovelemma3} and the coefficients were computed in Propositions \ref{plprop}, \ref{prop04}, and \ref{connectionprop}.
\vspace{0.2cm}\\
{\sl Proof of Theorem \ref{maind2ge}:} This theorem the result of a combination of Corollary \ref{cor327} and Corollary \ref{cor329}.
 \vspace{0.2cm}\\
{\sl Proof of Theorem \ref{maintheorem5}:} This theorem is the result of combination of Theorem \ref{resexpd2p1} and Proposition \ref{resexpd2p1no}.
  \vspace{0.2cm}\\
{\sl Proof of Theorem \ref{scatterexp}:} This was proved in Section \ref{ScatterSectionExp} and follows directly by applying Theorem \ref{AEexpansion} to 
the expansions of the generalised eigenfunctions as stated in Theorem \ref{maintheorem1}.
 \vspace{0.2cm}\\
{\sl Proof of Theorem \ref{maintheorem6}:} Was proved in Section \ref{ScatterSectionExp} and follows directly by applying Theorem \ref{AEexpansion} to 
the expansions of the generalised eigenfunctions as stated in Theorem \ref{maind2ge}.
 \vspace{0.2cm}\\
 {\sl Proof of Theorem \ref{maincohomo}:} This is a combination of Propositions  \ref{prop5.3} and \ref{prop5.5}.
  \vspace{0.2cm}\\
 {\sl Proof of Theorem \ref{maintheorem7}:} The Birman-Krein formula shows that the relation between the spectral shift function and $\eta$ is as claimed. Moreover, we have 
 $$\eta(\lambda^2)=\frac{1}{2 \pi \rmi}\log \det S(\lambda) = \frac{1}{2 \pi \rmi}\log \det \left( 1 + A(\lambda) \right),$$
 where the branch of the logarithm is chosen continuous on the positive real line.
Suppose that $p \geq 1$. Then $P^{(1)} =0$. Recall from Prop. \ref{superprop} that for $|\lambda|<1$ we have $E_\lambda(\Phi_\nu) = G(\Phi_\nu) \lambda^{\frac{d-1}{2}}+ o(\lambda^{\frac{d-1}{2}})$ where $dG(\Phi_\nu)  \in L^2$ and $G(\Phi_\nu) = C_{d,0} \Phi_\nu + O(\frac{1}{r^{d-2}})$ as $r \to \infty$. In particular, 
$dG(\Phi_\nu)$ is a trivial class in cohomology and from the discussion in  Section \ref{maincohomo} we conclude that $dG(\Phi_\nu)=0$. In particular, by Corollary \ref{supercor}, this implies that
there is no $\ell=0$ term in the multipole expansion of $G(\Phi_\nu)$ and since a term of order $\lambda^{d-2}$ in the expansion of 
$\langle A_\lambda \Phi_\nu, \Phi_\nu \rangle$ would give rise to such a term, we must have $\langle A_\lambda \Phi_\nu, \Phi_\nu \rangle= o(\lambda^{d-2})$.
This means in odd dimensions $\langle A_\lambda \Phi_\nu, \Phi_\nu \rangle = O(\lambda^{d-1})$ and in even dimensions $\langle A_\lambda \Phi_\nu, \Phi_\nu \rangle = O(\frac{\lambda^{d-2}}{-\log \lambda})$. Now simply note that, using $\| A_\lambda \|_{1} = O(\lambda^{d-2})$, we have
$$
 \log \det (1 +A_\lambda) = \tr(A_\lambda) + O(\lambda^{2d-4}).
$$
Then the leading order terms consist solely of the $\ell_\nu=1$ contributions.
The expansions \ref{scatterexp} and \ref{maintheorem6} then imply the claimed formulae when $p \geq 1$. The case $p=0$ follows from the fact that $A_\lambda$ commutes with $\der r \wedge$ and $\iota_{\der r}$. Therefore the expansion for $p=1$ can be derived from the expansion for $p=2$ and for $p=0$.
\appendix
 
\section{Hahn holomorphic and Hahn meromorphic functions} \label{hahnapp}
The theory of Hahn analytic functions was developed in \cite{muller2014theory} in a very general setting.
For the purposes of this paper we restrict our considerations to so-called $z$-$\log(z)$-Hahn
holomorphic functions and refer to these as Hahn holomorphic.
To be more precise, suppose that $\Gamma \subset \mathbb{R}^2$ is a subgroup of $\mathbb{R}^2$. We endow
$\Gamma$ and $\mathbb{R}^2$ with the lexicographical order. Recall that a subset $A \subset \Gamma$
is called well-ordered if any subset of $A$ has a smallest element.
A formal series
$$
  \sum_{(\alpha,\beta) \in \Gamma} a_{\alpha,\beta} \: z^\alpha (-\log{z})^{-\beta}
$$
will be called a Hahn-series if the set of all $(\alpha,\beta) \in \Gamma$ with $a_{\alpha,\beta} \not=0$
is a well ordered subset of $\Gamma$.

In the following let $Z$ be the logarithmic covering surface of the complex plane without the origin.
We will use polar coordinates $(r,\varphi)$ as global coordinates to identify $Z$ as a set with $\Reell_+ \times \Reell$.
Adding a single point $\{0\}$ to $Z$ we obtain a set $Z_0$ and a projection map $\pi: Z_0 \to \Complex$ by extending the covering map $Z \to \Complex\backslash \{0\}$ by sending $0 \in Z_0$ to $0 \in \Complex$.
We endow $Z$ with the covering topology and $Z_0$ with the topology generated by the open sets in $Z$
together with the open discs $D_\epsilon:=\{0\} \cup \{(r,\varphi)\mid 0\le r<\epsilon \}$.
This means a sequence $((r_n,\varphi_n))_n$ converges to zero if and only if $r_n \to 0$. The covering map is continuous
with respect to this topology.
For a point $z \in Z_0$
we denote by $| z |$ its $r$-coordinate and by $\arg(z)$ its $\varphi$ coordinate. We will think of the positive
real axis embedded in $Z$ as the subset $\{ z \mid \arg(z)=0\}$. 
Define the following sectors
$D_\delta^{[\sigma]}=\{z \in Z_0 \mid 0\le  |z| < \delta,\; |\varphi| < \sigma\}$.

In the following fix $\sigma>0$ and a complex Banach space $V$. We say a function $f: D_\delta^{[\sigma]} \to V$ 
is Hahn holomorphic near $0$ in $D_\delta^{[\sigma]}$
if there exists a Hahn series with coefficients in $V$ that converges normally to $f$, i.e. such that
$$
 f(z) = \sum_{(\alpha,\beta) \in \Gamma} a_{\alpha,\beta} \: z^\alpha (-\log{z})^{-\beta},
$$
$$
 \sum_{(\alpha,\beta) \in \Gamma} \| a_{\alpha,\beta} \| \| z^\alpha (-\log{z})^{-\beta}\|_{L^\infty(D_\delta^{[\sigma]})} < \infty,
$$
and there exists a constant $C>0$ such that $a_{\alpha,\beta}=0$ if $-\beta > C \alpha$.
This implies also that $a_{\alpha,\beta}=0$ in case $(\alpha,\beta)<(0,0)$.
As shown in \cite{muller2014theory}, in case $V$ is a Banach algebra the set of Hahn holomorphic functions with values
in $V$ is an algebra. A meromorphic function on $D_\delta^{[\sigma]} \backslash \{0\}$ is called Hahn meromorphic
with values in a Banach space $V$ if near zero it can be written as a quotient of a Hahn holomorphic function with values in $V$
and a Hahn holomorphic function with values $\mathbb{C}$. Note that the algebra of Hahn holomorphic functions with values in
$\mathbb{C}$ is an integral domain and Hahn meromorphic functions with values in $\mathbb{C}$ form a field.
There exists a well defined injective ring homomorphism from the field of Hahn meromorphic functions
into the field of Hahn series. This ring homomorphism associates to each Hahn meromorphic function its Hahn series.
The theory is in large parts very similar to the theory of meromorphic functions. In particular the following very useful
statement holds: if $V$ is a Banach space and $f: D_\delta^{[\sigma]} \to V$ Hahn meromorphic and bounded, then
$f$ is Hahn holomorphic. The main result of \cite{muller2014theory} states that the analytic Fredholm theorem holds for this class of functions.

\section{Resolvent gluing}

In this section all Laplace operators will be operators acting on $p$-forms. As in Section \ref{sec:1.1} we fix a compact subset $K \subset X$
such that $X \setminus K$ is isometric to $\R^d \setminus \tilde K$. We assume here that $K$ is chosen large enough so that $\overline{\calO} \subset \mathrm{int}(K)$. We can also assume without loss of generality that $\tilde K$ is a ball of radius $R>0$ so that $\partial K$ is smooth. 
Then $K \setminus \calO$ is a smooth manifold with boundary $\partial \calO \cup \partial K$. We can also choose a closed set $K_1 \subset X$
that does not intersect $\overline{\calO}$ that is isometric to $\R^d \setminus B_{R-\delta}$ for a suitably small $\delta>0$. Thus, $K_1$ and $K$ cover 
$X$ with a slight overlap.
The resolvent kernel of $\Delta_{rel}$ can be constructed by gluing the free resolvent $R_{0,\lambda}$ of the Laplace operator $\Delta_{0}$ on $\R^d$ 
and the resolvent $R_{D,\lambda}$ of the self-adjoint operator $\Delta_{D}$ constructed by imposing Dirichlet boundary conditions at the additional boundary 
$\partial K$ and relative boundary conditions at $\partial \calO$ on $L^2(K \backslash \calO)$. We consider the free resolvent as a function with values in $\mathcal{B}(H^s_\compp(\R^d),H^{s+2}_\loc(\R^d))$. As such it is defined as a meromorphic function in the complex plane in case $d$ is odd, and it is a Hahn-meromorphic function defined in any sector of the logarithmic cover of the complex plane in case $d$ is even. In this sense we have
\begin{equation} \label{gluing}
 R_\lambda = \left( \chi_1 R_{D,\lambda} \eta_1 +  \chi_2 R_{0,\lambda} \eta_2 \right) (1 + Q_\lambda),
\end{equation}
where, for any $s \in \R$, the family $Q_\lambda$ is a (Hahn)-meromorphic family of operators mapping $H^s_\compp(M;\Lambda^p T^*M)$ to smooth functions
with compact support in a neighborhood of $K$ defined on the same set as the free resolvent. The usual proof of this using meromorphic Fredholm theory carries over to the Hahn-meromorphic setting without change (see \cite{muller2014theory}).
Here $\chi_1,\eta_1,\chi_2,\eta_2$ are suitably chosen cutoff functions such that
\begin{gather*}
 \eta_1 \chi_1 = \eta_1, \quad \eta_2 \chi_2 = \eta_2, \quad  \eta_1 + \eta_2 = 1,\\
\supp \chi_1 \subset K, \quad \supp \chi_2 \subset K_1,\quad
\mathrm{dist}(\supp \chi_1',\eta_1)>0, \quad \mathrm{dist}(\supp \chi_2',\eta_2)>0.
\end{gather*}
It follows that the resolvent $R_\lambda$ admits a (Hahn-)meromorphic extension as a map 
$$H^s_\compp (M;\Lambda^p T^*M) \to H^{s+2}_\loc(M;\Lambda^p T^*M)$$ whenever $R_{0,\lambda}$ does.

The technique of gluing resolvents can be slightly modified to show trace-class properties of differences of operator functions. Let $N \in \mathbb{N}$
and consider the operator $T_\lambda = (\Delta_{rel} +1)^{-N} (\Delta_{rel} - \lambda^2)^{-1}$. Similarly, let
\begin{gather*}
 T_{0,\lambda} = (\Delta_{0} +1)^{-N} (\Delta_{0} - \lambda^2)^{-1},\\
 T_{D,\lambda} = (\Delta_{D} +1)^{-N} (\Delta_{D} - \lambda^2)^{-1},
\end{gather*}
and define
$$
 \tilde T_\lambda =  \chi_1 T_{D,\lambda} \eta_1 +  \chi_2 T_{0,\lambda} \eta_2.
$$
One then computes $(\Delta +1)^N (\Delta - \lambda^2)  \tilde T_\lambda = 1 + Q_1 + (\Delta +1)^N Q_2(\lambda)$, where
\begin{gather*}
 Q_1 = [ (\Delta +1)^N, \chi_1 ] (\Delta_D +1)^{-N} \eta_1 +  [ (\Delta +1)^N, \chi_2 ] (\Delta_0 +1)^{-N} \eta_2, \\
 Q_2(\lambda)=  [\Delta , \chi_1 ] T_{D,\lambda} \eta_1 +   [\Delta , \chi_2 ] T_{0,\lambda} \eta_2.
\end{gather*}
Therefore, one has
$$
  \tilde T_\lambda = T_\lambda  + T_\lambda Q_1 + R_\lambda  Q_2(\lambda).
$$

By the support properties of the cutoff functions $Q_1$ is a smoothing operator mapping to a space of functions with support in a fixed compact set. Hence, $Q_1$ is a trace-class operator. For $Q_2(\lambda)$ we have
\begin{gather*}
 Q_2(\lambda)=  Q_3(\lambda) + Q_4(\lambda),\\
  Q_3(\lambda)=[\Delta , \chi_1 ] (\Delta_D +1)^{-N}  (\Delta_D -\lambda^2)^{-1}\eta_1,\\
  Q_4(\lambda) = [\Delta , \chi_2 ] (\Delta_0 +1)^{-N}  (\Delta_0 -\lambda^2)^{-1} \eta_2.
\end{gather*}
If $2N-1> d$ the operators  $[\Delta , \chi_1 ](\Delta_D +1)^{-N}$ and $[\Delta , \chi_2 ] (\Delta_0 +1)^{-N}$ continuously map $L^2$ into the space $H^{s}_\compp(\Omega)$
for $s = 2N -1$, where $\Omega$ is a bounded subset of $\R^d$. It follows that these operators are trace-class. We now conclude that $Q_3$ and $Q_4$ are trace class for any $\lambda$ in the upper half plane and that
$\|Q_3(\lambda)\|_1 \leq C_3 \frac{1}{\mathrm{Im}(\lambda^2)}, \|Q_4(\lambda)\|_1 \leq C_4 \frac{1}{\mathrm{Im}(\lambda^2)}$. Since $ \| R_\lambda \|_1 \leq \frac{1}{\mathrm{Im}(\lambda^2)}$ we finally have
$$
 \| T_\lambda Q_1 + R_\lambda  Q_2(\lambda) \|_1 \leq C  \frac{1}{|\mathrm{Im}(\lambda^2)|^2}
$$
for some constant $C>0$. We have proved:
\begin{lemma}
 If $N > \frac{d+1}{2}$ then $T_\lambda - \tilde T_\lambda$ is trace-class and there exists $C>0$ such that for the trace norm we have
 $$
   \| T_\lambda - \tilde T_\lambda \|_1  \leq C  \frac{1}{|\mathrm{Im}(\lambda^2)|^2}.
 $$
\end{lemma}

Now let $Z \subset M$ be such that $\eta_1(x) = \chi_1(x)=0$ and $\eta_2(x) = \chi_2(x)=1$ for all $x \in Z$ and let $p$ be the operator of multiplication by the indicator function $\chi_Z$ of $Z$. Then, by the above we have for all $N > \frac{d+1}{2}$ the bound
$$
 \| p (\Delta +1)^{-N} (\Delta - \lambda^2)^{-1} p - p (\Delta_{0} +1)^{-N} (\Delta_{0} - \lambda^2)^{-1} p \|_1 \leq C_N  \frac{1}{|\mathrm{Im}(\lambda^2)|^2}.
$$

\begin{corollary} \label{tracecor}
 For any even function $f \in \mathcal{S}(\R)$ we have that
 $
  p f(\Delta^{1/2}) p - p f(\Delta_0^{1/2}) p 
 $ is trace-class and the mapping
 $
  f \mapsto \Tr \left( p f(\Delta^{1/2}) p - p f(\Delta_0^{1/2}) p \right)
 $
 is a tempered distribution.
\end{corollary}
\begin{proof}
Define $g \in \mathcal{S}(\R)$ by  $g (\lambda) = (1 + \lambda^2)^N f(\lambda)$.
Let $\tilde g$ be an almost analytic extension of $g$ such that $\frac{\partial \tilde g}{\partial \overline z} = O(|\mathrm{Im}(z)|^m)$ for some fixed $m \geq 5$.
Such an almost analytic extension can always be constructed as
$$
 \tilde g(x + \rmi y) = \sum_{k=0}^m \frac{1}{k!} g^{(k)}(x) (\rmi y)^k \chi(y),
$$
where $\chi \in C^\infty_\comp(\R)$ is chosen such that it equals one near $0$.
By the Helffer-Sj\"ostrand formula we have
 $$
  f(\Delta^{1/2}) = \frac{2}{\pi} \int\limits_{\mathrm{Im}(z)>0}z \frac{\partial \tilde g}{\partial \overline z} T_z \der m(z),
 $$
 and the analogous formula holds for $\Delta_0$. Here $\der m$ denotes the Lebesgue measure on $\C$. Hence,
 \begin{gather*}
   \left( p f(\Delta^{1/2}) p - p f(\Delta_0^{1/2}) p \right) \\= \frac{2}{\pi} \int\limits_{\mathrm{Im}(z)>0}\frac{\partial \tilde g}{\partial \overline z}  \left( p (\Delta +1)^{-N} (\Delta - z^2)^{-1} p - p (\Delta_{0} +1)^{-N} (\Delta_{0} - z^2)^{-1} p \right) z \der m(z),
 \end{gather*}
 which implies the statement as the trace norm is finite and can be estimated as
 $$
  \Tr \left( p f(\Delta^{1/2}) p - p f(\Delta_0^{1/2}) p \right) \leq C \frac{2}{\pi} \int\limits_{\mathrm{Im}(z)>0} | \frac{\partial \tilde g}{\partial \overline z}|  \frac{| z |}{|\mathrm{Im}(z^2)|^2} \der m(z).
 $$
\end{proof}

\section{Incoming and outgoing sections} \label{outgoingapp}

Recall that $f \in C^\infty(M; \Lambda^p T^*M)$ is called outgoing if $f = R_\lambda h$ for some compactly supported $h \in \CnI(M; \Lambda^p T^*M)$. 
Obviously the space of outgoing functions for $\lambda \in \R \setminus \{0\}$
is a vector space. In the following we fix a real $\lambda \not=0$.

\begin{lemma}
 Suppose that $f \in \CnI(M; \Lambda^p T^*M)$ and $g = (\Delta_p  - \lambda^2) f$, then $f = R_\lambda g$.
\end{lemma}
\begin{proof}
 Let $g_\mu = (\Delta_p - \mu^2 ) f$, then $g_\mu$ is a holomorphic family of compactly supported functions
such that $g_\lambda=g$. Since $f$ is in the domain of the operator, we have $f = R_\mu g_\mu$ for all $\mu$ in the upper half plane.
 Now simply take the limit $\mu \to \lambda$.  
\end{proof}

This implies immediately the following corollary.

\begin{corollary}
 Any $f \in \CnI(M; \Lambda^p T^*M)$ is outgoing  for $\lambda$.
\end{corollary}

\begin{corollary}
 Let $\chi \in C^\infty(M)$ be supported in $M \setminus K$ such that $1-\chi \in \CnI(M)$. Then
 $f \in \CnI(M)$ is outgoing if and only if $\chi f$ is outgoing.
\end{corollary}
\begin{proof}
 This follows immediately from the fact that $(1-\chi) f$ is outgoing.
\end{proof}

Note that if $\chi$ is supported in $M \setminus K$ then we can understand it as a function on $\R^d$
as $M \setminus K$ is identified with $\R^d \setminus \tilde K$.

\begin{proposition} \label{outgoingnice}
 Let $\chi \in C^\infty(M)$ be supported in $M \setminus K$ such that $1-\chi \in \CnI(M)$, then 
 $f \in \CnI(M)$ is outgoing if and only if $\chi f$ is outgoing for the Laplace operator on $\R^d$.
\end{proposition}
\begin{proof}
 Let $R_{0,\lambda}$ be the free resolvent of $\Delta_p$ on $\R^d$. Then both $R_{0,\mu}$ and $R_\mu$
 are holomorphic in $\mu$ near $\R \setminus \{0\}$ and map to $L^2$ in the lower half plane.
 Now $f$ is outgoing if and only if $\chi f$ is outgoing, by the above corollary. Let us therefore assume w.l.o.g. that
 $f$ is supported in $M \setminus K$. Since $f = R_\lambda h$ for some compactly supported $h$ we have
 $(\Delta_p -\lambda^2) f = h$. Therefore, $h$
 is compactly supported in $M \setminus K$. 
 Now define $f_\mu:= \chi R_\mu h$. The section $f_\mu$ is in $L^2(M; \Lambda^p T^*M)$ for $\Im(\mu)>0$.
 We can now think of $f_\mu$ as a $p$-form on $\R^d$, and then we have
 $$
  (\Delta_p -\mu^2) f_\mu = [\Delta_p, \chi] f_\mu + h \in \CnI(\R^d,\Lambda^p \R^d) .
 $$
 For $\mu$ in the upper half-plane we therefore have
 $$
  f_\mu = R_{0,\mu} \left( [\Delta_p, \chi] f_\mu + h) \right).
 $$
 Now taking the limit $\mu \to \lambda$ we obtain
 $$
  f_\lambda = R_{0,\lambda} \left( [\Delta_p, \chi] f_\lambda + h) \right), 
 $$
 so $f_\lambda=\chi f$ is outgoing for the free Laplacian on $\R^d$. 
 This proves that if $f$ is outgoing, then so is $\chi f$ for the free Laplacian. To show the converse, exactly the same argument with $M$ and $\R^d$ interchanged applies, and one can conclude that $f$ is outgoing whenever $\chi f$ is for the free Laplacian.
 \end{proof}

\begin{lemma} \label{repoutgoing}
 Suppose that $f \in C^\infty(M; \Lambda^p T^*M)$ is outgoing  for $\lambda$. Then there exists a function 
 $\Phi \in C^\infty(\sphere; \Lambda^p\mathbb{R}^d)$ such that for $r$ sufficiently large
 $$
  f(r \theta) = \tilde h^{(1)}_\lambda(\Phi)(r \theta) = \lambda^{\frac{d-1}{2}} \sum\limits_{\nu} a_{\nu}(\Phi) \phi_{\nu}(\theta)h^{(1)}_{\ell_{\nu}}(\lambda r)(-\rmi)^{\ell_{\nu}},
 $$
 where the notation is as in Definition \ref{hankelsums}.
\end{lemma}
\begin{proof}
 By Prop. \ref{outgoingnice} it is sufficient to prove this for $\R^d$. In this case the bundle $\Lambda^p \R^d$ is trivial and the operators do not mix components. It is therefore sufficient to consider the case $p=0$. We have $u =  R_{0,\lambda} f$, where $f$ is compactly supported.
 It is easy to see from the formula of the outgoing resolvent kernel
 $$
  R_{0,\lambda}(x,y) = \frac{\rmi}{4}\left (\frac{\lambda}{2 \pi |x-y|} \right)^{\frac{d-2}{2}}H^{(1)}_{\frac{d-2}{2}}(\lambda |x-y|)
 $$
 that 
 \begin{gather} \label{expans}
 u(r \theta) = \frac{\ee^{\rmi r \lambda} \ee^{-\frac{\rmi \pi(d-1)}{4}} }{r^{\frac{d-1}{2}}} \Phi(\theta) +  O\left(\frac{1}{r^{\frac{d+1}{2}}}\right).
 \end{gather} 
 for a smooth function $\Phi \in C^\infty(\sphere)$. In fact $\Phi$ is
 analytic, but we will not need this at this stage. Let $(\phi_\nu)$ be an orthonormal basis in $L^2(\sphere)$ consisting of spherical harmonics of degree $\ell_\nu$. The function $\Phi$ can be expanded as $$\Phi = \sum\limits_\nu a_\nu \phi_\nu,$$ with convergence in $C^\infty(\sphere)$.
The functions 
$$
 u_\nu(r) = \int_\sphere u(r \theta) \phi_\nu(\theta) \der \theta
$$ 
are defined for $r>0$, and solve the spherical Bessel equation of order $\ell_\nu$. Moreover,
$$
 u(r \theta) = \sum\limits_\nu u_\nu(r) \phi_\nu(\theta)
$$
converges in $C^\infty(\R^d \setminus \{ 0 \})$.

By \eqref{expans} they have an asymptotic expansion
\begin{gather*}
 u_\nu(r) = \frac{\ee^{\rmi r \lambda} \ee^{-\frac{i\pi(d-1)}{4}} }{r^{\frac{d-1}{2}}} \langle \Phi, \phi_\nu \rangle_{L^2(\sphere)} + O\left(\frac{1}{{r^{\frac{d+1}{2}}}} \right) \\ = \frac{\ee^{\rmi r \lambda} \ee^{-\frac{\rmi \pi(d-1)}{4}} }{r^{\frac{d-1}{2}}} a_\nu + O\left(\frac{1}{{r^{\frac{d+1}{2}}}} \right)
\end{gather*}
for sufficiently large $r$. Comparing the expansions using \eqref{sphankelexp} one obtains
$$
 u_\nu(r) = \lambda^{\frac{d-1}{2}} (-\rmi)^{\ell_\nu}a_\nu h^{(1)}_{d,\ell_\nu}(\lambda r),
$$
and therefore, $u = \tilde h^{(1)}_\lambda(\Phi)(r \theta)$.
\end{proof}

\section{Multipole expansions}\label{Amulti}

Let $u \in C^\infty(\R^d \setminus B_R)$ and suppose that $\Delta u =0$, where $\Delta$ is the Laplace operator on functions.
We denote by $(\phi_\nu)_\nu$ be an orthonormal basis in $L^2(\sphere)$ consisting of spherical harmonics of degree $\ell_\nu$.
Then one can use separation of variables to expand $u$ into spherical harmonics.
\begin{lemma} \label{multilemma}
 In case $d>2$, we have that
 \begin{gather} \label{multipole}
 u(r \theta) = \sum_\nu \left( a_\nu \frac{1}{r^{d-2+\ell_\nu}} \phi_\nu(\theta) + b_\nu r^{\ell_\nu} \phi_\nu(\theta) \right),
\end{gather}
which converges in $C^\infty(\R^d \setminus B_R)$. If $b_\nu=0$ the above series converges uniformly on $\R^d \setminus B_R$ together with its derivatives.
\end{lemma}
\begin{proof}
 The functions 
$$
 u_\nu(r) = \int_\sphere u(r \theta) \phi_\nu(\theta) \der \theta
$$ 
are defined for $r>R$. Moreover,
$$
 u(r \theta) = \sum\limits_\nu u_\nu(r) \phi_\nu(\theta)
$$
converges in $C^\infty(\R^d \setminus B_R)$.
Then, the $u_\nu(r)$ satisfy an ordinary differential equation that has $\frac{1}{r^{d-2+\ell_\nu}}$ and $r^{\ell_\nu}$ as a system of fundamental solutions. From this one obtains the claimed expansion.
\end{proof}

The same result holds in case $d=2$ with the only modification that when $\ell=0$ the two fundamental solutions of the resulting ODE are $1$ and $\log r$. Therefore one has
 \begin{gather} \label{multipole2d}
 u(r \theta) = a_0 \log(r) +b_0 + \sum_{\nu,\ell_\nu>0} \left( a_\nu \frac{1}{r^{d-2+\ell_\nu}} \phi_\nu(\theta) + b_\nu r^{\ell_\nu} \phi_\nu(\theta) \right).
\end{gather}

\section{Spherical Bessel functions}\label{Abf}

The spherical Bessel functions $j_\ell$ are usually defined as $j_\ell(x)=\sqrt{\frac{\pi}{2x}} J_{\ell + \frac{1}{2}}(x)$. These functions appear when separating variables for the three dimensional Helmholtz equation. Here we will need the higher dimensional analog, which we define as
$$
 j_{d,\ell}(x)=\sqrt{\frac{\pi}{2}}x^{\frac{2-d}{2}}J_{\ell + \frac{d-2}{2}}(x),
$$
and refer to as the $d$-dimensional spherical Bessel function. Similarly, the corresponding $d$-dimensional spherical Hankel functions are defined as
\begin{gather*}
  h^{(1)}_{d,\ell}(x)=\sqrt{\frac{\pi}{2}}x^{\frac{2-d}{2}}H^{(1)}_{\ell + \frac{d-2}{2}}(x),\\
  h^{(2)}_{d,\ell}(x)=\sqrt{\frac{\pi}{2}}x^{\frac{2-d}{2}} H^{(2)}_{\ell + \frac{d-2}{2}}(x).
\end{gather*}
The properties of the Hankel functions (\cite[10.11.1-10.11.9]{olver2010nist}) then imply that
 \begin{gather} \label{rotationhankel1}
  h^{(1)}_{d,\ell}(x e^{\rmi \pi}) = - (-1)^{\ell+d} h^{(2)}_{d,\ell}(x),\\   \label{rotationhankel2}
  h^{(2)}_{d,\ell}(x e^{\rmi \pi})=  (-1)^{\ell} \left( h^{(1)}_{d,\ell}(x) + (1+(-1)^d) h^{(2)}_{d,\ell}(x) \right).
\end{gather}

For real $\lambda \not=0$ the asymptotic behaviour as $x \to \infty$ of these functions is as follows
 \begin{gather} \label {sphankelexp}
  h^{(1)}_{d,\ell}(x) = \frac{1}{x^{\frac{d-1}{2}}} \ee^{\rmi (x -\frac{\pi}{2} \ell -\frac{\pi}{4}(d-1))} + O\left(\frac{1}{x^{\frac{d+1}{2}}}\right)\\
  h^{(2)}_{d,\ell}(x) = \frac{1}{x^{\frac{d-1}{2}}} \ee^{-\rmi (x -\frac{\pi}{2} \ell -\frac{\pi}{4}(d-1))} + O\left(\frac{1}{x^{\frac{d+1}{2}}}\right).
\end{gather}

The asymptotic behavior as $x \to 0_+$ is
 \begin{gather} 
  h^{(1)}_{d,\ell}(x) = - \rmi \frac{1}{\sqrt{\pi}} 2^{\ell+\frac{d-3}{2}} \Gamma(\ell +\frac{d-2}{2}) x^{-\ell - d + 2} + O(x^{-\ell - d + 4}),\\ \label {sphankelexpzero}
  h^{(2)}_{d,\ell}(x) =   \rmi \frac{1}{\sqrt{\pi}} 2^{\ell+\frac{d-3}{2}} \Gamma(\ell +\frac{d-2}{2}) x^{-\ell - d + 2}+ O(x^{-\ell - d + 4}),
\end{gather}
if $\ell + \frac{d -2}{2} >0$ and 
 \begin{gather} \label {sphankelexpzerod2}
  h^{(1)}_{d,\ell}(x) =  -\rmi \sqrt{\frac{2}{\pi}} (-\log x) + O(1),\\
  h^{(2)}_{d,\ell}(x) =  \rmi \sqrt{\frac{2}{\pi}} (-\log x) + O(1),
\end{gather}
if $d=2$ and $\ell=0$.
In case $\ell + \frac{d -2}{2} >0$ one also has, uniformly in $x$ on compact sets,
 \begin{gather} 
   x^{\ell + d - 2} h^{(1)}_{d,\ell}(x) \left( - \rmi \frac{1}{\sqrt{\pi}} 2^{\ell+\frac{d-3}{2}} \Gamma(\ell +\frac{d-2}{2})\right)^{-1} \to 1,\\ \label {sphankelexpnu}
  x^{\ell + d - 2} h^{(2)}_{d,\ell}(x) \left( \rmi \frac{1}{\sqrt{\pi}} 2^{\ell+\frac{d-3}{2}} \Gamma(\ell +\frac{d-2}{2}) \right)^{-1} \to 1, 
\end{gather}
as $\ell \to \infty$. This can be inferred from the series expansions for the Hankel functions \cite[10.8.1, 10.53.2]{olver2010nist}.

In this section we give a simple proof of the formula
$$
\frac{1}{(2\pi)^{\frac{d-1}{2}}}\int\limits_{\mathbb{S}^{d-1}}\exp(-i\lambda x\cdot\omega)g(\omega)\der \omega=2\sum\limits_{\nu} a_{\nu}\phi_{\nu}\left(\frac{x}{ r}\right)  j_{d,l_{\nu}}(\lambda r) (-\rmi)^{l_{\nu}},
$$
where $a_\nu = \langle \phi_\nu, g \rangle_{L^2(\sphere)}$.
This formula follows for $d=3$ from integral formulae for spherical Bessel functions and Legendre polynomials. In higher dimensions it is equivalent to the following identity for Gegenbauer polynomials $C^{\frac{d-2}{2}}_k$ 
$$
\frac{1}{(2\pi)^{\frac{d-1}{2}}}\int\limits_{\mathbb{S}^{d-1}}\exp(-i\lambda x\cdot\omega) C^{\frac{d-2}{2}}_k(\theta \cdot\omega)\der \omega= 2j_{d,l_{\nu}}(\lambda r) (-\rmi)^{l_{\nu}} C^{\frac{d-2}{2}}_k(\theta \cdot\frac{x}{r}).
$$
Since we could not find this identity in the literature
we give here a very short proof which is based on Rellich's uniqueness theorem. First we note that it is sufficient to prove the identity for $g = \phi_\nu$.
Next we observe that both sides of the equation satisfy the eigenfunction equation $(\Delta-\lambda^2) f =0$. By Rellich's uniqueness theorem such solutions for real $\lambda>0$
are uniquely determined by their asymptotic expansion
$$
 f(r \theta) = \frac{\ee^{-\rmi \lambda r}}{r^{\frac{d-1}{2}}} h_-(\theta)+\frac{\ee^{\rmi \lambda r}}{r^{\frac{d-1}{2}}} h_+(\theta) + O\left(\frac{1}{r^{\frac{d+1}{2}}}\right)
$$
for sufficiently large $r$. It is therefore sufficient to show that the expansions on both sides are the same. An application of the stationary phase Lemma to the left hand side gives
\begin{align} \label{appendix_bexp}
\int\limits_{\mathbb{S}^{d-1}} \ee^{-\rmi \lambda x\cdot\omega} g(\omega)\der \omega= \left(\frac{2\pi}{\lambda r}\right)^{\frac{d-1}{2}}
\left(\ee^{-\rmi \lambda r} \ee^{\rmi \frac{(d-1)\pi}{4}} g(\theta)+\ee^{\rmi \lambda r} \ee^{-\rmi \frac{(d-1)\pi}{4}} g(-\theta)\right)+O((\lambda r)^{-(\frac{d+1}{2})})
\end{align} 
as $\lambda r \to \infty$.
 Here the remainder term depends on $\| g(\omega)\|_{H^2(\mathbb{S}^{d-1})}$ by stationary phase, provided $|\lambda r| \gg 1$. This expansion can be differentiated term by term in $x$, and expanded again using stationary phase to obtain asymptotics of the differentiated terms. 
A comparison to the asymptotics of the spherical Bessel (\cite{olver2010nist}) function shows the result.

\section*{Acknowledgements}

We would like to thank Gilles Carron, Colin Guillarmou, Werner M\"uller, and Andras Vasy for useful comments and sharing their insights. We also are grateful to the anonymous referees for their helpful suggestions.

\bibliographystyle{plain}

\end{document}